\documentclass[12pt]{amsart} 
\usepackage{amssymb,amsmath} 
\usepackage[dvipdfmx]{graphicx} 
\usepackage{color} 
\usepackage{comment} 
\usepackage{amsthm} 
\newtheorem{df}{Definition}[section]
\newtheorem{thm}[df]{Theorem} 
\newtheorem{cor}[df]{Corollary} 
\newtheorem{prop}[df]{Proposition} 
\newtheorem{lemma}[df]{Lemma} 
\newtheorem{rem}[df]{Remark} 
\makeatletter
 
\@addtoreset{equation}{section} 
\makeatother
\renewcommand{\phi}{\varphi} 
\newcommand{\e}{\varepsilon}

\newcommand{\bmu}{\mbox{\boldmath{$u$}}} 
 
\newcommand{\bmf}{\mbox{\boldmath{$f$}}}
\newcommand{\bmz}{\mbox{\boldmath{$z$}}} 
\newcommand\re{\textcolor{red}}
\usepackage{framed}

\allowdisplaybreaks[1]

\begin{document} 

\title[Existence of traveling waves for surface diffusion]{Existence of non-convex traveling waves \\
for surface diffusion of curves \\with constant contact angles}

\author{Takashi Kagaya}
\address{Institute of Mathematics for Industry, Kyushu University}
\email{kagaya@imi.kyushu-u.ac.jp}

\author{Yoshihito Kohsaka}
\address{Graduate School of Maritime Sciences, Kobe University}
\email{kohsaka@maritime.kobe-u.ac.jp}

\thanks{The authors are partially supported by JSPS KAKENHI Grant Number (A) 18H03670.}

\maketitle

\begin{abstract}
The traveling waves for surface diffusion of plane curves are studied. 
We consider an evolving plane curve with two endpoints, which can move freely 
on the $x$-axis with generating constant contact angles. For the evolution of 
this plane curve governed by surface diffusion, we discuss the existence, the uniqueness 
and the convexity of traveling waves. The main results show that the uniqueness and 
the convexity can be lost in depending on the conditions of the contact angles, 
although the existence holds for any contact angles in the interval $(0,\pi/2)$. 
%
\end{abstract}

\section{Introduction} 
Let $\gamma(t)\subset\mathbb{R}^2$ be an evolving closed plane curve with respect to time $t$ 
governed by the geometric evolution law
\begin{equation} 
V=-\kappa_{ss}\,\,\ \mbox{on}\,\ \gamma(t),
\label{SD} 
\end{equation} 
where $V$ is the normal velocity of $\gamma(t)$, $\kappa$ is the curvature of $\gamma(t)$, and 
$s$ is the arc-length parameter along $\gamma(t)$. In our sign convention, 
the curvature $\kappa$ for a circle with an outer unit normal is negative. 
The surface diffusion equation \eqref{SD} was first derived by Mullins \cite{M} to model 
the development of surface grooves at the grain boundaries of a heated polycrystal.  

In this paper, we consider the following free boundary problem. 
Set $\Pi:=\{(x,y)\in\mathbb{R}^2\,|\,x\in\mathbb{R},\,y=0\}$ and 
let $\Gamma(t)$ be an evolving plane curve with two endpoints which are named 
$\partial\Gamma(t):=\{P_\pm(t)\}$. 
The motion of $\Gamma(t)$ is governed by 
\begin{equation} 
\left\{\begin{array}{l} 
V=-\kappa_{ss}\,\,\ \mbox{on}\,\ \Gamma(t), \\
P_\pm(t)\in\Pi, \\
\sphericalangle(\Gamma(t),\Pi)=\psi_\pm\,\,\ \mbox{at}\,\ P_\pm(t), \\
\kappa_s=0\,\,\ \mbox{at}\,\ P_\pm(t). 
\end{array}\right. 
\label{SDB} 
\end{equation} 
From the boundary condition $P_\pm(t) \in \Pi$, 
the endpoints $P_\pm(t)$ can be distinguished by the $x$-coordinate, namely, 
the endpoint with the smaller $x$-coordinate is $P_-(t)$ and the other is $P_+(t)$. 
The direction of the arc-length parameter $s$ is positive from $P_-(t)$ to $P_+(t)$ and 
$\psi_+\in(-\pi,\pi)$ (resp.\ $\psi_- \in (-\pi,\pi)$) is an interior constant angle 
which is measured clockwise (resp.\ counter-clockwise) from 
the $x$-axis at $P_+(t)$ (resp.\ $P_-(t)$). 

\begin{figure}[t]
\begin{center}
\scalebox{0.75}{\includegraphics{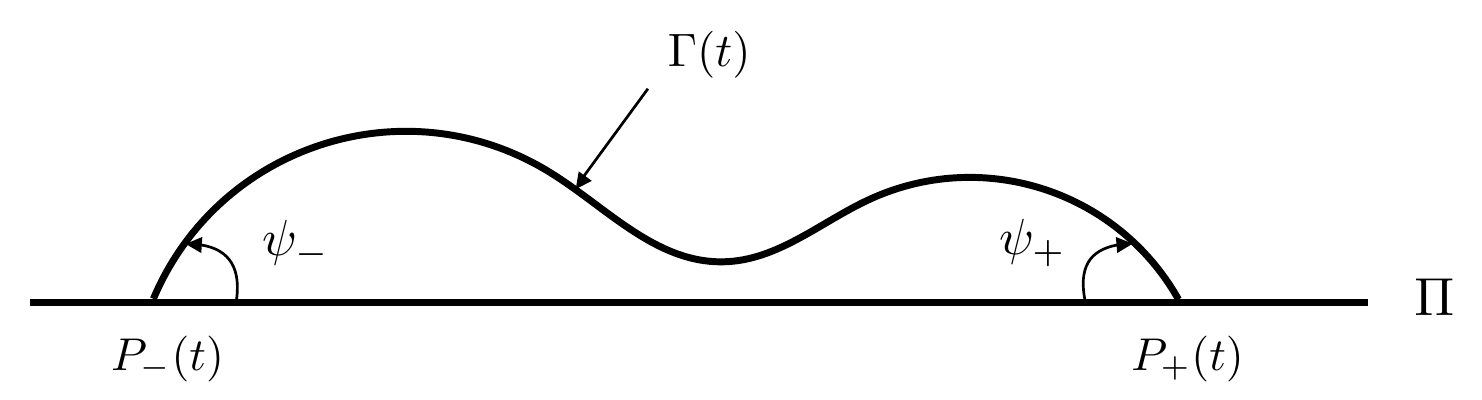}}
\end{center}
\caption{A plane curve satisfying the boundary conditions in \eqref{SDB}.}
\end{figure}

One of the our aims is to find traveling waves for \eqref{SDB} which are 
classical solutions formed by 
\begin{equation}\label{def-traveling} 
\mathcal{W}(t) = \mathcal{W}(0) + ct \vec{e}_1 \,\,\ \mbox{for}\,\  t > 0, 
\end{equation}
where $\vec{e}_1 = (1,0)$, $\mathcal{W}(0)$ is some smooth plane curve and 
$c \in \mathbb{R}$ is some constant.  
Note that $\mathcal{W}(0)$ and $c$ represent the profile and the speed of 
a traveling wave, respectively. 
Thus our aim can be restated to be to find the pairs of a profile curve $\mathcal{W}(0)$ 
and a wave speed $c \in \mathbb{R}$ such that 
$\mathcal{W}(t)$ defined as \eqref{def-traveling} is a solution of \eqref{SDB}.  
Furthermore, we prove the non-uniqueness of traveling waves by constructing those profile 
curves $\mathcal{W}(0)$ as the ``oscillating'' graphs when one of the contact angles is 
sufficiently small. 

In order to state the main results and the motivation in detail, we refer to known results 
on the convexity of closed plane curves, the evolution of which is governed by 
the area preserving curvature flow and the surface diffusion. 

For an evolving closed plane curve $\gamma(t)$, the area preserving curvature flow 
\begin{equation}\label{AP} 
V = \kappa - \dfrac{\int_{\gamma(t)} \kappa \; ds}{\int_{\gamma(t)} ds} \,\,\ \mbox{on}\,\ \gamma(t) 
\end{equation}
was introduced by Gage \cite{G} as the $L^2$-gradient flow of the length of the Jordan curves 
under the area-preserving variations.  
Thus, for the family of Jordan curves $\{\gamma(t)\}_{t \ge 0}$ govened by \eqref{AP}, 
the length of $\gamma(t)$ is non-increasing in time $t$, while the area  enclosed by 
$\gamma(t)$ is preserved. 
This variational structure indicates that if $\gamma(t)$ is a global solution in time 
of \eqref{AP}, it converges to a critical point of the length under the area constraint, 
that is, one of the circles. Indeed, he proved this fact for convex initial curves through 
the analysis of the preservation of the convexity. 
We also refer to Escher and Simonett \cite{ES}.
They proved the local stability of the circles (resp. spheres) without the assumption of the convexity of initial curves (resp. hypersurfaces).

For the surface diffusion of closed curves, Cahn and Taylor \cite{TC} showed that \eqref{SD} 
can be derived as the $H^{-1}$-gradient flow of the length of plane curves. 
Thus this geometric evolution equation has a variational structure similar 
to that of \eqref{AP}, that is, the length of the curves is non-increasing in time under 
the area constraint. 
Therefore we expect the stability of the circles for \eqref{SD}. Indeed, the local stability 
of the circles was proved by Elliott and Gracke \cite{EG}. On the other hand, Giga and Ito \cite{GI} 
proved a loss of convexity by constructing a simple and convex initial curve 
which loses its convexity during the evolution by \eqref{SD}, while the evolving curves stays simple. 

The difference associated to the preservation of the convexity between the solutions of 
\eqref{AP} and \eqref{SD} comes from the difference of the ype of the equations. 
The area preserving curvature flow \eqref{AP} is a second order parabolic equation with 
a non-local term. By a modified maximum principle, the negativity of 
the curvature of the solutions to \eqref{AP} at arbitrary time can be proved  if initial curve is 
convex. Thus we see that the convexity of the solutions is 
preserved for \eqref{AP}. 
On the other hand, the surface diffusion equation \eqref{SD} is a fourth order parabolic 
equation. Therefore, \eqref{SD} does not fulfill the maximum principle. 
This fact indicates the loss of convexity for \eqref{SD}. For more details, see \cite{GI}. 

Now, we introduce Simojo and the first author \cite{SK} that studied the area preserving 
curvature flow with boundary conditions similar to \eqref{SDB}:
\begin{equation}\label{APB}
\left\{\begin{array}{l} 
V=\kappa - \dfrac{\int_{\Gamma(t)} \kappa \; ds}{\int_{\Gamma(t)}ds}\,\,\ \mbox{on}\,\ \Gamma(t), \\
P_\pm(t) \in \Pi, \\
\sphericalangle(\Gamma(t),\Pi)=\psi_\pm\,\,\ \mbox{at}\,\ P_\pm(t).
\end{array}\right. 
\end{equation}
Here $\psi_\pm \in (0,\pi/2)$ and the initial curve $\Gamma(0)$ fulfills 
the following assumptions: 
\begin{itemize}
\item[(A1)] $\Gamma(0)$ is represented by a graph, 
\item[(A2)] $\Gamma(0)$ is convex, 
\item[(A3)] $\Gamma(0)$ satisfies the boundary conditions $P_\pm(0) \in \Pi$ and 
$\sphericalangle(\Gamma(0),\Pi)=\psi_\pm$ at $P_\pm(0)$. 
\end{itemize}
For this problem, they proved, in particular, the preservation of the convexity, 
the existence of a traveling wave and the local asymptotic stability 
at an exponential rate of it. Furthermore, they showed that 
its traveling wave is unique up to the translation and the scaling 
$(\lambda\mathcal{W}(0)+a \vec{e}_1, c/\lambda)$ for any $a \in \mathbb{R}$ and $\lambda >0$. 
In their proof of the existence of the traveling wave, they needed the convexity of the profile curve 
$\mathcal{W}(0)$ in advance besides the assumption that $\mathcal{W}(0)$ is represented by a graph. 
We remark that if we use the method in \cite{KK}, 
we can prove that an only convex traveling wave exists without assuming the convexity 
of $\mathcal{W}(0)$ in advance. Hence we obtain the uniqueness of a traveling wave 
for \eqref{AP} in the above sense 
under an only assumption that $\mathcal{W}(0)$ is represented by a graph. 

In this paper, we study the existence of traveling waves for \eqref{SDB} and compare the 
properties of the traveling waves for \eqref{SDB} with those for \eqref{APB}, in particular, 
the convexity and the uniqueness. 
Therefore, under the assumption $\psi_\pm \in (0,\pi/2)$, we find traveling waves 
for \eqref{SDB} such that the profile curve $\mathcal{W}(0)$ is represented by a graph.

\begin{thm}\label{thm:main}
\begin{list}{}{\leftmargin=0.2cm}
\item[(i)] Assume that $\psi_\pm \in (0,\pi/2)$ and $(\mathcal{W}(0),c)$ construct a traveling wave 
$\mathcal{W}(t)$ for \eqref{SDB} defined as \eqref{def-traveling} such that $\mathcal{W}(0)$ 
is represented by a graph. 
Then, the sign of the wave speed $c$ is determined by the contact angles $\psi_\pm$ as 
\begin{equation}\label{con-c-angle}
\psi_{-}
\left\{
\begin{array}{l}
> \\ = \\ <
\end{array}
\right\}
\psi_{+}
\quad
\iff
\quad
c
\left\{
\begin{array}{l}
> \\ = \\ <
\end{array}
\right\}
0. 
\end{equation}
\item[(ii)] For any contact angles $\psi_\pm \in (0,\pi/2)$, there exists at least one traveling wave 
$\mathcal{W}(t)$ for \eqref{SDB} defined as \eqref{def-traveling} such that the profile curve 
$\mathcal{W}(0)$ is represented by a graph. 
%
\item[(iii)] If $\psi_- = \psi_+ \in (0,\pi/2)$, the traveling wave is unique up to 
the translation and the scaling $(\lambda\mathcal{W}(0)+a \vec{e}_1, c/\lambda^3)$ for 
any $a \in \mathbb{R}$ and $\lambda >0$, where $\vec{e}=(1,0)$. 
Furthermore, its traveling wave is constructed by $(\mathcal{W}(0),c)$ such that 
$\mathcal{W}(0)$ is one of the arcs and $c=0$.  
\item[(iv)] Assume $\pi/2 > \psi_- > \psi_+ > 0$. 
Then, there exists a positive sequence $\{m_j\}_{j \in \mathbb{N}}$, which depends only on 
$\psi_-$, such that $\psi_- > m_1 > m_2 > \cdots$ and the following statement hold:  
if $\psi_+ \in [m_{j+1}, m_j)$,  there exist at least $2j - 1$ traveling waves 
$\mathcal{W}_k(t)\,(k = 1,2, \cdots, 2j-1)$ for \eqref{SDB} defined as \eqref{def-traveling}  
such that each profile curve $\mathcal{W}_k(0)$ is represented by a graph and its length 
is equal to $1$.  
Furthermore, for the representation 
\[ 
\mathcal{W}_k(0) = \{ (x,w_k(x))\,|\, l_k^- < x < l_k^+ \}, 
\]
the following statements on $(\mathcal{W}_k(0), c_k)$ hold: 
\begin{list}{}{}
\item[(a)] 
$0 < c_1 < c_2 < \cdots < c_{2j-1}$. 
\item[(b)] 
$w_k(x)$ is positive for $x > l_k^-$ sufficiently close to 
$l_k^-$, $(w_k)_x(l_j^-) = \tan\,\psi_-$ and $(w_k)_{xx}(l_k^-)$ is negative 
for $k = 1,2, \cdots, 2j-1$. 
\item[(c)] 
The sign of $(w_k)_x, (w_k)_{xx}$ and $w_k$ changes alternately 
in the interval $(l_k^-, l_k^+)$ from the left end point 
and the numbers of the sign changes of these functions are 
$k$ (resp.\ $k+1$), $k$ (resp.\ $k$) and $k-1$ (resp.\ $k$), respectively, if $k$ is odd (resp.\ even). 
\end{list}
For the case $\pi/2 > \psi_+ > \psi_- > 0$, the pairs of 
\[
\widetilde{\mathcal{W}}_k(0) = \{ (-x,y)\,|\,(x,y) \in \mathcal{W}_k(0) \}, \quad \tilde{c}_k = -c_k 
\]
construct traveling waves for \eqref{SDB} as in \eqref{def-traveling}, where 
$(\mathcal{W}_k(0), c_k)$ are obtained as above by exchanging the role of $\psi_-$ and $\psi_+$.  
\end{list}
\end{thm}

\begin{figure}[t]
\begin{center}
\scalebox{0.75}{\includegraphics{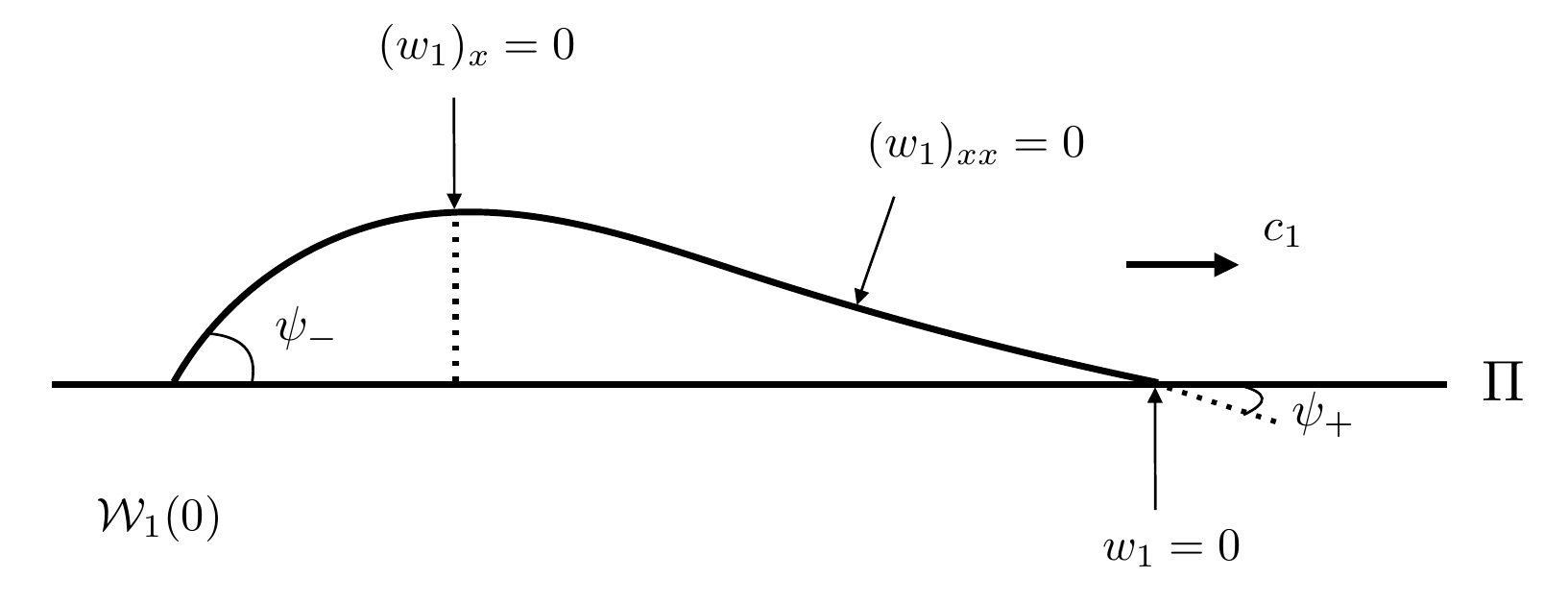}} \\
\scalebox{0.75}{\includegraphics{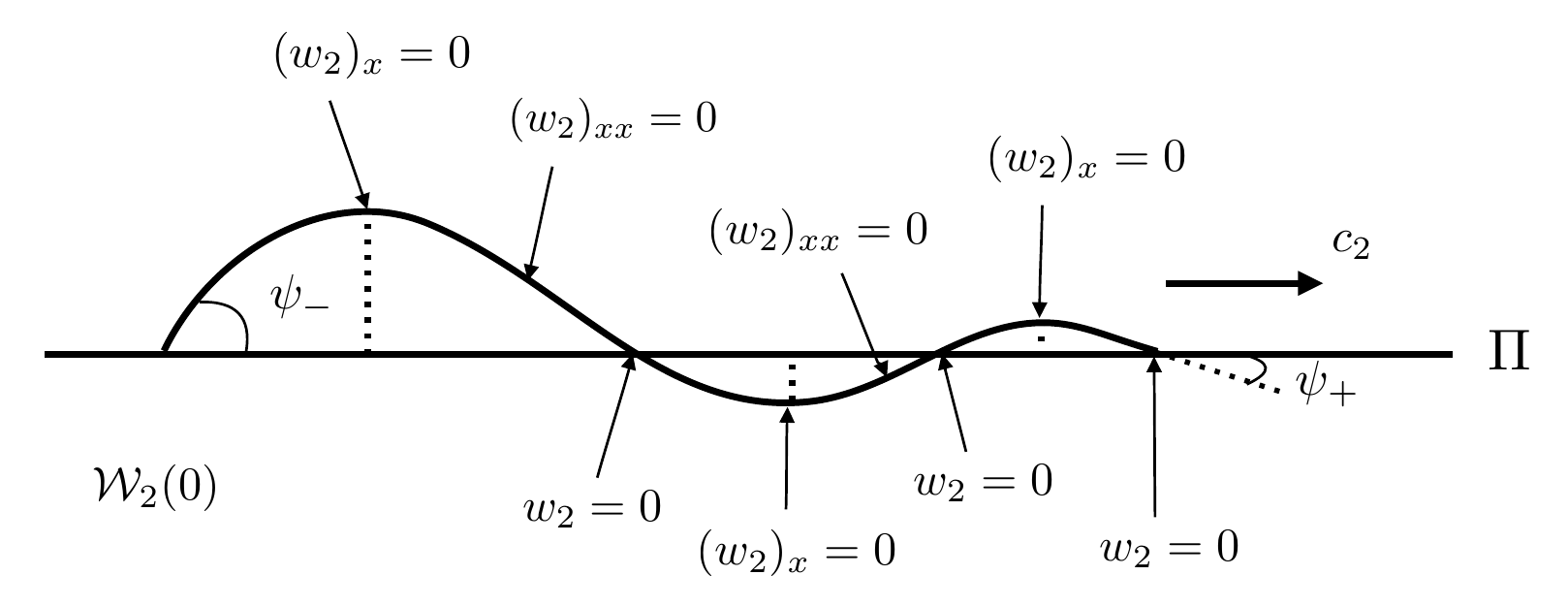}} \\[0.15cm]
\scalebox{0.75}{\includegraphics{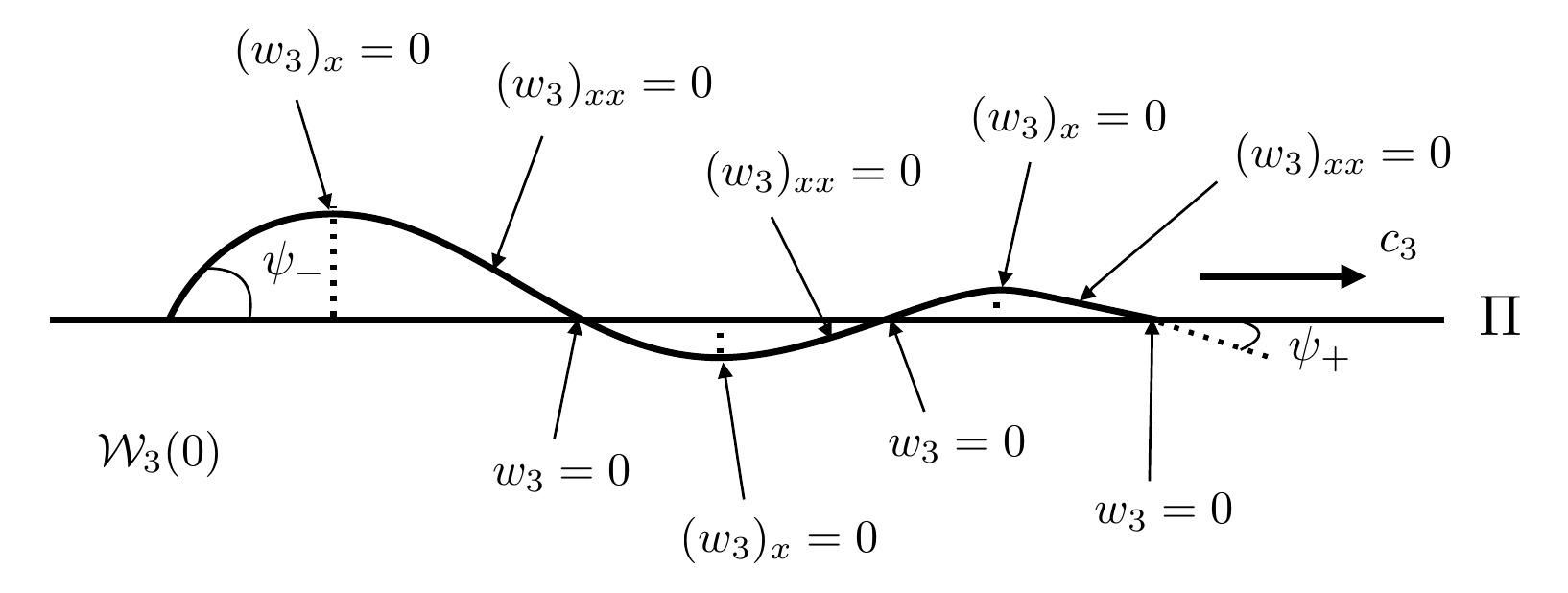}}
\end{center}
\caption{Examples of the profile curves in the case $\psi_+\in[m_3,m_2)$, that is, the case $j=2$. 
At least three kinds of profile curves can be constructed.}
\end{figure}

Since the sign changes of $(w_k)_{xx}$ occur for any $k$, 
the result (iv) shows the existence of non-convex traveling waves if $\psi_-$ or $\psi_+$ 
is sufficiently small. 
The result (iv) also contains the non-uniqueness of traveling waves up to the translation and 
the scaling $(\mathcal{W}(0), c) \mapsto (\lambda \mathcal{W}(0) + a \vec{e}_1, c/\lambda^3)$ 
for $a \in \mathbb{R}$ and $\lambda > 0$. 
The traveling waves with arbitrary length can be constructed by the above scaling. 
The results (i) and (iii) also hold for \eqref{APB} (see \cite{SK}). 
Therefore, the differences associated to the convexity and the uniqueness of traveling waves 
for \eqref{SDB} and \eqref{APB} appear only if $\psi_- \neq \psi_+$. 
We also remark that we can construct traveling waves even if $\psi_+ < 0$ as in (iv)
(see Remark \ref{rem:last}). 
Note that if $\psi_+ < 0$, there is no traveling wave for \eqref{APB} 
due to the convexity of the traveling waves. 

We refer to the results related to the evolution of plane curves by 
surface diffusion. 
Giga and Ito \cite{GI2} proved a loss of embeddedness by constructing 
a dumbbell-like initial curve which ceases to be embedded during the evolution by surface diffusion. 
Elliott and Maier-Paape \cite{EM} showed a graph-breaking by constructing 
a smooth initial curve which is represented as a graph, but loses the graph property during 
the evolution by surface diffusion. 
Note that a loss of embeddedness and a graph-breaking can occur in 
our problem \eqref{SDB}, in particular, if the contact angles are not restricted in 
the interval $(-\pi/2,\pi/2)$.  
There are several results on the stability of steady states for the evolution by 
surface diffusion in a bounded domain,  see \cite{GIK1} for one curve and \cite{EGI,GIK2,IK1,IK2} 
for three curves with a triple junction.  
In particular, Garcke, Ito and the second author \cite{GIK1,GIK2} studied the stability 
of steady states by investigating the sign of the eigenvalues corresponding to the 
linearized problem around steady states. 
This will be one of the effective arguments in the analysis of the stability 
of the traveling waves for our problem since Shimojo and the first author \cite{SK} also 
studied the spectral theory for the linearized problem to prove the stability of 
the traveling wave for \eqref{APB}. 
Asai and Giga \cite{AG} studied the existence of a self-similar solution and its stability for 
the evolution of the unbounded curves in the half space $\{(x,y)\in\mathbb{R}^2\,|\,x>0\}$, 
which intersect the $y$-axis with the constant angle. 
Kanel, Novick-Cohen and Vilenkin \cite{KNV} studied the existence and the uniqueness of 
a traveling wave for the the evolution of three curves with a triple junction in the whole space 
$\mathbb{R}^2$ by surface diffusion coupled with the curvature flow. 
Our approach based on the Gauss map is motivated by this paper. 
For the evolution of curves by other geometric flows with a similar boundary condition to ours, 
see \cite{CGK,CG,GH,GMSW}.

The present paper is organized as follows. 
In Section \ref{sec:profile}, we introduce the Gauss map and derive a profile equation of 
traveling waves for \eqref{SDB}, which is represented as a boundary value problem 
for a 3rd order ODE (denoted by BVP). 
Section \ref{sec:initial}\,--\,\ref{sec:main} are devoted to find solutions of BVP by applying 
the shooting method. In Section \ref{sec:initial}, we discuss the existence and the uniqueness 
of the solution to the initial value problem (denoted by IVP), which is derived from BVP and 
contains two parameters $c>0$ and $\alpha \in \mathbb{R}$. 
Here $c$ and $\alpha$ are the wave speed and the curvature of 
the profile curve at the left endpoint, respectively. 
The continuity of the solution of IVP with respect to $c$ and $\alpha$ is also proved. 
In Section \ref{sec:shoo-alpha}, the shooting with respect to $\alpha$ is studied. 
Section \ref{sec:zero-point}\,--\,\ref{sec:order-est} are dedicated to deriving the relation 
between ``oscillation'' of the solution of IVP and the parameter $c>0$. 
For the solution and its derivatives, the analysis of the zero points of them is carried out in 
Section \ref{sec:zero-point} and the estimates of them with the order of convergence on $c>0$ 
at their zero points are obtained in Section \ref{sec:order-est}. Finally, in Section \ref{sec:main}, 
we have the shooting argument with respect to $c$ in order to obtain the solution of BVP. 

%
%

\section{Profile equation of traveling wave}\label{sec:profile} 

In this section, we introduce a profile equation of traveling waves for \eqref{SDB}. 
We first observe that $(\lambda\mathcal{W}(0)+a \vec{e}_1, c/\lambda^3)$ constructs 
a traveling wave as in \eqref{def-traveling} for any $a \in \mathbb{R}$ and $\lambda >0$ 
if $(\mathcal{W}(0), c)$ constructs a traveling wave. 
Note that our aim is to find traveling waves such that those profile curves are represented by graphs. 
Therefore, it is sufficient to find a traveling wave $\mathcal{W}(t)$ 
constructed by a pair $(\mathcal{W}(0),c)$ such that 
\begin{itemize}
\item[(W1)] 
the length of $\mathcal{W}(0)$ is equal to $1$, 
\item[(W2)] 
the left endpoint of $\mathcal{W}(0)$ is $(0,0)$ in the $(x,y)$-coordinate plane, 
\item[(W3)] 
$\mathcal{W}(0)$ is represented by a graph. 
\end{itemize}
Then, $\mathcal{W}(0)$ can be parametrized as 
\[
\mathcal{W}(0) = \{(x(s), y(s))\,|\, 0 \le s \le 1 \}, 
\]
where $s$ is the arc-length parameter of $\mathcal{W}(0)$. 
Here, we introduce the smooth function $\Theta : [0,1] \to \mathbb{R}$ associated to 
the Gauss map defined as 
\begin{equation}\label{def-theta} 
x'(s) = \cos \Theta(s), \quad y'(s) = \sin \Theta(s) \,\,\ \mbox{for}\,\  0 \le s \le 1, 
\end{equation}
where $'$ means the differential with respect to $s$. 
We note that the range of $\Theta$ can be restricted as $|\Theta(s)| < \pi/2$ for 
$s \in [0,1]$ due to the assumption (W3). 
By using this function $\Theta$, $\mathcal{W}(t)$ can be parametrized as 
\begin{equation}\label{para-tra} 
\mathcal{W}(t) = \left\{ \left(ct + \int_0^s \cos\Theta(\tilde{s}) \; d\tilde{s}, 
\int_0^s \sin\Theta(\tilde{s}) \; d\tilde{s}\right) \,\bigg|\, 0 \le s \le 1\right\} 
\end{equation}
since $\mathcal{W}(t)$ is defined as \eqref{def-traveling}. 
Then the normal velocity $V$ and the curvature $\kappa$ of $\mathcal{W}(t)$ are 
represented as
\[ 
V(s,t) = V(s) = c \sin \Theta(s), \quad \kappa(s,t) = \kappa(s) = \Theta'(s). 
\]
If $\mathcal{W}(t)$ is a solution of \eqref{SDB}, $\Theta$ satisfies  
\begin{equation} 
\left\{\begin{array}{l} 
\Theta^{(3)}=c\sin\Theta,\,\,\ s \in(0,1) \\[0.1cm]
\Theta(0)=\psi_-,\,\,\ \Theta(1)=-\psi_+, \\[0.1cm]
\Theta''(0)=\Theta''(1)=0. 
\end{array}\right. 
\label{SD_TW_3rdODE} 
\end{equation}
In oder to obtain the boundary condition $\Theta(0)=\psi_-$ and 
$\Theta(1)=-\psi_+$, we use the restriction $|\Theta(s)| < \pi/2$ for $s \in [0,1]$. 
Thus, our problems are to find pairs of a function $\Theta: [0,1] \mapsto (-\pi/2, \pi/2)$ and 
a constant $c$ satisfying \eqref{SD_TW_3rdODE} and analyze their properties 
on the configuration. 

\begin{rem}\label{rem:pro}
We can construct the traveling wave for \eqref{SDB} from the solution $(\Theta, c)$ of  \eqref{SD_TW_3rdODE} by parametrizing $\mathcal{W}(t)$ as in \eqref{para-tra}. 
It is easy to see by the above argument that $\mathcal{W}(t)$ satisfies \eqref{SDB} 
except the condition $P_\pm(t) \in \Pi$. We confirm that the condition $P_\pm(t) \in \Pi$ 
is fulfilled in the cases $c \neq 0$ and $c = 0$, respectively. 

In the case $c \neq 0$, \eqref{SD_TW_3rdODE} implies
\[ 
\int_0^1 \sin \Theta(s) \; ds = \dfrac{1}{c}\int_0^1 \Theta^{(3)}(s) \; ds 
= \dfrac{1}{c}(\Theta''(1)-\Theta''(0)) = 0, 
\]
so that we see that the condition $P_\pm(t) \in \Pi$ holds. 

In the case $c = 0$, it follows from \eqref{SD_TW_3rdODE} that 
\[ 
\Theta(s) = -(\psi_- + \psi_+) s + \psi_- \,\,\ \mbox{for}\,\ s \in [0,1]. 
\]
Hence we see that the condition $P_\pm(t) \in \Pi$ holds only if  
$\psi_+ = \psi_-$. 

We also note that the restriction $|\Theta| < \pi/2$ for $s \in [0,1]$ is equivalent to 
the assumption (W3). It is known that a loss of embeddedness \cite{GI2} and 
a graph-breaking \cite{EM} can occur in the evolution by surface diffusion. Thus,   
if we study traveling waves for \eqref{SDB} without the assumption (W3), 
there is a possibility that we will find traveling waves which are not represented by 
a graph or not even simple curves. 
\end{rem}

Let us prove the first property of Theorem \ref{thm:main}. 

\begin{proof}[Proof of Theorem \ref{thm:main}(i)] 
It is sufficient to prove \eqref{con-c-angle} for any pair $(\Theta,c)$ satisfying 
\eqref{SD_TW_3rdODE}. First, we consider the case $c \neq 0$. 
Since $\Theta$ satisfies $\Theta^{(3)}=c\sin\Theta$ and $\Theta(0) = \psi_- > 0$, 
there exists a constant $\hat{s} \in (0,1]$ such that 
\[ 
\Theta(s) > 0 \quad {\rm and} \quad \Theta^{(3)}(s) \neq 0\,\,\ \mbox{for}\,\ s \in [0,\hat{s}]. 
\]
Therefore, by the condition $\Theta''(0) = 0$, we have 
\begin{equation}\label{main-1-1}
\Theta''(s) \neq 0\,\,\ \mbox{for}\,\ s \in (0,\hat{s}]. 
\end{equation}
Multiplying the equation $\Theta^{(3)}=c\sin\Theta$ by $\Theta'$ and integrating it 
on the interval $[0,1]$ with the help of the boundary conditions, 
we obtain 
\begin{equation}\label{main-1-2} 
- \int_0^1 (\Theta''(s))^2 \; ds
= c(\cos \psi_- - \cos \psi_+). 
\end{equation}
Since \eqref{main-1-1} guarantees that the left hand side of \eqref{main-1-2} is 
strictly negative, we see that 
\[ 
c(\cos \psi_- - \cos \psi_+) < 0.
\]
Hence the conditions $\psi_->\psi_+$ and $\psi_-<\psi_+$ are equivalent to $c > 0$ 
and $c < 0$, respectively. 
These equivalences also imply that the condition $\psi_- = \psi_+$ is equivalent to 
$c=0$. 
\end{proof}

From Theorem \ref{thm:main}(i) and the argument for $c=0$ in Remark \ref{rem:pro}, 
it is easy to conclude that Theorem \ref{thm:main}(iii) holds. 
Therefore, in the following sections, we may prove Theorem \ref{thm:main}(ii) and (iv) 
only for the case $c>0$ and $\psi_- > \psi_+$. We remark that, as we mentioned in 
Theorem \ref{thm:main}(iv), the traveling waves for the case $c<0$ and $\psi_- < \psi_+$ 
can be constructed by reflecting those for the above case with respect to the $y$-axis. 
Furthermore, according to Remark \ref{rem:pro}, the solutions of \eqref{SD_TW_3rdODE} 
construct the traveling waves $\mathcal{W}(t)$ for \eqref{SDB} by using \eqref{para-tra}. 
This implies that for the representation $\mathcal{W}(0)=\{(x,w(x))\,|\,l^-<x<l^+\}$ 
\begin{align*}
&w_x(x)=\frac{y'(s)}{x'(s)}=\frac{\sin\Theta(s)}{\cos\Theta(s)}, \\
&w_{xx}(x)=\frac{-x''(s)y'(s)+x'(s)y''(s)}{(x'(s))^3}=\frac{\Theta'(s)}{\cos^3\Theta(s)}, \\
&w(x)=y(s)=\int_0^s\sin\Theta(\tilde{s})\,d\tilde{s}
=\frac1{c}\int_0^s\Theta^{(3)}(\tilde{s})\,d\tilde{s}=\frac{\Theta''(s)}{c}. 
\end{align*}
Since $\cos\Theta(s)>0$ by the restriction $\Theta(s)\in(-\pi/2,\pi/2)$ for $s\in[0,1]$, 
the sign of $w_x$, $w_{xx}$ and $w$ are the same as that of 
$\Theta$, $\Theta'$ and $\Theta''$, respectively. 
Thus it is sufficient to analyze the sign changes of $\Theta$ and 
its derivatives in order to obtain the ``oscillation'' in Theorem \ref{thm:main}(iv). 

\section{Initial value problem} \label{sec:initial}
Let us consider the existence of a solution $(\Theta, c)$ to the boundary value problem 
\eqref{SD_TW_3rdODE} in the case $c>0$ and $\psi_- > \psi_+$. 
Our strategy is based on the shooting method which consists of 
the following three steps.
First, we study the initial value problem 
\begin{equation} 
\left\{\begin{array}{l} 
\Psi^{(3)}=c\sin\Psi, \\[0.15cm]
\Psi(0)=\psi_-,\,\ \Psi'(0)=\alpha,\,\ \Psi''(0)=0
\end{array}\right. 
\label{IVP} 
\end{equation} 
for constants $c\in\mathbb{R}$ and $\alpha\in\mathbb{R}$. 
Then we will find a unique solution $\Psi(\,\cdot\,; \alpha, c) \in C^\infty([0,1])$ 
of \eqref{IVP} for each $\alpha \in \mathbb{R}$ and $c\in\mathbb{R}$. 
The second step is to find a $\hat{\alpha}(c) \in \mathbb{R}$ such that 
$\Psi''(1; \hat{\alpha}(c), c) = 0$ for each $c>0$, where $\Psi(\,\cdot\,; \alpha, c)$ 
is a solution of \eqref{IVP}. 
In this step, the monotonicity of $\Psi''(1; \alpha, c)$ with respect to $\alpha$ is one of 
the key tools to prove the existence and the uniqueness of $\hat{\alpha}(c)$. Set 
$\hat{\Psi}(\,\cdot\,;c):=\Psi(\,\cdot\,; \hat{\alpha}(c), c)$. Then $\hat{\Psi}(\,\cdot\,;c)$ fulfills 
\eqref{SD_TW_3rdODE} except the boundary condition $\hat{\Psi}'(1;c)=-\psi_+$. 
Hence the third step is to find constants $c > 0$ such that 
$\hat{\Psi}'(1;c)=-\psi_+$. As a result, 
the pair $(\hat{\Psi}(\,\cdot\,;c), c)$ is a solution of \eqref{SD_TW_3rdODE}. 
In the third step, we observe the ``oscillation'' of $\hat{\Psi}(\,\cdot\,;c)$. 
Roughly speaking, the period of oscillation of $\hat{\Psi}(\,\cdot\,;c)$ contracts 
as $c$ increase,  
while the amplitude of it around $s=1$ is smaller than that around $s=0$ 
since a kind of ``energy loss'' as in \eqref{main-1-2} occurs. 
Therefore, non-uniqueness of the traveling wave will be shown 
if $\psi_+$ is sufficiently small. 

Let us discuss the existence of a solution of \eqref{IVP} and the continuity of 
its solution with respect to the parameters $(\alpha,c)$. 
Define a function $\bmf : \mathbb{R}^4 \to \mathbb{R}^3$ by 
\[
\bmf(\bmz,c):=\begin{pmatrix}z_2 \\ z_3 \\ c\sin z_1 \end{pmatrix} 
\,\,\ \mbox{for}\,\ \bmz = \begin{pmatrix} z_1 \\ z_2 \\ z_3 \end{pmatrix} \in \mathbb{R}^3, 
\,\ c \in \mathbb{R}.  
\]
Then the initial value problem \eqref{IVP} can be rewritten as 
\begin{equation} 
\bmu'=\bmf(\bmu,c), \quad \bmu(0)=\bmu_0,
\label{IVP_vector} 
\end{equation}
where $\bmu(s) = {}^t (\Psi(s), \Psi'(s), \Psi''(s))$ and $\bmu_0 = {}^t (\psi_-, \alpha, 0)$. 
Thus we show the existence of a solution to \eqref{IVP_vector} instead of that 
to \eqref{IVP}. 
Fix $\rho>1$. Let $K>\sqrt{(2\rho+1)^2+4}\,\psi_-$ and $b>0$ be arbitrary large constants and 
define the domain $\Omega(K,b)$ as
\[
\Omega(K,b):=\bigl\{(s,\bmz,c)\,\big|\,s\in(-\rho,\rho),\ 
\|\bmz\|_{\mathbb{R}^3}<K,\ c\in(-b,b)\bigr\}, 
\]
where $\|\cdot\|_{\mathbb{R}^3}$ is the Euclidean norm. 
Clearly, $\bmf$ is continuous in $\Omega(K,b)$ for any $K$ and $b$. 
Moreover, we observe that $\bmf$ satisfies the Lipschitz condition with respect to 
$\bmz$ in $\Omega(K,b)$. 
Indeed, using the mean value theorem and 
\[
\nabla_{\mbox{\footnotesize{$\bmz$}} }\bmf(\bmz,c)
=\begin{pmatrix}\,0&0&c\cos z_1 \\ \,1&0&0& \\ \,0&1&0 \end{pmatrix},
\]
we are led to
\begin{align*}
\|\bmf(\bmz,c)-\bmf(\tilde{\bmz},c)\|_{\mathbb{R}^3} 
\le&\,\|\nabla_{\mbox{\footnotesize{$\bmz$}} }\bmf((1-\sigma)\bmz+\sigma\tilde{\bmz},c)\|
_{\mathbb{R}^{3\times3}} 
\|\bmz-\tilde{\bmz}\|_{\mathbb{R}^3}\,\,\ \\[0.05cm]
\le&\,\sqrt{2+b^2}\,\|\bmz-\tilde{\bmz}\|_{\mathbb{R}^3} 
\end{align*}
for any $\bmz, \tilde{\bmz} \in \mathbb{R}^3$ and some $\sigma \in (0,1)$, 
where $\|\cdot\|_{\mathbb{R}^{3 \times 3}}$ is also the Euclidean norm. 

Set 
\[
\widetilde{\bmu}(s):=\begin{pmatrix}-2 \psi_- s+\psi_- \\ \,\,-2\psi_- \\ \,0 \end{pmatrix}, \quad
\widetilde{\bmu}_0:=\begin{pmatrix}\,\,\psi_- \\ -2\psi_- \\ 0 \end{pmatrix}.
\]
We easily see that $\widetilde{\bmu}$ is a solution of 
\[
\bmu'=\bmf(\bmu,0), \quad \bmu(0)=\widetilde{\bmu}_0
\]
for $s \in(-\rho,\rho)$. Note that $\widetilde{\bmu}$ satisfies 
$\|\widetilde{\bmu}(s)\|_{\mathbb{R}^3}\le \sqrt{(2\rho+1)^2+4}\,\psi_-$ for any $s \in (-\rho, \rho)$ 
and $\|\bmu_0-\widetilde{\bmu}_0\|_{\mathbb{R}^3}=|\alpha-(-2\psi_-)|$. 
Then, applying \cite[Theorem 7.4]{CL}, there exists $h>0$ such that for any 
\[
(\alpha,c)\in U_h=\{(\alpha,c)\,|\,|\alpha-(-2\psi_-)|+|c|<h\} 
\]
\eqref{IVP_vector} has a unique smooth solution $\bmu(s;\alpha,c)$ for 
$s\in(-\rho,\rho)$ and $\bmu(\,\cdot\,;\alpha,c)$ is continuous with respect to 
$(\alpha,c)$ in $U_h$. 

Applying the Taylor's theorem to the components of the solution $\bmu(s;\alpha,c)$ 
obtained as above and using the equation $\Psi^{(3)}=c\sin\Psi$,  
we can represent $\Psi'', \Psi'$ and $\Psi$ as 
\begin{align} 
&\Psi''(s;\alpha,c)
=c\,\int_0^s \sin\Psi(\tilde{s};\alpha,c)\,d\tilde{s}, \label{Psi2}\\
&\Psi'(s;\alpha,c)
=\alpha+c\int_0^s (s-\tilde{s})\sin\Psi(\tilde{s};\alpha,c)\,d\tilde{s}, \label{Psi1}\\
&\Psi(s;\alpha,c)
=\psi_-+\alpha s+\frac{c}2\int_0^s (s-\tilde{s})^2\sin\Psi(\tilde{s};\alpha,c)\,d\tilde{s}. 
\label{Psi} 
\end{align} 
Then it follows that for $s \in (-\rho,\rho)$ 
\[ 
|\Psi''(s;\alpha,c)| \le |c|\rho, \quad 
|\Psi'(s;\alpha,c)| \le |\alpha| + \dfrac{|c|\rho^2}{2}, \quad 
|\Psi(s;\alpha,c)| \le \psi_- + |\alpha|\rho + \dfrac{|c| \rho^3}{6}. 
\]
Therefore, we can choose sufficiently large $\tilde{K}>0$ such that  
$\|\bmu(s;\alpha,c)\|_{\mathbb{R}^3} < \tilde{K}$ for any 
$s \in (-\rho,\rho)$, and also take sufficiently large $\tilde{b}>0$ such that $|c|<\tilde{b}$. 
A similar argument as above works well to obtain a constant $\tilde{h}>0$ such that for any 
\[
(\tilde{\alpha}, \tilde{c}) \in U_{\tilde{h}} 
= \{(\tilde{\alpha},\tilde{c}) \,|\,|\tilde{\alpha}-\alpha|+|\tilde{c} - c|<\tilde{h}\} 
\]
\eqref{IVP_vector} has a unique smooth solution $\bmu(s;\tilde{\alpha},\tilde{c})$ 
for $s\in(-\rho,\rho)$ and $\bmu(\,\cdot\,;\tilde{\alpha},\tilde{c})$ is continuous 
with respect to $(\tilde{\alpha},\tilde{c})$ in $U_{\tilde{h}}$.
Since this argument can be applied whenever $\alpha$ and $c$ are finite, we see that 
there exists a unique solution $\Psi(\,\cdot\,;\alpha,c)\in C^3(-\rho,\rho)$ of \eqref{IVP} 
for any $\alpha \in \mathbb{R}$ and $c \in \mathbb{R}$, and 
$\Psi^{(k)}(\,\cdot\,;\alpha,c)\,(k=0,1,2,3)$ is continuous with respect to $(\alpha,c)$. 
Therefore, we obtain the following lemma. 

\begin{lemma}\label{lem:conti_ineq} 
Let $\rho > 1$. 
For any $\alpha,c \in \mathbb{R}$, there exists a unique solution 
$\Psi(\,\cdot\,;\alpha,c)\in C^3(-\rho,\rho)$ of \eqref{IVP}. 
Moreover, 
\[
\lim_{(\tilde{\alpha},\tilde{c}) \to (\alpha,c)} 
\|\Psi(\,\cdot\,;\tilde{\alpha},\tilde{c})-\Psi(\,\cdot\,;\alpha,c)\|_{C^3([-\rho,\rho])} = 0. 
\]
\end{lemma} 

%
%

\section{Shooting with respect to $\alpha$}\label{sec:shoo-alpha}

As we mentioned in the beginning of Section \ref{sec:initial}, we will find 
a constant $\hat{\alpha}(c)$ such that the solution $\Psi(s;\alpha,c)$ of \eqref{IVP} 
fulfills $\Psi''(1;\hat{\alpha}(c), c) = 0$ for each $c>0$. Note that if we set 
$\hat{\Psi}(s; c) := \Psi(s; \hat{\alpha}(c), c)$, hen $\hat{\Psi}$ satisfies 
\begin{equation}\label{IVP2}
\left\{\begin{array}{l} 
\hat{\Psi}^{(3)}=c\sin\hat{\Psi},\,\,\ s \in(0,1) \\[0.1cm]
\hat{\Psi}(0)=\psi_-,\,\ \hat{\Psi}''(0)=\hat{\Psi}''(1)=0.
\end{array}\right. 
\end{equation}
In order to find this $\hat{\alpha}(c)$, we prove the following proposition in this section. 

\begin{prop}\label{thm:exit_a} 
Let $\psi_-$ be a given constant and satisfy $\psi_-\in(0,\pi/2)$. 
Assume that $\Psi(s;\alpha,c)$ is a solution of \eqref{IVP} 
for each $\alpha\in\mathbb{R}$ and $c>0$. 
\begin{list}{}{\leftmargin=0cm\itemindent=0.2cm\topsep=0cm\itemsep=0cm}
\item[(i)]
Then, for any $c>0$, there exists a unique $\hat{\alpha}(c)$ such that 
\begin{equation}\label{exit_a} 
\Psi''(1;\hat{\alpha}(c),c)=0 \,\,\ \mbox{and}\,\,\ 
-\frac{\pi}2<\Psi(s;\hat{\alpha}(c),c)<\frac{\pi}2 
\,\,\ \mbox{for}\,\ s \in [0,1]. 
\end{equation}
\end{list}
Moreover, $\hat{\alpha}(c)$ satisfies the following properties. 
\begin{list}{}{\leftmargin=0cm\itemindent=0.2cm\topsep=0cm\itemsep=0cm}
\item[(ii)] $\hat{\alpha}(c)$ is continuous with respect to $c>0$. 
\item[(iii)] $\hat{\alpha}(c)$ is negative for all $c>0$. 
\item[(iv)] As $c > 0$ tend to $0$, the limit of $\hat{\alpha}(c)$ exists and 
\[
\hat{\alpha}(c) \to -2\psi_-, \quad \Psi(s; \hat{\alpha}(c), c) \to \psi_- - 2\psi_- s \quad 
{\rm in}\,\ C^2([0,1]). \]
\end{list}
\end{prop} 

In the proof of the existence of traveling waves for \eqref{SDB} given by 
\eqref{para-tra}, one of the main purposes is to construct 
a function $\Theta$ associated to the Gauss map \eqref{def-theta} such that 
$|\Theta(s)| < \pi/2$ for $s \in [0,1]$. Thus we need to 
find $\hat{\alpha}$ satisfying the second property of \eqref{exit_a}. 
Define a subset $I(c) \subset \mathbb{R}$ for any $c>0$ as
\begin{equation}\label{def-ic} 
I(c) := \{ \alpha \in \mathbb{R}\,|\, |\Psi(s; \alpha, c)| < \pi/2\,\,\ \mbox{for}\,\ s \in [0,1] \} 
\end{equation}
and let us find $\hat{\alpha}(c)$ in the set $I(c)$. 
Note that it follows from the continuity of $\Psi(\,\cdot\,; \alpha, c)$ and 
$\Psi(0; \alpha, c) = \psi_- \in (0,\pi/2)$ that  
\begin{equation}\label{short-bdd} 
0 < \Psi(s; \alpha, c) < \pi/2 \,\,\ \mbox{for}\,\  s \in [0,\varepsilon), 
\end{equation} 
where $[0, \varepsilon)$ is a short interval depending only on $\alpha$ and $c$. 
In Section \ref{sec:ex-alpha}, we prove the existence of $\hat{\alpha}(c)$ as in 
Proposition \ref{thm:exit_a}(i). The additional properties (ii)--(iv) of $\hat{\alpha}(c)$ 
are useful for  analyzing ``oscillation'' of $\Psi(\,\cdot\,; \hat{\alpha}(c),c)$ and 
the shooting argument with respect to $c$. 
Proposition \ref{thm:exit_a}(ii)--(iv) will be proved in Section \ref{sec:add}. 

\subsection{Existence of $\hat{\alpha}$}\label{sec:ex-alpha}

In this subsection, we prove the existence of $\hat{\alpha}(c)$ satisfying \eqref{exit_a} 
for any $c>0$ through the analysis for $I(c)$ 
defined by \eqref{def-ic}. 
The following monotonicity is one of the key tools in the shooting argument. 

\begin{lemma}\label{lem:mono_a} 
Fix $c>0$. Assume that $\alpha_1<\alpha_2$ and 
$|\Psi(s; \alpha_i, c)| < \pi/2$ for $i=1,2$ and any $s \in [0,\hat{s})$ with some $\hat{s} \in (0,1]$. 
Then 
\begin{equation}\label{ep_mono_a} 
\Psi(s; \alpha_1,c) < \Psi(s; \alpha_2, c)\,\,\ \mbox{for}\,\ s \in (0, \hat{s}]. 
\end{equation} 
\end{lemma} 

\begin{proof} 
First, we note that the inequality $\Psi(s; \alpha_1, c) < \Psi(s; \alpha_2, c)$ holds in a short interval $[0, \varepsilon)$ from the continuity of $\Psi(\,\cdot\,; \alpha_i,c)$ for $i=1,2$ and 
\[ \Psi(0; \alpha_1, c) = \Psi(0; \alpha_2, c) = \psi_-, \quad \Psi'(0; \alpha_1, c) = \alpha_1 < \alpha_2 = \Psi'(0; \alpha_2, c). \]
In proof by contradiction, suppose that there exists $\hat{\varepsilon} \in (0, \hat{s}]$ such that 
\begin{align} 
&-\frac{\pi}2<\Psi(s;\alpha_1,c)<\Psi(s;\alpha_2,c)<\frac{\pi}2\,\,\ 
\mbox{for}\,\ s\in(0,\hat{\varepsilon} ) \label{mono_1}\\
&\Psi(\hat{\varepsilon};\alpha_2,c)-\Psi(\hat{\varepsilon};\alpha_1,c)=0. \label{mono_2} 
\end{align} 
By using \eqref{Psi}, we have
\begin{equation} \label{mono_3} 
\begin{aligned}
&\Psi(\hat{\varepsilon};\alpha_2,c)-\Psi(\hat{\varepsilon};\alpha_1,c) \\
&=(\alpha_2-\alpha_1)\hat{\e} 
+\frac{c}2\int_0^{\hat{\e}}(\hat{\e}-\tilde{s})^2\{
\sin\Psi(\tilde{s};\alpha_2,c)-\sin\Psi(\tilde{s};\alpha_1,c)\}\,d\tilde{s}.
\end{aligned}
\end{equation}
It follows from $\alpha_2-\alpha_1>0$, $c>0$, \eqref{mono_1} and the monotonicity of 
the sine function on $[-\pi/2, \pi/2]$ that 
$\Psi(\hat{\varepsilon}; \alpha_2, c) - \Psi(\hat{\varepsilon}; \alpha_1, c) > 0$. 
This contradicts \eqref{mono_2} and we obtain the conclusion. 
\end{proof} 

\noindent
This monotonicity immediately implies the following corollary. 

\begin{cor} 
Fix $c>0$ and assume that $\alpha_1<\alpha_2$. If $\alpha_1,\alpha_2 \in I(c)$, 
\[
\Psi(s;\alpha_1,c)<\Psi(s;\alpha_2,c)\,\,\ \mbox{for}\,\ s \in (0,1]. 
\]
\end{cor} 

Let us show some properties of $I(c)$.

\begin{lemma}\label{lem:conect} 
For each $c>0$, $I(c)$ is connected. 
\end{lemma} 

\begin{proof} 
In proof by contradiction, suppose that for $\alpha_1,\alpha_2\in I(c)$ with 
$\alpha_1 < \alpha_2$ there exists $\alpha_\ast$ such that 
\[
\alpha_\ast\in(\alpha_1,\alpha_2) \quad\mbox{and}\quad 
\alpha_\ast\not\in I(c). 
\]
Since \eqref{short-bdd} holds, 
$\alpha_\ast\not\in I(c)$ implies the existence of $\hat{s} \in (0,1)$ such that 
\begin{equation}\label{conect-1}
|\Psi(s; \alpha_\ast, c)| < \dfrac{\pi}{2} \,\,\ \mbox{for} \,\ s \in [0,\hat{s}), \quad 
|\Psi(\hat{s}; \alpha_\ast, c)| = \dfrac{\pi}{2}. 
\end{equation}
On the other hand, $\alpha_i \in I(c)$ for $i=1,2$ gives
\begin{equation}\label{conect-2}
|\Psi(s; \alpha_i, c)| < \dfrac{\pi}{2}\,\,\ \mbox{for}\,\ s \in [0,1],\,\ i=1,2. 
\end{equation}
If $\Psi(\hat{s}; \alpha_\ast, c) = \pi/2$, it follows from \eqref{conect-1}, \eqref{conect-2}, Lemma \ref{lem:mono_a} and $\alpha_\ast<\alpha_2$ that 
\[
\Psi(\hat{s};\alpha_2,c)>\Psi(\hat{s};\alpha_\ast,c)=\frac{\pi}2. 
\]
This contradicts $\alpha_2\in I(c)$. Applying a similar argument 
to the case $\Psi(\hat{s}; \alpha_\ast, c) = -\pi/2$, 
we obtain
\[
\Psi(\hat{s};\alpha_1,c)
<\Psi(\hat{s};\alpha_\ast,c)=-\frac{\pi}2. 
\]
This contradicts $\alpha_1\in I(c)$. 
%
%
\end{proof} 

\begin{lemma}\label{lem:bound_I} 
Fix $c>0$. Then 
\[
\overline{I(c)}\subset
\left(-\dfrac{\pi}{2}-\psi_--\dfrac{c}{6} ,\dfrac{\pi}{2}-\psi_-+\dfrac{c}{6}\right). 
\]
\end{lemma} 

\begin{proof} 
Let us prove that $\alpha\not\in\overline{I(c)}$ for any $\alpha\in
\left(-\infty, -\pi/2-\psi_--c/6\right]\cup\left[\pi/2-\psi_-+c/6,\infty\right)$. 
If $\alpha\in\left(-\infty, -\pi/2-\psi_--c/6\right]$, it follows from \eqref{Psi} that 
\[
\Psi(1;\alpha, c)
=\psi_-+\alpha+\frac{c}2\int_0^1(1-\tilde{s})^2\sin\Psi(\tilde{s};\alpha,c)\,d\tilde{s}
<\psi_-+\alpha+\dfrac{c}{6} 
\le-\dfrac{\pi}{2}. 
\]
Therefore, we have $\alpha\not\in\overline{I(c)}$. 
If $\alpha\in\left[\pi/2-\psi_-+c/6,\infty\right)$, \eqref{Psi} implies that 
\[
\Psi(1;\alpha, c) 
=\psi_-+\alpha+\frac{c}2\int_0^1(1-\tilde{s})^2\sin\Psi(\tilde{s};\alpha,c)\,d\tilde{s}
>\psi_-+\alpha-\dfrac{c}{6} 
\ge\dfrac{\pi}{2},
\]
which gives $\alpha\not\in\overline{I(c)}$. 
\end{proof} 

For $c>0$ satisfying $I(c)\ne\emptyset$, set
\begin{align*} 
\overline{\alpha}(c):=\sup I(c), \quad \underline{\alpha}(c):=\inf I(c). 
\end{align*} 
By  means of Lemma \ref{lem:conect} and Lemma \ref{lem:bound_I}, we obtain 
the following corollary. 

\medskip
\begin{cor} 
For $c>0$ satisfying $I(c)\ne\emptyset$, 
\begin{align*} 
&-\dfrac{\pi}{2}-\psi_--\dfrac{c}{6}\le\underline{\alpha}(c) 
<\overline{\alpha}(c)\le\dfrac{\pi}{2}-\psi_-+\dfrac{c}{6}, \\[0.05cm]
&I(c)=(\underline{\alpha}(c), \overline{\alpha}(c)). 
\end{align*} 
\end{cor} 

In the following, we analyze the maximum and minimum point of $\Psi(\,\cdot\,; \alpha, c)$ 
at $\alpha = \overline{\alpha}(c)$ and $\underline{\alpha}(c)$. 
To do it, the following ``energy estimate'' 
\begin{equation}\label{int_eq} 
\begin{aligned}
&\bigl[\Psi''(s; \alpha, c)\Psi'(s; \alpha, c)\bigr]_{s=s_1}^{s=s_2}
-\int_{s_1}^{s_2}\bigl(\Psi''(s; \alpha, c)\bigr)^2\;ds \\
&=-c\bigl(\cos\Psi(s_2; \alpha, c)-\cos\Psi(s_1; \alpha, c)\bigr).
\end{aligned} 
\end{equation}
for any $0 \le s_1 < s_2 \le 1$ will be useful. In order to obtain 
the equality \eqref{int_eq}, we multiply the equation $\Psi^{(3)}=c\sin\Psi$ by $\Psi'$ 
and integrate it on $s \in [s_1,s_2]$. 

\begin{lemma}\label{lem:bdy_a} 
Fix $c>0$ and assume that $I(c)\ne\emptyset$. 
If $\alpha=\overline{\alpha}(c)$ or $\underline{\alpha}(c)$, 
\begin{equation}\label{bdy_a} 
-\frac{\pi}2<\Psi(s;\alpha,c)<\frac{\pi}2\,\,\ \mbox{for}\,\ s\in[0,1). 
\end{equation} 
\end{lemma} 

\begin{proof} 
For each $\alpha=\overline{\alpha}(c)$ and $\underline{\alpha}(c)$ 
there exists a sequence $\{\alpha_n\}_{n \in \mathbb{N}}\subset I(c)$ such that 
\[
\lim_{n\to\infty}\alpha_n=\alpha. 
\]
It follows from $\alpha_n\in I(c)$ and the continuity of $\Psi$ with respect to 
$\alpha$ that 
\begin{equation}\label{min-max} 
-\frac{\pi}2\le\Psi(s;\alpha,c)\le\frac{\pi}2\,\,\ \mbox{for}\,\ s\in[0,1]. 
\end{equation} 
Note that $\Psi(0;\alpha,c)=\psi_-\ne\pm\pi/2$. 

In proof by contradiction, suppose that there exists $\hat{s} \in (0,1)$ such that 
$\Psi(\hat{s};\alpha,c)=\pi/2$. Then \eqref{min-max} implies that $\hat{s}$ is 
the maximum point of $\Psi$, so that we have
\begin{equation}\label{pt_max} 
\Psi(\hat{s};\alpha,c)=\frac{\pi}2, \quad \Psi'(\hat{s};\alpha,c)=0.
\end{equation} 
Taking $s_1=0$ and $s_2=\hat{s}$ in \eqref{int_eq} and substituting 
\eqref{pt_max} into it, we are led to 
\begin{align*} 
0\ge-\int^{\hat{s}}_0\{\Psi''(s;\alpha,c)\}^2\,ds=c\cos\psi_->0.
\end{align*} 
This is a contradiction. 

On the other hand, suppose that there exists $\hat{s} \in (0,1)$ such that 
$\Psi(\hat{s};\alpha,c)=-\pi/2$. Then, by virtue of \eqref{min-max}, we see 
that $\hat{s}$ is the minimum point of $\Psi$. This fact gives 
\[
\Psi(\hat{s};\alpha,c)=-\frac{\pi}2, \quad \Psi'(\hat{s};\alpha,c)=0.
\]
Applying a similar argument above, we also obtain a contradiction. 

Consequently, we see that $\Psi(s;\alpha,c)\ne\pm\pi/2$ for $s \in[0,1)$. 
From this fact and \eqref{min-max}, we conclude \eqref{bdy_a}. 
\end{proof} 

\begin{lemma}\label{pro_bdy_a} 
For $c>0$ satisfying $I(c)\ne\emptyset$, 
\begin{align*} 
&\Psi(1;\underline{\alpha}(c),c)=-\frac{\pi}2,\,\,\ 
\Psi'(1;\underline{\alpha}(c),c)<0,\,\,\ 
\Psi''(1;\underline{\alpha}(c),c)<0, \\
&\Psi(1;\overline{\alpha}(c),c)=\frac{\pi}2,\,\,\  
\Psi'(1;\overline{\alpha}(c),c)>0,\,\,\ 
\Psi''(1;\overline{\alpha}(c),c)>0.
\end{align*} 
\end{lemma} 

\begin{proof} 
By the definition of $I(c)$ and Lemma \ref{lem:bdy_a}, 
we have
\begin{equation}\label{value-s=1} 
|\Psi(1;\alpha,c)| = \dfrac{\pi}{2}\,\,\ \mbox{at}\,\ 
\alpha = \overline{\alpha}(c), \underline{\alpha}(c). 
\end{equation}
We only prove the properties for $\underline{\alpha}(c)$. A similar argument is 
applicable to the proof for $\overline{\alpha}(c)$. 

First, let us prove that 
\begin{equation}\label{under-a_1} 
\Psi(1;\underline{\alpha}(c),c)=-\frac{\pi}2.  
\end{equation} 
In proof by contradiction, suppose that \eqref{under-a_1} does not hold. 
Then it follows from Lemma \ref{lem:bdy_a} and \eqref{value-s=1} that 
\[
-\frac{\pi}2<\Psi(s;\underline{\alpha}(c),c)<\frac{\pi}2\,\,\ 
\mbox{for}\,\ s\in[0,1), \quad  
\Psi(1;\underline{\alpha}(c),c)=\frac{\pi}2. 
\]
Recalling Lemma \ref{lem:mono_a} and the continuity of $\Psi$ with respect to $\alpha$, 
we see
\[
-\frac{\pi}2<\Psi(s;\alpha,c)<\frac{\pi}2\,\,\ \mbox{for}\,\ s\in[0,1], 
\]
provided that $\alpha$ is close to $\underline{\alpha}(c)$ and satisfies 
$\alpha<\underline{\alpha}(c)$. This implies $\alpha\in I(c)$, which 
contradicts the definition of $\underline{\alpha}(c)$. 
Thus \eqref{under-a_1} holds. 

By means of \eqref{under-a_1} and Lemma \ref{lem:bdy_a}, we obtain
\begin{equation}\label{under-a_2} 
\Psi'(1;\underline{\alpha}(c),c)\le0.
\end{equation} 
Taking $s_1=0$ and $s_2=1$ in \eqref{int_eq} and 
using \eqref{under-a_1}, we have
\[ 
\Psi''(1;\underline{\alpha}(c),c)\Psi'(1;\underline{\alpha}(c),c) 
=c\cos\psi_-+\int_0^1\{\Psi''(s;\underline{\alpha}(c),c)\}^2\;ds>0.
\]
This implies that the sign of $\Psi''(1;\underline{\alpha}(c),c)$ is the same as 
that of $\Psi'(1;\underline{\alpha}(c),c)$ and both of them are not 
equal to $0$. Hence, by \eqref{under-a_2}, we are led to 
$\Psi'(1;\underline{\alpha}(c),c)<0$ and $\Psi''(1;\underline{\alpha}(c),c)<0$. 
\end{proof} 

From Lemma \ref{lem:mono_a}, Lemma \ref{lem:bdy_a}, \eqref{Psi2} and 
\eqref{Psi1}, we easily obtain the following result. 

\begin{cor}\label{cor:mono-deri} 
Fix $c>0$ and assume that $I(c)\ne\emptyset$. 
If $\alpha_1,\alpha_2\in[\underline{\alpha}(c),\overline{\alpha}(c)]$ and 
$\alpha_1<\alpha_2$, 
\[ 
\Psi^{(k)}(s;\alpha_1,c)<\Psi^{(k)}(s;\alpha_2,c)\,\,\ 
\mbox{for}\,\ s\in(0,1],\ k=1,2. 
\]
\end{cor} 

\medskip
Consequently, we see that for $c>0$ satisfying $I(c)\ne\emptyset$ there exists 
a unique $\hat{\alpha}(c)\in I(c)$ such that $\Psi''(1;\hat{\alpha}(c),c)=0$. 

The rest of the proof is to prove that $I(c)\ne\emptyset$ for any $c>0$. 
Set 
\begin{equation}\label{def-D}
D:=\bigcup_{c>0} I(c)\times\{c\}. 
\end{equation}
From the continuity of $\Psi$ with respect to $(\alpha,c)$, 
$D$ is an open set. 


\begin{lemma} 
$0\in I(c)$ for $c\in(0,6(\pi/2-\psi_-))$. 
\end{lemma} 

\begin{proof} 
Using \eqref{Psi}, we obtain 
\[ 
|\Psi(s;0,c)|\le\psi_-+\dfrac{c}{6}<\dfrac{\pi}{2}\,\,\ 
\mbox{for}\,\ s\in[0,1],
\]
so that $0\in I(c)$. 
\end{proof} 

Thus we can define $\overline{\alpha}(c)$ and $\underline{\alpha}(c)$ for 
$c\in(0,6(\pi/2-\psi_-))$. Set
\[ 
\hat{c}:=\sup\{\tilde{c}> 0\,|\,I(c)\ne\emptyset\,\ \mbox{for}\ c\in(0,\tilde{c})\}. 
\]

\begin{lemma}\label{lem:conti_bdy_a} 
$\overline{\alpha}(c)$, $\underline{\alpha}(c)$ is continuous on 
$(0,\hat{c} )$. 
\end{lemma} 

\begin{proof} 
We only prove the continuity of $\overline{\alpha}(c)$. 
Let us show that 
\begin{equation}\label{conti_over-a} 
\lim_{n\to\infty}\overline{\alpha}(c_n)=\overline{\alpha}(c) 
\end{equation} 
for any sequences $\{c_n\}_{n \in \mathbb{N}}\subset(0,\hat{c} )$ 
which converges to $c\in(0,\hat{c} )$. Since $c_n\to c\,(n\to\infty)$, 
there exists $n_0\in\mathbb{N}$ such that
\[
c_n\le\max\{c_1,\cdots,c_{n_0} ,c+1\} =:c_{max}\,\,\ (n\in\mathbb{N} ). 
\]
Then Lemma \ref{lem:bound_I} implies that 
\[
\overline{\alpha}(c_n)\in\overline{I(c_n)}\subset 
\left(-\dfrac{\pi}{2}-\psi_--\dfrac{c_{max}}{6} ,
\dfrac{\pi}{2}-\psi_-+\dfrac{c_{max}}{6}\right). 
\]
Thus $\{\overline{\alpha}(c_n)\}_{n\in\mathbb{N}}$ includes a convergent subsequence 
$\{\overline{\alpha}(c_{n_k} )\}_{k\in\mathbb{N}}$. Let $\alpha^\ast$ be a limit of 
its subsequence, that is, 
\[ 
\lim_{k\to\infty}\overline{\alpha}(c_{n_k} )=\alpha^\ast. 
\]
By means of the continuity of $\Psi$ with respect to $(\alpha,c)$, 
Lemma \ref{lem:bdy_a} and Lemma \ref{pro_bdy_a}, we obtain
\begin{align} 
&\Psi(1;\alpha^\ast,c)=\lim_{k\to\infty}\Psi(1;\overline{\alpha}(c_{n_k} ),c_{n_k} )
=\frac{\pi}2, \label{a-ast_at1}\\
&-\frac{\pi}2\le\Psi(s;\alpha^\ast,c)\le\frac{\pi}2\,\,\ 
\mbox{for}\,\ s\in[0,1). 
\label{a-ast_in} 
\end{align} 
Hence, it follows that 
\begin{equation}\label{conti_bdy-1} 
-\frac{\pi}2<\Psi(s;\alpha^\ast,c)<\frac{\pi}2\,\,\ \mbox{for}\,\ s\in[0,1). 
\end{equation}
Indeed, if we suppose that \eqref{conti_bdy-1} does not hold, it follows from \eqref{a-ast_in} that 
there exists $\hat{s} \in (0,1)$ such that 
\[ 
\Psi(\hat{s};\alpha^\ast,c)
=\frac{\pi}2\ \mbox{or}\ -\frac{\pi}2, \quad 
\Psi'(\hat{s};\alpha^\ast,c)=0.
\]
Taking $s_1=0$ and $s_2=\hat{s}$ in \eqref{int_eq} and 
applying a similar argument to the proof of Lemma \ref{lem:bdy_a}, 
we are led to a contradiction. As a result, we have \eqref{conti_bdy-1}. 
Using the continuity of $\Psi$ with respect to $\alpha$, Lemma \ref{lem:mono_a}, 
\eqref{a-ast_at1} and \eqref{conti_bdy-1}, we see that for any $\varepsilon>0$
\[
(\alpha^\ast-\varepsilon,\alpha^\ast)\cap I(c)\ne\emptyset, \quad
(\alpha^\ast,\alpha^\ast+\varepsilon)\cap (\mathbb{R}\setminus I(c))\ne\emptyset, 
\]
which implies $\alpha^\ast\in\partial I(c)$. Thus, by virtue of Lemma \ref{pro_bdy_a} 
and \eqref{a-ast_at1}, we have $\alpha^\ast=\overline{\alpha}(c)$. 
This means that a limit of any convergent subsequences is unique and 
its value is equal to $\overline{\alpha}(c)$, so that \eqref{conti_over-a} holds. 
\end{proof} 

\begin{lemma}\label{lem:non-empty} 
$\hat{c} =\infty$. 
\end{lemma} 

\begin{proof} 
In proof by contradiction, suppose that $\hat{c}<\infty$. Applying an argument 
similar to the first part of the proof of Lemma \ref{lem:conti_bdy_a}, 
for any sequences $\{c_n\}_{n\in\mathbb{N}}\subset(0,\hat{c} )$ satisfying 
$c_n\to\hat{c}\,(n\to\infty)$ the sequences $\{\overline{\alpha}(c_n)\}_{n\in\mathbb{N}}$ 
include a convergent subsequence $\{\overline{\alpha}(c_{n_k} )\}_{k\in\mathbb{N}}$. 
Let $\alpha^\ast$ be a limit its subsequence, that is, 
\[ 
\lim_{k\to\infty}\overline{\alpha}(c_{n_k} )=\alpha^\ast. 
\]
Using a similar argument to the second part of the proof of Lemma \ref{lem:conti_bdy_a}, 
we see 
\[ 
\Psi(1;\alpha^\ast,\hat{c} )=\frac{\pi}2, \quad 
-\frac{\pi}2<\Psi(s;\alpha^\ast,\hat{c} ),<\frac{\pi}2\,\,\ 
\mbox{for}\,\ s\in[0,1). 
\]
Thus it follows from the continuity of $\Psi$ with respect to $\alpha$ and 
Lemma \ref{lem:mono_a} that there exists a small $\varepsilon>0$ such that 
$\alpha^\ast - \varepsilon \in I(\hat{c})$. Since $D$ defined as \eqref{def-D} is open, we can choose 
a sufficiently small $r>0$ such that the ball with a radius $r$ and a center 
$(\alpha^\ast-\varepsilon,\hat{c})$ is contained in $D$.
Hence $\alpha^\ast -\varepsilon \in I(c)$ for $c \in (\hat{c}-r, \hat{c}+r)$, which contradicts 
the definition of $\hat{c}$. 
\end{proof} 

\begin{proof}[Proof of Proposition \ref{thm:exit_a}(i)] 
From Lemma \ref{lem:non-empty}, $I(c)\ne\emptyset$ for any $c>0$. 
Therefore, by applying the intermediate value theorem for $\Psi''(1;\alpha, c)$ with respect to $\alpha$ 
and using Lemma \ref{pro_bdy_a} and Corollary \ref{cor:mono-deri}, 
we obtain the unique $\hat{\alpha}(c)\in I(c)$ such that $\Psi''(1;\hat{\alpha}(c),c)=0$ for any $c>0$. 
\end{proof}

\subsection{The proofs of Proposition \ref{thm:exit_a}(ii)--(iv)}\label{sec:add}

First, we prove the continuity of $\hat{\alpha}(c)$ with respect to $c>0$.  

\begin{proof}[Proof of Proposition \ref{thm:exit_a}(ii)]
For $c>0$, let $\{c_n\}_{n \in \mathbb{N}}$ be a sequence satisfying 
$c_n>0\,(n\in\mathbb{N} )$ and $c_n\to c\,(n\to\infty)$. 
By means of a similar argument to the proof of Lemma \ref{lem:conti_bdy_a}, we see that 
$\{\hat{\alpha}(c_n)\}_{n\in\mathbb{N}}$ includes a convergent sequence 
$\{\hat{\alpha}(c_{n_k} )\}_{k\in\mathbb{N}}$. Let $\alpha^\ast$ be a limit of 
its subsequence, that is, 
\[
\lim_{k\to\infty}\hat{\alpha}(c_{n_k} )=\alpha^\ast. 
\]
Then it follows from $\hat{\alpha}(c_{n_k} )\in I(c_{n_k} )
=(\underline{\alpha}(c_{n_k} ),\overline{\alpha}(c_{n_k} ))$ and Lemma \ref{lem:conti_bdy_a} 
that 
\begin{equation}\label{bound_a-ast} 
\underline{\alpha}(c)\le\alpha^\ast\le\overline{\alpha}(c). 
\end{equation} 
Applying a similar argument to the proof of Lemma \ref{lem:conti_bdy_a}, we obtain
\[ 
\Psi''(1;\alpha^\ast,c)=0, \quad 
-\frac{\pi}2<\Psi(s;\alpha^\ast,c)<\frac{\pi}2\,\,\ \mbox{for} 
\,\ s\in[0,1). 
\]
Lemma \ref{pro_bdy_a} implies $\alpha^\ast\ne\overline{\alpha}(c),\underline{\alpha}(c)$, 
so that we get $\alpha^\ast\in I(c)$ by \eqref{bound_a-ast}. 
Since $\alpha\in I(c)$ satisfying $\Psi''(1;\alpha,c)=0$ is unique for each $c>0$, 
we have $\alpha^\ast=\hat{\alpha}(c)$. This completes the proof. 
\end{proof} 

In order to prove the negativity of $\hat{\alpha}(c)$, we prepare the following lemma.

\begin{lemma}\label{lem:delta_non-exist} 
For any $c>0$, $\Psi(\,\cdot\,; \hat{\alpha}(c), c)$ has at least one zero point in the interval $(0,1)$. 
\end{lemma} 
\begin{proof} 
In proof by contradiction, suppose that $\Psi(\,\cdot\,; \hat{\alpha}(c), c)$ does not 
have any zero points in $(0,1)$ for some $c>0$.
Then it follows from $\Psi(1; \hat{\alpha}(c), c) = \psi_- > 0$ and $\hat{\alpha}(c) \in I(c)$ 
that 
\[
0<\Psi(s;\hat{\alpha}(c),c)<\frac{\pi}2\,\,\ \mbox{for}\,\ s\in[0,1), 
\]  
so that $\sin\Psi(s;\hat{\alpha}(c),c)>0$ for $s\in[0,1)$. 
Recalling \eqref{Psi2}, we are led to
\[
\Psi''(1;\hat{\alpha},c)=c\int_0^1\sin\Psi(\tilde{s};\alpha,c)\,d\tilde{s}>0, 
\]
which contradicts $\Psi''(1; \hat{\alpha}(c), c)=0$. 
\end{proof} 

\begin{proof}[Proof of Proposition \ref{thm:exit_a}(iii)]
From \eqref{short-bdd} and Lemma \ref{lem:delta_non-exist}, there exists $\hat{s} \in (0,1)$ such that
\begin{equation}\label{posi-iii}
\Psi(s;\hat{\alpha}(c),c)>0\,\,\ \mbox{for}\,\ s\in[0,\hat{s}), \quad
\Psi(\hat{s};\hat{\alpha}(c),c)=0.
\end{equation}
Using \eqref{Psi}, we obtain
\[
0=\Psi(\hat{s};\hat{\alpha}(c),c) =\,\psi_- + \hat{\alpha}(c)\hat{s}
+\frac{c}2\int_0^{\hat{s}}(\hat{s}-\tilde{s})^2
\sin\Psi(\tilde{s};\hat{\alpha}(c),c)\,d\tilde{s}. 
\]
This equality implies $\hat{\alpha}(c) < 0$ since $\psi_-$ and $\hat{s}$ are positive and the positivity 
of $\sin \Psi$ follows from \eqref{posi-iii} and $\hat{\alpha}(c) \in I(c)$. 
\end{proof} 

Finally, we study the limit of $\hat{\alpha}(c)$ and $\Psi(\,\cdot\,;\hat{\alpha}(c), c)$ as $c \to 0+$. 
These limits play a key role in the shooting argument with respect to $c$. 

\begin{proof}[Proof of Proposition \ref{thm:exit_a}(iv)] 
Let $\{c_n\}_{n \in \mathbb{N}}$ be a sequence satisfying 
$c_n>0\,(n\in\mathbb{N} )$ and $c_n\to0\,(n\to\infty)$. 
We show that a limit of the sequence $\{\hat{\alpha}(c_n)\}$ exists and satisfies 
\begin{equation} 
\lim_{n\to\infty}\hat{\alpha}(c_n)=-2\psi_-. 
\label{limit_0+_a} 
\end{equation} 
Applying a similar argument to the proof of Lemma \ref{lem:conti_bdy_a}, we see that 
$\{\hat{\alpha}(c_n)\}_{n\in\mathbb{N}}$ includes a convergent sequence 
$\{\hat{\alpha}(c_{n_k} )\}_{k\in\mathbb{N}}$. Let $\alpha^\ast$ be a limit of 
its subsequence, that is, 
\[ 
\lim_{k\to\infty}\hat{\alpha}(c_{n_k} )=\alpha^\ast. 
\]
Using \eqref{int_eq} with $s_1=0$ and $s_2=1$, we obtain 
\begin{equation}\label{int_eq_seq} 
-c_{n_k}\bigl(\cos\psi_--\cos\Psi(1;\hat{\alpha}(c_{n_k} ),c_{n_k} )\bigr)
=\int_0^1\left(\Psi''(s;\hat{\alpha}(c_{n_k} ),c_{n_k} )\right)^2\,ds. 
\end{equation} 
Since \eqref{Psi2} implies 
\[
|\Psi''(s;\hat{\alpha}(c_{n_k} ),c_{n_k} )|
\le c_{n_k}\int_0^s|\sin\Psi(\tilde{s};\hat{\alpha}(c_{n_k} ),c_{n_k} )|\,d\tilde{s}
\le c_{n_k}\,\,\ \mbox{for}\,\ s\in[0,1], 
\]
we are led to
\[ 
0\le\int_0^1\left(\Psi''(s;\hat{\alpha}(c_{n_k} ),c_{n_k} )\right)^2\,ds
\le c_{n_k}^2. 
\]
Then it follows from \eqref{int_eq_seq} that 
\[ 
0\le-\bigl(\cos\psi_--\cos\Psi(1;\hat{\alpha}(c_{n_k} ),c_{n_k} )\bigr)\le c_{n_k}. 
\]
Taking $k\to\infty$ which gives $c_{n_k}\to0$ and using the continuity of $\Psi$ 
with respect to $(\alpha,c)$, we get 
\[ 
\Psi(1;\alpha^\ast,0)=\psi_- \quad\mbox{or}\quad \Psi(1;\alpha^\ast,0)=-\psi_-. 
\]
In proof by contradiction, suppose that $\Psi(1;\alpha^\ast, 0) = \psi_-$. 
Since $\Psi(s;\alpha^\ast, 0)=\alpha^\ast s+\psi_-$, 
we have $\alpha^\ast=0$, so that $\Psi(s;\alpha^\ast,0)\equiv\psi_-$ for $s\in[0,1]$. 
From the continuity of $\Psi$ with respect to $\alpha$ and $c$, 
$\Psi(\,\cdot\,; \hat{\alpha}(c_{n_k}),c_{n_k})$ does not have any zero points in the interval $(0,1)$ for sufficiently large $k$, which contradicts Lemma \ref{lem:delta_non-exist}.
Thus we get $\Psi(1;\alpha^\ast, 0)=-\psi_-$, which gives $\alpha^\ast=-2\psi_-$. 
This means that a limit of any convergent sequences is unique and its limit is equal to $-2\psi_-$, 
therefore we conclude the first convergence of Proposition \ref{thm:exit_a}(iv). 
The second convergence follows from $\Psi(s;-2\psi_-, 0) = -2\psi_- s + \psi_-$ and the continuity 
of $\Psi$ with respect to $\alpha$ and $c$.
\end{proof} 

\section{Zero points of $\Psi$ and its derivatives} \label{sec:zero-point}
From now on, we analyze the ``oscillation'' of $\Psi(\,\cdot\,;\hat{\alpha}(c),c)$ 
in order to use it in the shooting argument with respect to $c$. 
Set 
\[ 
\hat{\Psi}(\,\cdot\,;c) := \Psi(\,\cdot\,; \hat{\alpha}(c), c). 
\] 
In this section, we prove that for any $c>0$, the zero points of $\hat{\Psi}(\,\cdot\,;c)$, 
$\hat{\Psi}'(\,\cdot\,;c)$ and $\hat{\Psi}''(\,\cdot\,;c)$ appear alternately in the interval $(0,1)$ 
and the sign of $\hat{\Psi}$ and its derivatives changes at their zero points. 
Therefore, 
we see that the graph of $\hat{\Psi}(\,\cdot\,;c)$ should ``oscillate''. 
These properties will be proved inductively from the left zero point. 

It follows from the boundary condition $\hat{\Psi}(0;c) = \psi_- > 0$ and Proposition \ref{thm:exit_a}(iii) 
that $\hat{\Psi}(\,\cdot\,;c)>0$ and $\hat{\Psi}'(\,\cdot\,;c)<0$ 
in some short interval $[0,\varepsilon)$.  
Furthermore, the equation $\hat{\Psi}^{(3)}=c\sin\hat{\Psi}$, the positivity of $\hat{\Psi}$ 
and the boundary condition $\hat{\Psi}''(0;c) = 0$ imply that $\hat{\Psi}''(\,\cdot\,;c)>0$ 
in the interval $(0,\varepsilon)$. 
Therefore, in order to analyze the first zero points of $\Psi$ and its derivatives, 
we define 
\begin{align*} 
&\hat{\delta}^+_1(c)
:=\sup\{\delta\,|\,\Psi(s;c)>0\,\ \mbox{for}\ s\in[0,\delta)\}, \\
&\hat{\mu}^-_1(c)
:=\sup\{\mu\,|\,\hat{\Psi}'(s;c)<0\,\,\ \mbox{for}\,\ s\in[0,\mu) \}, \\
&\hat{\gamma}^+_1(c)
:=\sup\{\gamma\,|\,\hat{\Psi}''(s;c)>0\,\,\ \mbox{for}\,\ s\in(0,\gamma)\}, \\
&J_{\hat{\delta}^+_1}:=\{c>0\,|\,\mbox{$\hat{\delta}^+_1(c)$ exists in the interval $(0,1)$}\}, \\
&J_{\hat{\mu}^-_1}:=\{c>0\,|\,\mbox{$\hat{\mu}^-_1(c)$ exists in the interval $(0,1)$}\}, \\
&J_{\hat{\gamma}^+_1}:=\{c>0\,|\,\mbox{$\hat{\gamma}^+_1(c)$ exists in the interval $(0,1)$}\}.
\end{align*} 
According to Lemma \ref{lem:delta_non-exist}, $\hat{\delta}^+_1(c)$ exists 
in the interval $(0,1)$ for any $c>0$, so that $J_{\hat{\delta}^+_1}=(0,\infty)$. 
In addition, by virtue of $\hat{\Psi}''(1;c)=0$, $\hat{\gamma}^+_1(c)$ exists 
in $(0,1]$ for any $c>0$. 
The existence of $\hat{\mu}_1^-(c)$ and $\hat{\gamma}_1^+(c)$ in the interval $(0,1)$ for 
sufficiently large $c>0$, which implies that $J_{\hat{\mu}^-_1}$ 
and $J_{\hat{\gamma}^+_1}$ are not empty, will be proved in the next section. 

\begin{lemma}\label{lem:com_d1+g1+} 
$\hat{\delta}^+_1(c)<\hat{\gamma}^+_1(c)$ for any $c>0$. 
\end{lemma} 

\begin{proof} 
Using \eqref{Psi2} and the definition of $\hat{\gamma}^+_1(c)$, 
we obtain 
\begin{equation}\label{int_sin_0}
0=\hat{\Psi}''(\hat{\gamma}^+_1(c),c)
=c\int_0^{\hat{\gamma}^+_1(c)}\sin\hat{\Psi}(s;c)\,ds. 
\end{equation}
Thus we can easily see $\hat{\delta}^+_1(c)<\hat{\gamma}^+_1(c)$. 
\end{proof}

\begin{lemma}\label{lem:com_d1+m1-} 
$\hat{\delta}^+_1(c)<\hat{\mu}^-_1(c)$ for $c\in J_{\hat{\mu}^-_1}$. 
\end{lemma} 

\begin{proof} 
In proof by contradiction, suppose that $\hat{\delta}^+_1(c)\ge\hat{\mu}^-_1(c)$. 
Then, by virtue of Lemma \ref{lem:com_d1+g1+}, we see 
$\hat{\mu}^-_1(c)<\hat{\gamma}^+_1(c)$. Since $\hat{\Psi}'(0;c)=\hat{\alpha}(c)<0$ 
by Proposition \ref{exit_a} (iii), it follows from this fact and the definition of 
$\hat{\gamma}^+_1(c)$ that 
\[
\hat{\Psi}'(s;c)\left\{\begin{array}{ll} 
<0&(s\in[0,\hat{\mu}^-_1(c))), \\ =0&(s=\hat{\mu}^-_1(c)), \\
>0&(s\in(\hat{\mu}^-_1(c),\hat{\gamma}^+_1(c)]). 
\end{array}\right. 
\]
This implies that $\hat{\Psi}(s;c)\ge\hat{\Psi}(\hat{\mu}^-_1(c);c)$ for 
$s\in[0,\hat{\gamma}^+_1(c)]$. Moreover, by means of 
$\hat{\delta}^+_1(c)\ge\hat{\mu}^-_1(c)$, we see $\hat{\Psi}(\hat{\mu}^-_1(c);c)\ge0$, 
so that $\hat{\Psi}(s;c)\ge0$ for $s\in[0,\hat{\gamma}^+_1(c)]$. 
Since $\hat{\Psi}$ is not a constant function, this fact and \eqref{Psi2} give
\[
\hat{\Psi}''(\hat{\gamma}^+_1(c),c)
=c\int_0^{\hat{\gamma}^+_1(c)}\sin\hat{\Psi}(\tilde{s};c)\,d\tilde{s}>0, 
\]
which contradicts $\hat{\Psi}''(\hat{\gamma}^+_1(c),c)=0$. 
\end{proof}

\begin{lemma}\label{lem:com_m1-g1+} 
If $c\in J_{\hat{\gamma}^+_1}$, then $c\in J_{\hat{\mu}^-_1}$ and 
$\hat{\mu}^-_1(c)<\hat{\gamma}^+_1(c)$. 
\end{lemma} 

\begin{proof} 
In proof by contradiction, suppose that $c\not\in J_{\hat{\mu}^-_1}$ or 
$\hat{\mu}^-_1(c)\ge\hat{\gamma}^+_1(c)$. Recalling the definition of $\hat{\gamma}^+_1(c)$, 
we have 
\begin{equation} 
\hat{\Psi}''(s;c)>0\,\,\ \mbox{for}\,\ s\in(0,\hat{\gamma}^+_1(c)), \quad
\hat{\Psi}''(\hat{\gamma}^+_1(c);c)=0.
\label{Psi2+_sup} 
\end{equation} 
By virtue of $\hat{\Psi}'(0;c)=\hat{\alpha}(c)<0$ and the supposition that 
$c\not\in J_{\hat{\mu}^-_1}$ or $\hat{\mu}^-_1(c)\ge\hat{\gamma}^+_1(c)$, 
we see
\[
\hat{\Psi}'(s;c)<0\,\,\ \mbox{for}\,\ s\in[0,\hat{\gamma}^+_1(c)). 
\]
Thus it follows from this fact and Lemma \ref{lem:com_d1+g1+} that 
$\hat{\Psi}(\hat{\gamma}^+_1(c);c)<0$, so that 
\[
\hat{\Psi}^{(3)}(\hat{\gamma}^+_1(c);c)
=c\sin\hat{\Psi}(\hat{\gamma}^+_1(c);c)<0.
\]
Since $\hat{\gamma}^+_1(c)<1$ by $c\in J_{\hat{\gamma}^+_1}$ and $\hat{\Psi}^{(3)}$ is 
continuous with respect to $s$, there exists $\e>0$ such that
\[
\hat{\Psi}^{(3)}(s;c)<0\,\,\ \mbox{for}\,\ 
s\in[\hat{\gamma}^+_1(c),\hat{\gamma}^+_1(c)+\e). 
\]
Then, with the help of \eqref{Psi2+_sup}, we obtain
\[
\hat{\Psi}''(s;c)<0\,\,\ \mbox{for}\,\ 
s\in(\hat{\gamma}^+_1(c),\hat{\gamma}^+_1(c)+\e). 
\]
Set
\begin{equation}\label{def_gamma1-} 
\hat{\gamma}^-_1(c):=\sup\{\gamma\,|\,
\hat{\Psi}''(s;c)<0\,\ \mbox{for}\ s\in(\hat{\gamma}^+_1(c),\gamma)\}. 
\end{equation} 
Note that $\hat{\gamma}^-_1(c)\le1$ due to $\hat{\gamma}^+_1(c)<\hat{\gamma}^-_1(c)$ 
and $\hat{\Psi}''(1;c)=0$. Then, by the definition of 
$\hat{\gamma}^-_1(c)$, we have
\[
\hat{\Psi}''(s;c)<0\,\,\ \mbox{for}\,\  
s\in(\hat{\gamma}^+_1(c),\hat{\gamma}^-_1(c)), \quad 
\hat{\Psi}''(\hat{\gamma}^-_1(c);c)=0.
\]
This fact and $\hat{\Psi}'(\hat{\gamma}^+_1(c);c)\le0$ imply
\[
\hat{\Psi}'(s;c)<0\,\,\ \mbox{for}\,\  
s\in(\hat{\gamma}^+_1(c),\hat{\gamma}^-_1(c)]. 
\]
Recalling $\hat{\Psi}(\hat{\gamma}^+_1(c);c)<0$, we are led to 
\[
\hat{\Psi}(s;c)<0\,\,\ \mbox{for}\,\  
s\in(\hat{\gamma}^+_1(c),\hat{\gamma}^-_1(c)]. 
\]
which gives
\[
0=\hat{\Psi}''(\hat{\gamma}^-_1(c);c)-\hat{\Psi}''(\hat{\gamma}^+_1(c);c)
=c\int_{\hat{\gamma}^+_1(c)}^{\hat{\gamma}^-_1(c)} 
\sin\hat{\Psi}(s;c)\,ds<0.
\]
This is a contradiction. 
\end{proof} 

By virtue of $\hat{\Psi}'(s;c)<0$ for $s\in[0,\min\{1,\mu^-_1(c)\} )$, 
$\hat{\delta}^+_1(c)<1$ and Lemma \ref{lem:com_d1+m1-}, there exists 
$\varepsilon>0$ such that $\hat{\Psi}(s;c)<0$ for 
$s\in(\hat{\delta}^+_1(c),\hat{\delta}^+_1(c)+\varepsilon)$. 
Set 
\begin{align*}
&\hat{\delta}^-_1(c)
:=\sup\{\delta\,|\,\hat{\Psi}(s;c)<0\,\ \mbox{for}\ 
s\in(\hat{\delta}^+_1(c),\delta)\}, \\
&J_{\hat{\delta}^-_1}:=\{c>0\,|\,\mbox{$\hat{\delta}^-_1(c)$ exists in $(0,1)$}\}.
\end{align*}

\begin{lemma}\label{lem:com_m1-d1-} 
If $c\in J_{\hat{\delta}^-_1}\cup\{c>0\,|\,\hat{\delta}^-_1(c)=1\}$, then $c\in J_{\hat{\mu}^-_1}$ 
and $\hat{\mu}^-_1(c)<\hat{\delta}^-_1(c)$. 
\end{lemma} 

\begin{proof} 
By means of Lemma \ref{lem:com_d1+m1-}, we have $\hat{\delta}^+_1(c)<\min\{1,\mu^-_1(c)\}$, 
which gives 
\begin{equation}\label{Psi1_d1+} 
\hat{\Psi}'(\hat{\delta}^+_1(c);c)<0.
\end{equation}
Since $c\in J_{\hat{\delta}^-_1}\cup\{c>0\,|\,\hat{\delta}^-_1(c)=1\}$ implies $\delta^-_1(c)\le1$, 
it follows from the definition of $\hat{\delta}^{\pm}_1(c)$ and Rolle's theorem that there exists 
$\mu^\ast(c)\in(\hat{\delta}^+(c),\hat{\delta}^-(c))$ such that $\hat{\Psi}'(\mu^\ast(c);c)=0$. 
Thus, by virtue of this fact, \eqref{Psi1_d1+} and the definition of $\mu^-_1(c)$, we see that 
$\mu^-_1(c)$ exists in $(\hat{\delta}^+(c),\mu^\ast(c)]\subset(0,1)$, so that we obtain 
$c\in J_{\hat{\mu}^-_1}$ and $\mu^-_1(c)\le\mu^\ast(c)<\hat{\delta}^-_1(c)$. 
\end{proof} 

\begin{lemma}\label{lem:com_g1+d1-} 
$\hat{\gamma}^+_1(c)<\hat{\delta}^-_1(c)$ for $c\in J_{\hat{\delta}^-_1}$. 
\end{lemma} 

\begin{proof} 
In proof by contradiction, suppose that $\hat{\gamma}^+_1(c)\ge\hat{\delta}^-_1(c)$. 
Since Lemma \ref{lem:com_m1-d1-} implies that $\hat{\mu}^-_1(c)$ exists in $(0,1)$, 
we have $\hat{\mu}^-_1(c)<\hat{\delta}^-_1(c)\le\hat{\gamma}^+_1(c)$. 

In the case of $\hat{\gamma}^+_1(c)>\hat{\delta}^-_1(c)$, we see 
\[
\hat{\Psi}''(\hat{\delta}^-_1(c);c)>0, \quad 
\hat{\Psi}(s;c)>0\,\,\ \mbox{for}\,\  
s\in(\hat{\delta}^-_1(c),\hat{\gamma}^+_1(c)). 
\]
Thus it follows that 
\[
\hat{\Psi}''(\hat{\gamma}^+_1(c),c)
=\hat{\Psi}''(\hat{\delta}^-_1(c),c)
+c\int_{\hat{\delta}^-_1(c)}^{\hat{\gamma}^+_1(c)}\sin\hat{\Psi}(s;c)\,ds>0.
\]
This contradicts $\hat{\Psi}''(\hat{\gamma}^+_1(c),c)=0$. 

In the case of $\hat{\gamma}^+_1(c)=\hat{\delta}^-_1(c)\,(\,<1)$, we have 
\[
\hat{\Psi}'(\hat{\gamma}^+_1(c);c)>0, \quad \hat{\Psi}(\hat{\gamma}^+_1(c);c)=0.
\]
Since  $\hat{\Psi}$ and $\hat{\Psi}'$ are continuous with respect to $s$, 
there exists $\varepsilon>0$ such that 
\[
\hat{\Psi}(s;c)>0\,\,\ \mbox{for}\,\  
s\in(\hat{\gamma}^+_1(c),\hat{\gamma}^+_1(c)+\e]. 
\]
Then we see
\[
\hat{\Psi}^{(3)}(s;c)=c\sin\hat{\Psi}(s;c)>0\,\,\ \mbox{for}\,\ 
s\in(\hat{\gamma}^+_1(c),\hat{\gamma}^+_1(c)+\e]. 
\]
This fact and $\hat{\Psi}''(\hat{\gamma}^+_1(c);c)=0$ imply 
\[
\hat{\Psi}''(s;c)>0\,\,\ \mbox{for}\,\  
s\in(\hat{\gamma}^+_1(c),\hat{\gamma}^+_1(c)+\e]. 
\]
Set
\[
\widetilde{\gamma}^+_2(c)
:=\sup\{\gamma\,|\,\hat{\Psi}''(s;c)>0\,\ \mbox{for}\ 
s\in(\hat{\gamma}^+_1(c),\gamma)\}. 
\]
Note that $\widetilde{\gamma}^+_2(c)\le1$ by virtue of 
$\hat{\gamma}^+_1(c)<\widetilde{\gamma}^+_2(c)$ and $\hat{\Psi}''(1;c)=0$. 
With the help of the definition of $\widetilde{\gamma}^+_2(c)$, we obtain 
\[
\hat{\Psi}''(s;c)>0\,\,\ \mbox{for}\,\ 
s\in(\hat{\gamma}^+_1(c),\widetilde{\gamma}^+_2(c)), \quad 
\hat{\Psi}''(\widetilde{\gamma}^+_2(c);c)=0.
\]
It follows from $\hat{\Psi}'(\hat{\gamma}^+_1(c);c)>0$ that  
\[
\hat{\Psi}'(s;c)>0\,\,\ \mbox{for}\,\ 
s\in(\hat{\gamma}^+_1(c),\widetilde{\gamma}^+_2(c)]. 
\]
Thus, by means of this fact and $\hat{\Psi}(\hat{\gamma}^+_1(c);c)=0$, 
we are led to 
\[
\hat{\Psi}(s;c)>0\,\,\ \mbox{for}\,\ 
s\in(\hat{\gamma}^+_1(c),\widetilde{\gamma}^+_2(c)],
\]
so that 
\[
0=\hat{\Psi}''(\widetilde{\gamma}^+_2(c),c)-\hat{\Psi}''(\hat{\gamma}^+_1(c),c)
=c\int_{\hat{\gamma}^+_1(c)}^{\widetilde{\gamma}^+_2(c)} 
\sin\hat{\Psi}(s;c)\,ds>0.
\]
This is a contradiction. 
\end{proof} 

\noindent
This lemma implies the following corollary. 

\begin{cor}\label{cor:com_g1+d1-}
If $c\in J_{\hat{\delta}^-_1}$, then $c\in J_{\hat{\gamma}^+_1}$. 
\end{cor}

By means of Lemma \ref{lem:com_d1+g1+} and Lemma \ref{lem:com_g1+d1-}, 
we have $\hat{\Psi}(\hat{\gamma}^+_1(c);c)<0$, which implies 
$\hat{\Psi}^{(3)}(\hat{\gamma}^+_1(c);c)=c\sin\hat{\Psi}(\hat{\gamma}^+_1(c);c)<0$. 
It follows from this fact and $\hat{\Psi}''(\hat{\gamma}^+_1(c);c)=0$ that 
there exists $\varepsilon>0$ such that $\hat{\Psi}''(s;c)<0$ for $s\in 
(\hat{\gamma}^+_1(c),\hat{\gamma}^+_1(c)+\varepsilon)$. 
Then we define $\hat{\gamma}^-_1(c)$ as \eqref{def_gamma1-}. 

\begin{lemma}\label{lem:com_d1-g1-} 
If $c\in J_{\hat{\gamma}^+_1}$, then $c\in J_{\hat{\delta}^-_1}$ and 
$\hat{\delta}^-_1(c)<\hat{\gamma}^-_1(c)$. 
\end{lemma} 

\begin{proof} 
By virtue of $c\in J_{\hat{\gamma}^+_1}$ and $\hat{\Psi}''(1;c)=0$, 
we have $\hat{\gamma}^-_1(c)\le1$. Then it follows from the definition 
of $\gamma^\pm_1(c)$ and Rolle's theorem that there exists 
$\delta^\ast(c)\in(\hat{\gamma}^+_1(c),\gamma^-_1(c))$ such that 
$\hat{\Psi}^{(3)}(\delta^\ast(c);c)=0$, which implies 
$\hat{\Psi}(\delta^\ast(c);c)=0$. Recalling $\hat{\Psi}(\hat{\gamma}^+_1(c);c)<0$ and 
the definition of $\delta^-_1(c)$, we see that $\delta^-_1(c)$ exists in 
$(\hat{\gamma}^+_1(c),\delta^\ast(c)]\subset(0,1)$, so that we obtain $c\in J_{\hat{\delta}^-_1}$ 
and $\hat{\delta}^-_1(c)\le\delta^\ast(c)<\hat{\gamma}^-_1(c)$. 
\end{proof} 

As a result, Corollary \ref{cor:com_g1+d1-} and Lemma \ref{lem:com_d1-g1-} imply 
\begin{equation}\label{Jg1+d1-}
J_{\hat{\gamma}^+_1}=J_{\hat{\delta}^-_1}.
\end{equation}
Moreover, it follows from Lemma \ref{lem:com_d1+g1+}\,-\,\ref{lem:com_d1-g1-} that 
if $c\in J_{\hat{\gamma}^+_1}$, then 
\[
\delta^+_1(c)<\mu^-_1(c)<\hat{\gamma}^+_1(c)
<\delta^-_1(c)\,(\,<\gamma^-_1(c)\le1). 
\] 
For $j\in\mathbb{N}$, set
\begin{align} 
&\hat{\delta}^-_j(c)
:=\sup\{\delta\,|\,\hat{\Psi}(s;c)<0\,\ \mbox{for}\ 
s\in(\hat{\delta}^+_j(c),\delta)\}, \notag\\
&\hat{\delta}^+_{j+1}(c)
:=\sup\{\delta\,|\,\hat{\Psi}(s;c)>0\,\ \mbox{for}\ 
s\in(\hat{\delta}^-_j(c),\delta)\}, \notag\\
&\hat{\mu}^+_j(c)
:=\sup\{\mu\,|\,\hat{\Psi}'(s;c)>0\,\ \mbox{for}\ 
s\in(\hat{\mu}^-_j(c),\mu)\}, \notag\\
&\hat{\mu}^-_{j+1}(c)
:=\sup\{\mu\,|\,\hat{\Psi}'(s;c)<0\,\ \mbox{for}\ 
s\in(\hat{\mu}^+_j(c),\mu)\}, \notag\\
&\hat{\gamma}^-_j(c)
:=\sup\{\gamma\,|\,\hat{\Psi}''(s;c)<0\,\ \mbox{for}\ 
s\in(\hat{\gamma}^+_j(c),\gamma)\}, \notag\\
&\hat{\gamma}^+_{j+1}(c)
:=\sup\{\gamma\,|\,\hat{\Psi}''(s;c)>0\,\ \mbox{for}\ 
s\in(\hat{\gamma}^-_j(c),\gamma)\}, \notag\\
&J_{\hat{\delta}^{\pm}_j}:=\{c>0\,|\,\mbox{$\hat{\delta}^{\pm}_j(c)$ exists in the interval $(0,1)$}\}, \label{def-dipm}\\
&J_{\hat{\mu}^{\pm}_j}:=\{c>0\,|\,\mbox{$\hat{\mu}^{\pm}_j(c)$ exists in the interval $(0,1)$}\}, 
\label{def-mipm}\\
&J_{\hat{\gamma}^{\pm}_j}:=\{c>0\,|\,\mbox{$\hat{\gamma}^{\pm}_j(c)$ exists in the interval $(0,1)$}\}. \notag
\end{align} 
We can obtain the order of zero points of $\hat{\Psi}$ and its derivatives by induction  
consisting of the following steps. 
Under the assumption $\hat{\delta}^+_j(c)<\hat{\gamma}^+_j(c)$ for any $c\in J_{\hat{\delta}^+_j}$, 
we can prove each step by a similar argument to the proof of 
Lemma \ref{lem:com_d1+m1-}--\ref{lem:com_d1-g1-}. 

\begin{list}{}{\topsep=0.25cm\itemsep=0cm\leftmargin=0cm\itemindent=0.25cm} 
\item Step 1:  
$\hat{\delta}^+_j(c)<\hat{\mu}^-_j(c)$ for $c\in J_{\hat{\mu}^-_j}$. 
Moreover, if $c\in J_{\hat{\gamma}^+_j}$, then $c\in J_{\hat{\mu}^-_j}$ and 
$\hat{\mu}^-_j(c)<\hat{\gamma}^+_j(c)$. 

\item Step 2: 
If $c\in J_{\hat{\delta}^-_j}\cup\{c>0\,|\,\hat{\delta}^-_j(c)=1\}$, then $c\in J_{\hat{\mu}^-_j}$ 
and $\hat{\mu}^-_j(c)<\hat{\delta}^-_j(c)$. 

\item Step 3: 
$\hat{\gamma}^+_j(c)<\hat{\delta}^-_j(c)$ for $c\in J_{\hat{\delta}^-_j}$. 

\item Step 4: 
If $c\in J_{\hat{\gamma}^+_j}$, then $c\in J_{\hat{\delta}^-_j}$ and 
$\hat{\delta}^-_j(c)<\hat{\gamma}^-_j(c)$. 

\item Step 5: 
$\hat{\delta}^-_j(c)<\hat{\mu}^+_j(c)$ for $c\in J_{\hat{\mu}^+_j}$. 
Moreover, if $c\in J_{\hat{\gamma}^-_j}$, then $c\in J_{\hat{\mu}^+_j}$ and 
$\hat{\mu}^+_j(c)<\hat{\gamma}^-_j(c)$. 

\item Step 6: 
If $c\in J_{\hat{\delta}^+_{j+1}}\cup\{c>0\,|\,\hat{\delta}^+_{j+1}(c)=1\}$, 
then $c\in J_{\hat{\mu}^+_j}$ and $\hat{\mu}^+_j(c)<\hat{\delta}^+_{j+1}(c)$. 

\item Step 7: 
$\hat{\gamma}^-_j(c)<\hat{\delta}^+_{j+1}(c)$ for $c\in J_{\hat{\delta}^+_{j+1}}$. 

\item Step 8: 
If $c\in J_{\hat{\gamma}^-_j}$, then $c\in J_{\hat{\delta}^+_{j+1}}$ and 
$\hat{\delta}^+_{j+1}(c)<\hat{\gamma}^+_{j+1}(c)$. 

\end{list} 

Note that Step 3, 4, 7 and 8 lead us to the equality 
\begin{equation}\label{Jgj+dj-}
J_{\hat{\gamma}^+_j} = J_{\hat{\delta}^-_j}, \quad J_{\hat{\gamma}^-_j} = J_{\hat{\delta}^+_{j+1}}. 
\end{equation}
We have already obtained in Lemma \ref{lem:com_d1+g1+} that 
$\hat{\delta}_1^+(c) < \hat{\gamma}_1^+(c)$ for any $c \in J_{\hat{\delta}_1^+} = (0,\infty)$, 
which is the first step of this induction. 
Carrying out the above steps for $j=1$, we have 
$\hat{\delta}_2^+(c) < \hat{\gamma}^+_2(c)$ for any $c \in J_{\hat{\delta}^+_2}$. 
Therefore, we can carry out them again for $j=2$ and repeat this argument for any 
$j \in \mathbb{N}$. As a result, we are led to the following proposition.

\begin{prop}\label{prop:com_zero-point}
\begin{list}{}{\topsep=0cm\itemsep=0cm\leftmargin=0cm\itemindent=0.25cm}
\item[(i)] 
The following inclusion holds. 
\[ 
\begin{aligned}
(0,\infty) =&\; J_{\hat{\delta}^+_1} \supset J_{\hat{\mu}_1^-} \supset 
J_{\hat{\gamma}_1^+} = J_{\hat{\delta}_1^-} \supset \cdots \\
\supset&\; J_{\hat{\gamma}_{j-1}^-} = J_{\hat{\delta}_j^+} \supset 
J_{\hat{\mu}_j^-} \supset J_{\hat{\gamma}_j^+} = J_{\hat{\delta}_j^-} 
\supset J_{\hat{\mu}_j^+} \supset J_{\hat{\gamma}_j^-} 
= J_{\hat{\delta}_{j+1}^+} \supset \cdots. 
\end{aligned}
\]
\item[(ii)] 
For any $j \in \mathbb{N}\setminus\{1\}$ and $c \in J_{\hat{\delta}_j^+}$, the zero points 
of $\hat{\Psi}(\,\cdot\,;c)$ and its derivatives appear alternately in the interval $(0,1)$ as 
\[ 
\hat{\delta}_1^+(c) < \cdots < \hat{\delta}_{j-1}^-(c) < \hat{\mu}_{j-1}^+(c) < \hat{\gamma}_{j-1}^-(c) 
< \hat{\delta}_j^+(c)\, ( < \hat{\gamma}_j^+(c) \le 1). 
\]
\item[(iii)] 
For any $j \in \mathbb{N}$ and $c \in J_{\hat{\delta}_j^-}$, the zero points of $\hat{\Psi}(\,\cdot\,;c)$ 
and its derivatives appear alternately in the interval $(0,1)$ as 
\[ 
\hat{\delta}_1^+(c) < \cdots < \hat{\delta}_j^+(c) < \hat{\mu}_j^-(c) < \hat{\gamma}_j^+(c) 
< \hat{\delta}_j^-(c)\, ( < \hat{\gamma}_j^-(c) \le 1). 
\]
\item[(iv)] 
For any $j \in \mathbb{N}$ and $c \in J_{\hat{\mu}_j^-}$ $($resp.\ $J_{\hat{\mu}_j^+}$$)$, 
the zero points of $\hat{\Psi}(\,\cdot\,;c)$ and its derivatives appear alternately in the interval 
$(0,1)$ as 
\[ 
\begin{aligned}
\hat{\delta}_1^+(c) < &\;\cdots < \hat{\delta}_j^+(c) < \hat{\mu}_j^-(c)\, 
( < \hat{\gamma}_j^+(c) \le 1) \\
&\;(\mbox{resp.\ $\cdots < \hat{\gamma}_j^+(c) < \hat{\delta}_j^-(c) < \hat{\mu}_j^+(c)\, 
( < \hat{\gamma}_j^-(c) \le 1)$}) 
\end{aligned}
\]
\item[(v)] 
For any $j \in \mathbb{N}$, $c \in J_{\hat{\mu}_j^-}$ (resp.\ $c \in J_{\hat{\mu}_j^+}$) and 
$\hat{\mu}_j^-(c) < \hat{\delta}_j^-(c)$ (resp.\ $\hat{\mu}_j^+(c) < \hat{\delta}_{j+1}^+(c)$) 
if $\hat{\delta}_j^-(c) = 1$ (resp.\ $\hat{\delta}_{j+1}^+(c) = 1$). 
\end{list}
\end{prop}

Note that the sign of $\hat{\Psi}$ and its derivatives changes at each zero point 
due to the equation $\hat{\Psi}^{(3)}=c\sin\hat{\Psi}$ and 
the order of zero points of $\hat{\Psi}$ and its derivatives in Proposition \ref{prop:com_zero-point}. 

\section{Estimates of $\Psi$ and its derivatives}\label{sec:order-est}

In Section \ref{sec:zero-point}, we obtained that the graph of $\hat{\Psi}(\,\cdot\,;c)$ 
``oscillates''. 
As we mentioned in the beginning of Section \ref{sec:initial}, the period of 
the oscillation of $\hat{\Psi}$ contracts as $c$ increase and 
the amplitude of the oscillation around $s=1$ is smaller than that around $s=0$. 
In this section, we first prove ``the contraction of the period'' 
by evaluating the distance between each pair of the zero points of $\hat{\Psi}$ and 
its derivatives with the order of convergence on $c>0$. 
These estimates imply that new zero points appear 
in the interval $(0,1)$ if $c>0$ becomes sufficiently large. 
In order to analyze ``the gap of the amplitude'', we derive the estimates  
of $\hat{\Psi}$ and its derivatives with the order of convergence on $c>0$ 
at each zero point.
In this analysis, ``energy loss'' as in \eqref{main-1-2} plays a key role. 
The following lemma is relevant to ``energy loss''. 

\begin{lemma}\label{lem:com_Psi1-2} 
For any $c>0$, the following properties hold. 

\noindent
(i) If $s_i\,(i=1,2)$ with $0\le s_1<s_2\le1$ 
satisfy the following (a) and (b):
\begin{center} 
(a) $\hat{\Psi}'(s_1;c)=0$ or $\hat{\Psi}''(s_1;c)=0$, \quad
(b) $\hat{\Psi}'(s_2;c)=0$ or $\hat{\Psi}''(s_2;c) = 0$,  
\end{center} 
then
\begin{equation}\label{com_Psi} 
|\hat{\Psi}(s_1;c)| \ge |\hat{\Psi}(s_2;c)|. 
\end{equation}
(ii) If $s_i\,(i=1,2)$ with $0\le s_1<s_2\le1$ satisfy 
the following (a) and (b):
\begin{center} 
(a) $\hat{\Psi}(s_1;c)=0$ or $\hat{\Psi}''(s_1;c)=0$, \quad
(b) $\hat{\Psi}(s_2;c)=0$ or $\hat{\Psi}''(s_2;c)=0$, 
\end{center} 
then 
\[
|\hat{\Psi}'(s_1;c)|\ge|\hat{\Psi}'(s_2;c)|. 
\]
(iii) If $s_i\,(i=1,2)$ with $0\le s_1<s_2\le1$ satisfy 
the following (a) and (b):
\begin{center} 
(a) $\hat{\Psi}(s_1;c)=0$ or $\hat{\Psi}'(s_1;c)=0$, \quad
(b) $\hat{\Psi}(s_2;c)=0$ or $\hat{\Psi}'(s_2;c)=0$,  
\end{center} 
then
\[
|\hat{\Psi}''(s_1;c)|\ge|\hat{\Psi}''(s_2;c)|. 
\]
\end{lemma} 

\begin{proof} 
For (i), recalling \eqref{int_eq} and using the conditions (a) and (b), we have
\[
-c\bigl(\cos\hat{\Psi}(s_2;c)-\cos\hat{\Psi}(s_1;c)\bigr)
=-\int_{s_1}^{s_2}\bigl(\hat{\Psi}''(s;c)\bigr)^2\,ds\le0.
\]
Since $\alpha\in I(c)$, this implies \eqref{com_Psi}. 

Let us prove (ii). Multiplying $\hat{\Psi}$ by $\hat{\Psi}^{(3)} =c\sin\hat{\Psi}$ and integrating 
it by parts in $[s_1,s_2]$, we obtain
\begin{equation}\label{int_eq_Psi} 
\left[\hat{\Psi}''(s;c)\hat{\Psi}(s;c) - \frac{\bigl(\hat{\Psi}'(s;c)\bigr)^2}2\right]_{s=s_1}^{s=s_2} 
=c\int_{s_1}^{s_2}\hat{\Psi}(s;c)\sin\hat{\Psi}(s;c)\,ds. 
\end{equation} 
Since $\hat{\Psi}\sin\hat{\Psi}>0$, we are led to (ii). 

Let us show (iii). Multiplying $\hat{\Psi}''$ by $\hat{\Psi}^{(3)} =c\sin\hat{\Psi}$ and integrating 
it by parts in $[s_1,s_2]$, we get
\begin{equation}\label{int_eq_Psi2} 
\left[\frac{\bigl(\hat{\Psi}''(s;c)\bigr)^2}2\right]_{s=s_1}^{s=s_2} 
=c\bigl[\hat{\Psi}'(s;c)\sin\hat{\Psi}(s;c)\bigr]_{s=s_1}^{s=s_2} 
-c\int_{s_1}^{s_2}\bigl(\hat{\Psi}'(s;c)\bigr)^2\cos\hat{\Psi}(s;c)\,ds. 
\end{equation} 
Since $\cos\hat{\Psi}>0$, we see (iii). 
\end{proof} 

By virtue of Lemma \ref{lem:com_Psi1-2}, 
we  easily derive the following corollary. 

\begin{cor}\label{cor:bound_Psi0-2} 
For $s\in[0,1]$ and $c>0$, 
\begin{align*} 
&|\hat{\Psi}(s;c)|\le|\hat{\Psi}(0;c)|=\psi_-, \\
&|\hat{\Psi}'(s;c)|\le|\hat{\Psi}'(0;c)|=-\hat{\alpha}(c), \\ 
&|\hat{\Psi}''(s;c)|\le|\hat{\Psi}''(\hat{\delta}^+_1(c);c)|=\hat{\Psi}''(\hat{\delta}^+_1(c);c). 
\end{align*} 
\end{cor} 

Recalling Proposition \ref{thm:exit_a}(iv), we have 
\[
\lim_{c\to+0}\hat{\Psi}'(s;c)=-2\psi_-<0\,\,\ \mbox{for}\,\ s \in [0,1], 
\]
which is uniform convergence. 
Thus, there exists $c_\ast>0$ such that 
\begin{equation}\label{def_c_ast}
\hat{\Psi}'(s;c)< 0\,\,\ \mbox{for}\,\ s\in[0,1],\,\ c\in(0,c_\ast). 
\end{equation}
This means that $\hat{\mu}^-_1(c)$ does not exist in the interval $(0,1)$ for $c\in(0,c_\ast)$. 

Let us derive the estimates of $\hat{\Psi}$ and its derivatives with the order 
of convergence on $c>0$ at zero points.  
These estimates can be proved inductively from the left zero points. 
First, we show the estimates of $\hat{\alpha}(c)$ and $\delta^+_1(c)$ 
with the order of convergence on $c>0$. 

\begin{lemma}\label{lem:c-rate_a} 
There exist $M_{\hat{\alpha}} ,m_{\hat{\alpha}}>0$ such that 
\[
-M_{\hat{\alpha}}c^{1/3}\le\hat{\alpha}(c)\le-m_{\hat{\alpha}}c^{1/3}\,\,\ 
\mbox{for}\,\ c\ge c_\ast. 
\]
\end{lemma} 

\begin{proof} 
First, let us derive the lower bound. 
Since $(c_\ast/c)^{1/3}$ is in $(0,1]$ for $c\ge c_\ast$, we obtain by \eqref{Psi}
\begin{align*} 
-\psi_-
<&\,\hat{\Psi}((c_\ast/c)^{1/3};c) \\
=&\,\psi_-+\hat{\alpha}(c)\cdot(c_\ast/c)^{1/3} 
+\frac{c}2\int_0^{(c_\ast/c)^{1/3}}((c_\ast/c)^{1/3}-s)^2\sin\hat{\Psi}(s;c)\,ds \\
=&\,\psi_-+c_\ast^{1/3}\hat{\alpha}(c)c^{-1/3} 
+\frac{c_\ast}2\int_0^1(1-\xi)^2\sin\hat{\Psi}((c_\ast/c)^{1/3}\xi;c)\,d\xi \\
\le&\,\psi_-+c_\ast^{1/3}\hat{\alpha}(c)c^{-1/3}+\frac{c_\ast} 6\sin\psi_-. 
\end{align*} 
This implies
\[
\hat{\alpha}(c)\ge
-\frac1{c_\ast^{1/3}}\left(2\psi+\frac{c_\ast}6\sin\psi_-\right)c^{1/3}. 
\]
Setting
\[
M_{\hat{\alpha}}:=\frac1{c_\ast^{1/3}}\left(2\psi+\frac{c_\ast}6\sin\psi_-\right), 
\]
we are led to the lower bound. 

Let us derive the upper bound. To do it, we divide in the following two cases: 
\begin{list}{}{\topsep=0.15cm\itemsep=0cm\leftmargin=1.5cm\labelwidth=1cm} 
\item[(I)\,]
$\mu^-_1(c)$ exists in $(0,1)$ and $\mu^-_1(c)<(c_\ast/c)^{1/3}$. 
\item[(II)]
$\mu^-_1(c)$ exists in $(0,1)$ and $\mu^-_1(c)\ge(c_\ast/c)^{1/3}$, or 
$\mu^-_1(c)$ does not exist in $(0,1)$. 
\end{list} 

Case (I): Recalling Lemma \ref{lem:com_d1+m1-}, we have 
$0<\hat{\delta}^+_1(c)<\hat{\mu}^-_1(c)<(c/c_\ast)^{1/3}$. In addition, 
by virtue of Lemma \ref{lem:com_m1-g1+} and $\hat{\gamma}^+_1(c)\le1$, we see 
$\mu^-_1(c)<\hat{\gamma}^+_1(c)$, so that $\hat{\Psi}''(s;c)>0$ for $s\in(0,\mu^-_1(c)]$. 
Since $\hat{\Psi}'(0;c)=\hat{\alpha}(c)<0$ by Proposition \ref{thm:exit_a} (iii), we are led to
\[
-\psi_-=\hat{\Psi}(\hat{\delta}^+_1(c),c)-\hat{\Psi}(0;c)
=\int_0^{\hat{\delta}^+_1(c)}\hat{\Psi}'(s;c)\,ds
\ge\hat{\delta}^+_1(c)\hat{\alpha}(c)
>(c_\ast/c)^{1/3}\hat{\alpha}(c). 
\]
Thus we get $\hat{\alpha}(c)\le-(\psi_-/c_\ast^{1/3} )c^{1/3}$. 

Case (II): Using \eqref{Psi1} and integrating it by parts, we obtain
\begin{align*} 
&\hat{\Psi}'((c_\ast/c)^{1/3};c) \\
&=\hat{\alpha}(c)+c\int_0^{(c_\ast/c)^{1/3}}((c_\ast/c)^{1/3}-\tilde{s})\sin\hat{\Psi}(\tilde{s};c)\,d\tilde{s} \\
&=\hat{\alpha}(c)+c_\ast^{2/3}c^{1/3}\int_0^1(1-\xi)\sin\hat{\Psi}((c_\ast/c)^{1/3}\xi;c)\,d\xi \\
&=\hat{\alpha}(c)+\frac{c_\ast^{2/3}c^{1/3}\sin\psi}2
+\frac{c_\ast}2\int_0^1(1-\xi)^2\hat{\Psi}'((c_\ast/c)^{1/3}\xi;c)\cos\hat{\Psi}((c_\ast/c)^{1/3}\xi;c)\,d\xi. 
\end{align*} 
Note that it follows from Lemma \ref{lem:com_m1-g1+} that $\hat{\gamma}^+_1(c)=1$ 
if $\mu^-_1(c)$ does not exist in $(0,1)$. 
Then, since $\hat{\alpha}(c)\le\hat{\Psi}'((c_\ast/c)^{1/3}\xi;c)\le0$ for $\xi\in[0,1]$, we see 
\begin{align*} 
0\ge\hat{\Psi}'((c_\ast/c)^{1/3};c)
\ge\hat{\alpha}(c)+\frac{c_\ast^{2/3}c^{1/3}\sin\psi_-}2+\frac{c_\ast} 6\hat{\alpha}(c),
\end{align*} 
so that $\hat{\alpha}(c)\le-\bigl\{c_\ast^{2/3}\sin\psi_-/2(1+c_\ast/6)\bigr\}c^{1/3}$. 

Consequently, setting 
\[
m_{\hat{\alpha}} 
:=\min\left\{\frac{\psi_-}{c_\ast^{1/3}} ,\frac{c_\ast^{2/3}\sin\psi_-}{2(1+c_\ast/6)}\right\} ,
\]
we are led to the upper bound. 
\end{proof} 

\begin{lemma}\label{lem:c-rate_d1+} 
For any $c\ge c_\ast$,  
\begin{equation}\label{c-rate_d1+} 
r_{\hat{\delta}^+_1}c^{-1/3}\le\delta^+_1(c)\le R_{\hat{\delta}^+_1}c^{-1/3} ,
\end{equation} 
where
\[
r_{\hat{\delta}^+_1}:=\frac{\psi_-}{M_{\hat{\alpha}} }, \quad 
R_{\hat{\delta}^+_1}:=\frac{2\psi_-}{m_{\hat{\alpha}} }. 
\]
\end{lemma} 

\begin{proof} 
First, let us derive the lower bound. 
By means of \eqref{Psi} with $s=\hat{\delta}^+_1(c)$, the definition of 
$\hat{\delta}^+_1(c)$ and Lemma \ref{lem:c-rate_a}, we see
\begin{align*} 
0=\hat{\Psi}(\hat{\delta}^+_1(c);c)
=&\,\psi_-+\hat{\alpha}(c)\hat{\delta}^+_1(c)
+\frac{c}2\int_0^{\hat{\delta}^+_1(c)}(\hat{\delta}^+_1(c)-\tilde{s})^2\sin\hat{\Psi}(\tilde{s};c)\,d\tilde{s} \\
\ge&\,\psi_--M_{\hat{\alpha}}c^{1/3}\hat{\delta}^+_1(c). 
\end{align*} 
This implies $\hat{\delta}^+_1(c)\ge(\psi_-/M_{\hat{\alpha}} )c^{-1/3}$. 

Let us derive the upper bound. Using \eqref{Psi} again and multiplying it by $1/s^2$, 
we obtain
\[
\frac{\hat{\Psi}(s;c)}{s^2} 
=\frac{\psi_-}{s^2}+\frac{\hat{\alpha}(c)}{s} 
+\frac{c}2\int_0^s\left(1-\frac{\tilde{s}}{s}\right)^2\sin\hat{\Psi}(\tilde{s};c)\,d\tilde{s}. 
\]
Differentiating with respect to $s$, we have
\[
\frac{\hat{\Psi}'(s;c)}{s^2}-\frac{2\hat{\Psi}(s;c)}{s^3} 
=-\frac{2\psi_-}{s^3}-\frac{\hat{\alpha}(c)}{s^2} 
+c\int_0^s\left(1-\frac{\tilde{s}}{s}\right)\frac{\tilde{s}}{s^2}\sin\hat{\Psi}(\tilde{s};c)\,d\tilde{s}. 
\]
Take $s=\hat{\delta}^+_1(c)$. Since $\hat{\Psi}'(\hat{\delta}^+_1(c);c)<0$ by 
Lemma \ref{lem:com_d1+m1-} and $\hat{\Psi}(s;c)>0$ for $s\in[0,\hat{\delta}^+_1(c))$, 
it follows from Lemma \ref{lem:c-rate_a} that 
\[
0>\frac{\hat{\Psi}'(\hat{\delta}^+_1(c);c)}{(\hat{\delta}^+_1(c))^2} 
\ge-\frac{2\psi_-}{(\hat{\delta}^+_1(c))^3}-\frac{\hat{\alpha}(c)}{(\hat{\delta}^+_1(c))^2} 
\ge-\frac{2\psi_-}{(\hat{\delta}^+_1(c))^3}+\frac{m_{\hat{\alpha}}c^{1/3}}{(\hat{\delta}^+_1(c))^2}. 
\]
As a result, we are led to $\hat{\delta}^+_1(c)\le(2\psi_-/m_{\hat{\alpha}} )c^{-1/3}$. 
\end{proof} 

In the rest of this section, we derive the estimates of $\hat{\Psi}$ and its derivatives  
at zero points, inductively. 

\begin{lemma}\label{lem:c-rate_Psi2_d1+} 
There exist $M_{2,\hat{\delta}^+_1}, m_{2,\hat{\delta}^+_1}>0$ such that 
\[
m_{2,\hat{\delta}^+_1}c^{2/3}\le\hat{\Psi}''(\hat{\delta}^+_1(c);c)\le M_{2,\hat{\delta}^+_1}c^{2/3}
\,\,\ \mbox{for}\,\ c\ge c_\ast. 
\]
\end{lemma} 

\begin{proof} 
First, we derive the lower bound. 
By means of $\hat{\Psi}(0;c)=\psi_-$, $\hat{\Psi}(\hat{\delta}^+_1(c);c)=0$ and the continuity of 
$\hat{\Psi}$ with respect to $s$, we can find $\hat{s}_{\frac{\psi_-}2}(c)\in
(0,\hat{\delta}^+_1(c))$ satisfying $\hat{\Psi}(\hat{s}_{\frac{\psi_-}2}(c);c)=\psi_-/2$. 
Then, using \eqref{Psi} with $s=\hat{s}_{\frac{\psi_-}2}(c)$ and 
Lemma \ref{lem:c-rate_a}, we obtain
\begin{align*} 
\frac{\psi_-}2
=&\,\hat{\Psi}(\hat{s}_{\frac{\psi_-}2}(c);c) \\
=&\,\psi_-+\hat{\alpha}(c)\,\hat{s}_{\frac{\psi_-}2}(c)
+\frac{c}2\int_0^{\hat{s}_{\frac{\psi_-}2}(c)}\left(\hat{s}_{\frac{\psi_-}2}(c)-\tilde{s}\right)^2
\sin\hat{\Psi}(\tilde{s};c)\,d\tilde{s} \\[0.15cm]
\ge&\,\psi_--M_{\hat{\alpha}}c^{1/3}\,\hat{s}_{\frac{\psi_-}2}(c). 
\end{align*} 
Thus the lower bound of $\hat{s}_{\frac{\psi_-}2}(c)$ is given by 
\[
\hat{s}_{\frac{\psi_-}2}(c)\ge\frac{\psi_-}{2M_{\hat{\alpha}} }c^{-1/3}. 
\]
By virtue of this fact, \eqref{Psi2} with $s=\hat{s}_{\frac{\psi_-}2}(c)$ 
and $\hat{\Psi}'(s;c)<0$ for $s\in[0,\hat{\delta}^+_1(c)]$, we see
\[
\hat{\Psi}''(\hat{s}_{\frac{\psi_-}2}(c);c)
=c\int_0^{\hat{s}_{\frac{\psi_-}2}(c)}\sin\hat{\Psi}(\tilde{s};c)\,d\tilde{s}
\ge c\,\hat{s}_{\frac{\psi_-}2}(c)\sin\frac{\psi_-}2
\ge \frac{\psi_-\sin(\psi_-/2)}{2M_{\hat{\alpha}} }c^{2/3}. 
\]
Since $\hat{\Psi}^{(3)}(s;c)=c\sin\hat{\Psi}(s;c)>0$ for 
$s\in[0,\hat{\delta}^+_1(c))$, we are led to
\[
\hat{\Psi}''(\hat{\delta}^+_1(c);c)>\hat{\Psi}''(\hat{s}_{\frac{\psi_-}2}(c);c)
\ge\frac{\psi_-\sin(\psi_-/2)}{2M_{\hat{\alpha}} }c^{2/3}. 
\]

Let us derive the upper bound. Using \eqref{Psi2} with $s=\hat{\delta}^+_1(c)$, 
we have
\[
\hat{\Psi}''(\hat{\delta}^+_1(c);c) 
=c\int_0^{\hat{\delta}^+_1(c)}\sin\hat{\Psi}(\tilde{s};c)\,d\tilde{s}
\le c\,\hat{\delta}^+_1(c)\sin\psi_-\le M_{\hat{\delta}^+_1}c^{2/3}\sin\psi_-. 
\]
This gives the upper bound. 
\end{proof} 

\begin{lemma}\label{lem:c-rate_Psi1_d1+} 
There exist $M_{1,\hat{\delta}^+_1} ,m_{1,\hat{\delta}^+_1}>0$ such that 
\[ 
-M_{1,\hat{\delta}^+_1}c^{1/3}\le\hat{\Psi}'(\hat{\delta}^+_1(c);c) \le-m_{1,\hat{\delta}^+_1}c^{1/3} 
\,\,\ \mbox{for}\,\ c\ge c_\ast. 
\]
\end{lemma} 

\begin{proof} 
By virtue of $\hat{\Psi}''(s;c)>0$ for $s\in(0,\hat{\delta}^+_1(c)]$ and 
Lemma \ref{lem:c-rate_a}, we get
\[
\hat{\Psi}'(\hat{\delta}^+_1(c);c)>\hat{\Psi}'(0;c)=\hat{\alpha}(c)
\ge-M_{\hat{\alpha}}c^{1/3}, 
\]
which implies the lower bound. 

Let us derive the upper bound. Using Lemma \ref{lem:c-rate_Psi2_d1+} and $\hat{\Psi}(s;c)\ge-\psi_-$ 
for $s\in[0,1]$ by Corollary \ref{cor:bound_Psi0-2}, we obtain 
\begin{align*}
\hat{\Psi}''(s;c)
=&\,\hat{\Psi}''(\hat{\delta}^+_1(c);c)+c\int_{\hat{\delta}^+_1(c)}^{s}\sin\hat{\Psi}(\tilde{s};c)\,d\tilde{s} \\
\ge&\,m_{2,\hat{\delta}^+_1}c^{2/3}-c(s-\hat{\delta}^+_1(c))\sin\psi_- \\
=&\,c\sin\psi_-\left(\frac{m_{2,\hat{\delta}^+_1}}{\sin\psi_-}c^{-1/3}-(s-\hat{\delta}^+_1(c))\right)
\end{align*}
for $s\in(\hat{\delta}^+_1(c),1]$. 
This implies that if $s\in[\hat{\delta}^+_1(c),
\hat{\delta}^+_1(c)+\bigl(m_{2,\hat{\delta}^+_1}/(2\sin\psi_-)\bigr)c^{-1/3}]$, then 
\begin{equation}
\label{Psi2_lower_d1+} 
\hat{\Psi}''(s;c)\ge\frac{m_{2,\hat{\delta}^+_1}}2c^{2/3}>0. 
\end{equation} 
Thus, by virtue of this fact, the definition of $\hat{\gamma}^+_1(c)$ and 
Lemma \ref{lem:com_d1+g1+}, we have 
\begin{equation}\label{d1+g1+} 
\hat{\delta}^+_1(c)+\bigl(m_{2,\hat{\delta}^+_1}/(2\sin\psi_-)\bigr)c^{-1/3} 
<\hat{\gamma}^+_1(c)\le1.
\end{equation}
Set $\hat{\mu}^-_{\rm min}(c)=\min\{1,\hat{\mu}^-_1(c)\}$. By means of Lemma \ref{lem:com_d1+m1-} and 
Lemma \ref{lem:com_m1-g1+}, we see $\hat{\delta}^+_1(c)<\hat{\mu}^-_1(c)<\hat{\gamma}^+_1(c)$ 
or $\hat{\delta}^+_1(c)<\hat{\gamma}^+_1(c)=1$. 
Here, we divide the following two cases:
\begin{list}{}{\topsep=0.15cm\itemsep=0cm\leftmargin=1.5cm\labelwidth=1cm} 
\item[(I)\,]
$\hat{\mu}^-_{\rm min}(c)\ge\hat{\delta}^+_1(c)+\bigl(m_{2,\hat{\delta}^+_1}/(4\sin\psi_-)\bigr)c^{-1/3}$. 
\item[(II)]
$\hat{\mu}^-_{\rm min}(c)<\hat{\delta}^+_1(c)+\bigl(m_{2,\hat{\delta}^+_1}/(4\sin\psi_-)\bigr)c^{-1/3}$. 
\end{list}

Case (I): It follows from $\hat{\Psi}''(s;c) > 0$ for $s \in [\hat{\delta}_1^+(c), \hat{\mu}_{\rm min}^-(c))$ 
and \eqref{Psi2_lower_d1+} that
\begin{align*} 
\hat{\Psi}'(\hat{\delta}^+_1(c);c)
=&\,\hat{\Psi}'(\hat{\mu}^-_{\rm min}(c);c)
-\int_{\hat{\delta}^+_1(c)}^{\hat{\mu}^-_{\rm min}(c)}\hat{\Psi}''(s;c)\,ds \\
\le&\,-\int_{\hat{\delta}^+_1(c)}^{\hat{\mu}^-_{\rm min}(c)}\hat{\Psi}''(s;c)\,ds \\
\le&\,\int_{\hat{\delta}^+_1(c)}^{\hat{\delta}^+_1(c)+(m_{2,\hat{\delta}^+_1}/(4\sin\psi_-))c^{-1/3}} 
(-\hat{\Psi}''(s;c))\,ds
\le-\frac{m_{2,\hat{\delta}^+_1}^2}{8\sin\psi_-}c^{1/3}.
\end{align*}

Case (II): By \eqref{d1+g1+}, we have 
$\hat{\mu}^-_{\rm min}(c)= \hat{\mu}^-_1(c)$. 
From $\hat{\Psi}''(s;c) > 0$ for $s \in (\hat{\mu}_1^-(c), \hat{\gamma}_1^+(c)]$ and \eqref{Psi2_lower_d1+}, 
we are led to 
\begin{align*} 
\hat{\Psi}'(\hat{\gamma}^+_1(c);c)
=&\,\int_{\hat{\mu}^-_1(c)}^{\hat{\gamma}^+_1(c)}\hat{\Psi}''(s;c)\,ds \\
\ge&\,\int_{\hat{\delta}^+_1(c)+(m_{2,\hat{\delta}^+_1}/(4\sin\psi_-))c^{-1/3}}
^{\hat{\delta}^+_1(c)+(m_{2,\hat{\delta}^+_1}/(2\sin\psi_-))c^{-1/3}} 
\hat{\Psi}''(s;c)\,ds
\ge\frac{m_{2,\hat{\delta}^+_1}^2}{8\sin\psi_-}c^{1/3}. 
\end{align*}
By means of Lemma \ref{lem:com_d1+g1+} and Lemma \ref{lem:com_Psi1-2}(i), we obtain 
\[
-\hat{\Psi}'(\hat{\delta}^+_1(c);c)\ge\hat{\Psi}'(\hat{\gamma}^+_1(c);c)
\ge\frac{m_{2,\hat{\delta}^+_1}^2}{8\sin\psi_-}c^{1/3} ,
\]
which gives the upper bound of $\hat{\Psi}'(\hat{\delta}^+_1(c);c)$. 

By the estimate for each case, we can choose
\[
m_{1,\hat{\delta}^+_1}:=\frac{m_{2,\hat{\delta}^+_1}^2}{8\sin\psi_-}
\]
as a constant of the upper bound. 
\end{proof} 

\begin{rem}\label{rem:order-est3}
The estimates of $\hat{\delta}_1^+(c)$, $\hat{\Psi}(\hat{\delta}^+_1(c);c)$ and 
$\hat{\Psi}'(\hat{\delta}^+_1(c);c)$ 
as in Lemma \ref{lem:c-rate_d1+}--\ref{lem:c-rate_Psi1_d1+} hold for $c>c_\ast$, where $c_\ast$ 
is given by \eqref{def_c_ast}. 
The following estimates of $\hat{\Psi}$ and its derivatives at the other zero points can be proved 
for more restricted $c>c_\ast$. 
For a better understanding of induction, we will give the remarks on the key points of the proofs 
when necessary.  
\end{rem}

\begin{lemma}\label{lem:exist_m1-} 
There exists $c_{\hat{\mu}^-_1},\,r_{\hat{\mu}^-_1},\,R_{\hat{\mu}^-_1}>0$ such that 
$\hat{\mu}^-_1(c)$ exists in $(0,1)$ for any $c>c_{\hat{\mu}^-_1}$ and fulfills
\begin{equation}
r_{\hat{\mu}^-_1}c^{-1/3}\le\hat{\mu}^-_1(c)-\hat{\delta}^+_1(c)\le R_{\hat{\mu}^-_1}c^{-1/3}.
\label{c-rate_m1-d1+}
\end{equation}
\end{lemma} 

\begin{proof} 
As the first consideration, we derive the estimate of $\hat{\Psi}'$ in the neighborhood 
of $s=\hat{\delta}^+_1(c)$. Using Corollary \ref{cor:bound_Psi0-2}, 
Lemma \ref{lem:c-rate_Psi2_d1+} and Lemma \ref{lem:c-rate_Psi1_d1+}, we see 
that for $s\in(\hat{\delta}^+_1(c),1]$
\begin{equation}\label{est_Psi1} 
\begin{aligned}
\hat{\Psi}'(s;c)
=&\,\hat{\Psi}'(\hat{\delta}^+_1(c);c)+\int_{\hat{\delta}^+_1(c)}^s\hat{\Psi}''(\tilde{s};c)\,d\tilde{s} \\
\le&\,-m_{1,\hat{\delta}^+_1}c^{1/3}+(s -\hat{\delta}^+_1(c))M_{2,\hat{\delta}^+_1}c^{2/3} \\
=&\,M_{2,\hat{\delta}^+_1}c^{2/3}\left(
-\frac{m_{1,\hat{\delta}^+_1}}{M_{2,\hat{\delta}^+_1}}c^{-1/3}+s -\hat{\delta}^+_1(c)
\right). 
\end{aligned} 
\end{equation}
Set $r_{\hat{\mu}^-_1}:=m_{1,\hat{\delta}^+_1}/M_{2,\hat{\delta}^+_1}$. 
Then, for $s\in[\hat{\delta}^+_1(c),
\min\{1,\hat{\delta}^+_1(c)+(r_{\hat{\mu}^-_1}/2)c^{-1/3}\}]$ 
\[
\hat{\Psi}'(s;c)\le-\frac{m_{1,\hat{\delta}^+_1}}2c^{1/3}<0.
\]
Thus it follows from this fact and Lemma \ref{lem:com_d1+m1-} that if $\hat{\mu}^-_1(c)$ 
exists in $(0,1)$, then $\hat{\delta}^+_1(c)+(r_{\hat{\mu}^-_1}/2)c^{-1/3}<\hat{\mu}^-_1(c)$. 

Let us prove that $\hat{\mu}^-_1(c)$ exists in $(0,1)$ for sufficiently large $c>c_\ast$. 
From Lemma \ref{lem:c-rate_d1+}, we obtain  $\hat{\delta}^+_1(c)+(r_{\hat{\mu}^-_1}/2)c^{-1/3}<1$ 
for $c > (R_{\hat{\delta}_1^+} + r_{\hat{\mu}_1^-}/2)^3$. 
Set $\hat{\mu}^-_{\rm min}(c):=\min\{1,\hat{\mu}^-_1(c)\}$. Since $\hat{\Psi}'(s;c)<0$ 
for $s\in[\hat{\delta}^+_1(c)+(r_{\hat{\mu}^-_1}/2)c^{-1/3},\hat{\mu}^-_{\rm min}(c))$, 
we see that for $s\in[\hat{\delta}^+_1(c)+(r_{\hat{\mu}^-_1}/2)c^{-1/3} ,\hat{\mu}^-_{\rm min}(c)]$
\begin{equation}\label{Psi_upper}
\begin{aligned} 
\hat{\Psi}(s;c)
\le&\,\hat{\Psi}\left(\hat{\delta}^+_1(c)+\frac{r_{\hat{\mu}^-_1}}2c^{-1/3};c\right) \\
=&\,\int_{\hat{\delta}^+_1(c)}^{\hat{\delta}^+_1(c)+(r_{\hat{\mu}^-_1}/2)c^{-1/3}}\hat{\Psi}'(s;c)\,ds
\le-\frac{r_{\hat{\mu}^-_1} m_{1,\hat{\delta}^+_1}}4. 
\end{aligned} 
\end{equation} 
Using \eqref{int_eq_Psi} with $s_1=\delta^+_1(c)$ and $s_2=\hat{\mu}^-_{\rm min}(c)$ and 
taking account of $\hat{\Psi}''(\hat{\mu}^-_{\rm min}(c);c)\ge0$ and 
$\hat{\Psi}(\hat{\mu}^-_{\rm min}(c);c)<0$, 
we have 
\begin{align*} 
\frac{\bigl(\hat{\Psi}'(\delta^+_1(c);c)\bigr)^2}2
\ge&\,c\int_{\delta^+_1(c)}^{\hat{\mu}^-_{\rm min}(c)}\hat{\Psi}(s;c)\sin\hat{\Psi}(s;c)\,ds \\
\ge&\,c\int_{\hat{\delta}^+_1(c)+(r_{\hat{\mu}^-_1}/2)c^{-1/3}}^{\hat{\mu}^-_{\rm min}(c)} 
\hat{\Psi}(s;c)\sin\hat{\Psi}(s;c)\,ds. 
\end{align*} 
Thus it follows from Lemma \ref{lem:c-rate_Psi1_d1+} and \eqref{Psi_upper} that 
\[
\frac{M_{1,\hat{\delta}^+_1}^2}2
\ge c^{1/3}\left(\hat{\mu}^-_{\rm min}(c)-\hat{\delta}^+_1(c)-\frac{r_{\hat{\mu}^-_1}}2c^{-1/3}\right)
\frac{r_{\hat{\mu}^-_1}m_{1,\hat{\delta}^+_1}}4\sin\frac{r_{\hat{\mu}^-_1} m_{1,\hat{\delta}^+_1}}4. 
\]
This gives 
\begin{equation}\label{est_m1-} 
\hat{\mu}^-_{\rm min}(c)\le \hat{\delta}^+_1(c)+R_{\hat{\mu}^-_1}c^{-1/3} 
\le(R_{\hat{\delta}^+_1}+R_{\hat{\mu}^-_1})c^{-1/3}.
\end{equation} 
where
\[
R_{\hat{\mu}^-_1}
:=\frac{2M_{1,\hat{\delta}^+_1}^2} 
{r_{\hat{\mu}^-_1} m_{1,\hat{\delta}^+_1}\sin\dfrac{r_{\hat{\mu}^-_1} m_{1,\hat{\delta}^+_1}}4} 
+\frac{r_{\hat{\mu}^-_1}}2. 
\]
Define $c_{\hat{\mu}_1^-}$ as $c_{\hat{\mu}_1^-}:=(R_{\hat{\delta}_1^+}+R_{\hat{\mu}_1^-})^3$. 
Note that $c_{\hat{\mu}_1^-}>(R_{\hat{\delta}_1^+} + r_{\hat{\mu}_1^-}/2)^3$. Then, 
for $c>c_{\hat{\mu}_1^-}$, 
we obtain $\hat{\mu}^-_{\rm min}(c)<1$, which implies that $\hat{\mu}^-_1(c)$ exists in $(0,1)$. 

The estimate \eqref{c-rate_m1-d1+} follows from \eqref{est_Psi1} with $s=\hat{\mu}^-_1(c)$ 
and \eqref{est_m1-}. 
\end{proof} 

\begin{lemma}\label{lem:c-rate_Psi_m1-} 
There exist 
$M_{0,\hat{\mu}^-_1},m_{0,\hat{\mu}^-_1}>0$ such that 
\[
-M_{0,\hat{\mu}^-_1}\le\hat{\Psi}(\hat{\mu}^-_1(c);c)\le-m_{0,\hat{\mu}^-_1} 
\,\,\ \mbox{for}\,\  c\in J_{\hat{\mu}^-_1}. 
\]
\end{lemma} 
\begin{proof} 
According to Corollary \ref{cor:bound_Psi0-2}, $\hat{\Psi}(s;c)\ge-\psi_-$ for $s\in[0,1]$. 
This gives the lower bound as $M_{0,\hat{\mu}^-_1}:=\psi_-$. 
On the other hand, the upper bound with 
$m_{0,\hat{\mu}_1^-} = r_{\hat{\mu}_1^-}m_{1,\hat{\delta}_1^+}/4$ follows from \eqref{Psi_upper} 
since $\hat{\mu}^-_{\rm min}(c) = \hat{\mu}_1^-(c)$ if $c \in J_{\hat{\mu}_1^-}$. 
\end{proof}

\begin{lemma}\label{lem:c-rate_Psi2_m1-}
There exist $M_{2,\hat{\mu}^-_1},m_{2,\hat{\mu}^-_1}>0$ such that 
\begin{equation}\label{c-rate_Psi2_m1-}
m_{2,\hat{\mu}^-_1}c^{2/3}\le\hat{\Psi}''(\hat{\mu}^-_1(c);c)\le M_{2,\hat{\mu}^-_1}c^{2/3} 
\,\,\ \mbox{for}\,\  c\in J_{\hat{\delta}^-_1}.
\end{equation}
\end{lemma}

\begin{proof} 
First, we show the upper bound. Applying Lemma \ref{lem:com_Psi1-2}(ii) and 
Lemma \ref{lem:c-rate_Psi2_d1+}, we have 
\[
\hat{\Psi}''(\hat{\mu}^-_1(c);c)<\hat{\Psi}''(\hat{\delta}^+_1(c);c)\le M_{2,\hat{\delta}^+_1}c^{2/3}.
\]
This gives the upper bound. 

Let us derive the lower bound. It follows from Lemma \ref{lem:c-rate_Psi_m1-}, 
Corollary \ref{cor:bound_Psi0-2} and Lemma \ref{lem:c-rate_a} that for $s\in(\hat{\mu}^-_1(c),1]$
\begin{equation}\label{est_Psi}
\begin{aligned}
\hat{\Psi}(s;c)
=&\,\hat{\Psi}(\hat{\mu}^-_1(c);c)+\int_{\hat{\mu}^-_1(c)}^{s}\hat{\Psi}'(\tilde{s};c)\,d\tilde{s} \\
\le&\,-m_{0,\hat{\mu}^-_1}+(s-\hat{\mu}^-_1(c))M_{\hat{\alpha}}c^{1/3} \\
\le&\,M_{\hat{\alpha}}c^{1/3}\left(-\frac{m_{0,\hat{\mu}^-_1}}{M_{\hat{\alpha}}}c^{-1/3}
+s-\hat{\mu}^-_1(c)\right). 
\end{aligned}
\end{equation}
Then, for $s\in[\hat{\mu}^-_1(c),
\min\bigl\{1,\hat{\mu}^-_1(c)+\bigl(m_{0,\hat{\mu}^-_1}/(2M_{\hat{\alpha}})\bigr)c^{-1/3}\bigr\}]$
\begin{equation}\label{Psi_upper_m1-} 
\hat{\Psi}(s;c)\le-\dfrac{m_{0,\hat{\mu}^-_1}}2<0.
\end{equation}
By means of $c\in J_{\hat{\delta}^-_1}$, we see that $\hat{\delta}^-_1(c)$ 
exists in the interval $(0,1)$, so that Lemma \ref{lem:com_d1+m1-} and \eqref{Psi_upper_m1-} 
imply that 
\begin{equation}\label{m1-d1-}
\hat{\mu}^-_1(c)+\frac{m_{0,\hat{\mu}^-_1}}{2M_{\hat{\alpha}}}c^{-1/3}<\hat{\delta}^-_1(c). 
\end{equation}
With the help of Lemma \ref{lem:com_m1-g1+} and Lemma \ref{lem:com_g1+d1-}, we have 
$\hat{\mu}^-_1(c)<\hat{\gamma}^+_1(c)<\hat{\delta}^-_1(c)$. 
Here, we divide the following two cases:
\begin{list}{}{\topsep=0.15cm\itemsep=0cm\leftmargin=1.5cm\labelwidth=1cm} 
\item[(I)\,]$\hat{\gamma}^+_1(c)\ge\hat{\mu}^-_1(c)+\bigl(m_{0,\hat{\mu}^-_1}/(4M_{\hat{\alpha}})\bigr)c^{-1/3}$. 
\item[(II)]$\hat{\gamma}^+_1(c)<\hat{\mu}^-_1(c)+\bigl(m_{0,\hat{\mu}^-_1}/(4M_{\hat{\alpha}})\bigr)c^{-1/3}$. 
\end{list}

Case (I): From the equation $\hat{\Psi}^{(3)}=c\sin\hat{\Psi}$, $\hat{\Psi}(s;c)<0$ for 
$s\in[\hat{\mu}^-_1(c),\hat{\gamma}^+_1(c)]$ and \eqref{Psi_upper_m1-}, 
we are led to
\begin{align*}
\hat{\Psi}''(\hat{\mu}^-_1(c);c)
=&\,-c\int_{\hat{\mu}^-_1(c)}^{\hat{\gamma}^+_1(c)}\sin\hat{\Psi}(s;c)\,ds \\
\ge&\,c\int_{\hat{\mu}^-_1(c)}^{\hat{\mu}^-_1(c)+(m_{0,\hat{\mu}^-_1}/(4M_{\hat{\alpha}}))c^{-1/3}}
\sin(-\hat{\Psi}(s;c))\,ds
\ge\frac{m_{0,\hat{\mu}^-_1}\sin(m_{0,\hat{\mu}^-_1}/2)}{4M_{\hat{\alpha}}}c^{2/3}.
\end{align*}

Case (II): From the equation $\hat{\Psi}^{(3)}=c\sin\hat{\Psi}$, $\hat{\Psi}(s;c)<0$ 
for $s \in [\hat{\gamma}^+_1(c),\hat{\delta}^-_1(c))$,  \eqref{Psi_upper_m1-} 
and \eqref{m1-d1-}, we obtain 
\begin{align*}
\hat{\Psi}''(\hat{\delta}^-_1(c);c)
=&\,c\int_{\hat{\gamma}^+_1(c)}^{\hat{\delta}^-_1(c)}\sin\hat{\Psi}(s;c)\,ds \\
\le&\,c\int_{\hat{\mu}^-_1(c)+(m_{0,\hat{\mu}^-_1}/(4M_{\hat{\alpha}}))c^{-1/3}}
^{\hat{\mu}^-_1(c)+(m_{0,\hat{\mu}^-_1}/(2M_{\hat{\alpha}}))c^{-1/3}}
\sin\hat{\Psi}(s;c)\,ds
\le-\frac{m_{0,\hat{\mu}^-_1}\sin(m_{0,\hat{\mu}^-_1}/2)}{4M_{\hat{\alpha}}}c^{2/3}.
\end{align*}
Then Lemma \ref{lem:com_Psi1-2}(ii) implies that
\[ 
\hat{\Psi}''(\hat{\mu}^-_1(c);c)\ge-\hat{\Psi}''(\hat{\delta}^-_1(c);c)
\ge\frac{m_{0,\hat{\mu}^-_1}\sin(m_{0,\hat{\mu}^-_1}/2)}{4M_{\hat{\alpha}}}c^{2/3}.
\]

Consequently, we have the lower bound in both cases. 
\end{proof} 

\begin{rem}\label{rem:order-est}
In the case (II) of the proof of Lemma \ref{lem:c-rate_Psi1_d1+} and 
Lemma \ref{lem:c-rate_Psi2_m1-}, we used the inequalities
$-\hat{\Psi}'(\hat{\delta}_1^+(c),c)\ge\hat{\Psi}'(\hat{\gamma}_1^+(c),c)$ and 
$\hat{\Psi}''(\hat{\mu}_1^-(c),c)\ge-\hat{\Psi}''(\hat{\delta}_1^-(c),c)$.
That is, in order to obtain the estimates of $\hat{\Psi}'(\hat{\delta}_1^+(c),c)$ and 
$\hat{\Psi}''(\hat{\mu}_1^-(c),c)$,  we need the existence of $\hat{\gamma}_1^+(c)$ and 
$\hat{\delta}_1^-(c)$ in the interval $(0,1]$, respectively. Note that, according to 
Proposition \ref{prop:com_zero-point}, the inclusion  
\[
(0,\infty) = J_{\hat{\delta}^+_1} \supset J_{\hat{\mu}_1^-} \supset 
J_{\hat{\gamma}_1^+} = J_{\hat{\delta}_1^-} \supset J_{\hat{\mu}_1^+} \supset 
J_{\hat{\gamma}_1^-} = J_{\hat{\delta}_2^+} \supset \cdots
\]
holds and $\hat{\gamma}_1^+(c)$ (resp. $\hat{\delta}_1^-(c)$) is the second zero point 
from $\hat{\delta}_1^+(c)$ (resp. $\hat{\mu}_1^-(c)$) in a larger direction for the order
\[
\hat{\delta}_1^+(c) < \hat{\mu}_1^-(c) < \hat{\gamma}_1^+(c) < \hat{\delta}_1^-(c) 
< \hat{\mu}_1^+(c) < \hat{\gamma}_1^-(c) < \hat{\delta}_2^+(c) < \cdots.
\]

In the former case, this inclusion and the boundary condition $\hat{\Psi}''(1;c)=0$ give 
$J_{\hat{\gamma}^+_1}\cup\{c>0\,|\,\hat{\gamma}^+_1(c)=1\}=(0,\infty)$, so that 
the existence of $\hat{\gamma}^+_1(c)$ in the interval $(0,1]$ is guaranteed for any $c>0$. 
In general, in order to derive the estimates of the derivatives of $\hat{\Psi}$ at 
$s=\delta^\pm_j(c)$, we need the existence of $\hat{\gamma}^\pm_j(c)$ in the interval $(0,1]$. 
Since Proposition \ref{prop:com_zero-point}(i) and the boundary condition $\hat{\Psi}''(1;c)=0$ imply  
$J_{\hat{\gamma}^\pm_j}\cup\{c>0\,|\,\hat{\gamma}^\pm_j(c)=1\}=J_{\hat{\delta}^\pm_j}$, 
the existence of $\hat{\gamma}^\pm_j(c)$ in the interval $(0,1]$ is guaranteed for 
$c\in J_{\hat{\delta}^\pm_j}$.

On the other hand, in the latter case, we do not know whether 
$J_{\hat{\mu}_1^-} \subset J_{\hat{\delta}_1^-} \cup \{c>0\,|\,\hat{\delta}_1^-(c) = 1\}$ 
or not. 
Therefore, the estimate in Lemma \ref{lem:c-rate_Psi2_m1-} can be proved only for 
$c \in J_{\hat{\delta}^-_1}$. 

\end{rem}

\begin{lemma}\label{lem:exist_g1+} 
There exist $c_{\hat{\gamma}^+_1}>c_{\hat{\mu}^-_1}$ and $R_{\gamma^+_1},
r_{\hat{\gamma}^+_1}>0$ such that $\hat{\gamma}^+_1(c)$ exists in $(0,1)$ for any 
$c>c_{\hat{\gamma}^+_1}$ and fulfills 
\begin{equation}\label{c-rate_g1+m1-}
r_{\hat{\gamma}^+_1}c^{-1/3}\le\hat{\gamma}^+_1(c)-\hat{\mu}^-_1(c)\le R_{\gamma^+_1}c^{-1/3}.
\end{equation}
\end{lemma} 

\begin{proof} 
Let us show that $\hat{\gamma}^+_1(c)$ exists in $(0,1)$ for sufficiently large $c>c_\ast$. 
Since $\hat{\gamma}^+_1(c) \le 1$ for any $c>0$, it is enough to prove $\hat{\gamma}^+_1(c) < 1$ and \eqref{c-rate_g1+m1-}. 
First, we prove that the estimate \eqref{c-rate_Psi2_m1-} holds for sufficiently large $c>0$. 
Let $c$ satisfy 
\begin{equation}\label{exist_g1+-2} 
c > \left(c_{\hat{\mu}^-_1}^{1/3} + \dfrac{m_{0,\hat{\mu}_1^-}}{4M_{\hat{\alpha}}} \right)^3. 
\end{equation} 
Then $\hat{\mu}^-_1(c)$ exists in $(0,1)$ since $c > c_{\hat{\mu}^-_1}$. 
Note that 
\begin{equation}\label{exist_g1+-3} 
\hat{\mu}^-_1(c) < c_{\hat{\mu}^-_1}^{1/3}c^{-1/3} 
\end{equation} 
due to \eqref{c-rate_d1+}, \eqref{c-rate_m1-d1+} and 
$c_{\hat{\mu}^-_1} = (R_{\hat{\delta}_1^+} + R_{\hat{\mu}_1^-})^3$. 
Therefore, \eqref{exist_g1+-2} leads us to
\begin{equation}\label{exist_g1+-1} 
\hat{\mu}_1^-(c) + \dfrac{m_{0,\hat{\mu}_1^-}}{4M_{\hat{\alpha}}} c^{-1/3} < 1. 
\end{equation} 
Here, we divide the two cases (I) and (II) as in the proof of Lemma \ref{lem:c-rate_Psi2_m1-}. 
For the case (I), we have the estimate \eqref{c-rate_Psi2_m1-} by the same argument 
as the proof of Lemma \ref{lem:c-rate_Psi2_m1-}. For the case (II), we obtain 
$c \in J_{\hat{\gamma}^+_1}$ from \eqref{exist_g1+-1}, so that 
\eqref{Jg1+d1-} implies that $c \in J_{\hat{\delta}_1^-}$. Thus   
the estimate \eqref{c-rate_Psi2_m1-} holds.
As a result, we can apply the estimate \eqref{c-rate_Psi2_m1-} to show that 
for $s\in(\hat{\mu}^-_1(c),1]$
\begin{equation}\label{est_Psi2}
\begin{aligned}
\hat{\Psi}''(s;c)
=&\,\hat{\Psi}''(\hat{\mu}^-_1(c);c)+c\int_{\hat{\mu}^-_1(c)}^{s}\sin\hat{\Psi}(\tilde{s};c)\,d\tilde{s} \\
\ge&\,m_{2,\hat{\mu}^-_1}c^{2/3}-c(s-\hat{\mu}^-_1(c))\sin\psi_- \\
=&\,c\sin\psi_-\left\{\frac{m_{2,\hat{\mu}^-_1}}{\sin\psi_-}c^{-1/3}
-\bigl(s-\hat{\mu}^-_1(c)\bigr)\right\}. 
\end{aligned}
\end{equation}
Set $r_{\hat{\gamma}^+_1}:=m_{2,\hat{\mu}^-_1}/\sin\psi_-$ and retake $c>0$ 
satisfying \eqref{exist_g1+-2} and $(c_{\hat{\mu}^-_1}^{1/3}+r_{\hat{\gamma}^+_1})c^{-1/3}<1$. 
Note that the second property is equivalent to $c>(c_{\hat{\mu}^-_1}^{1/3}+r_{\hat{\gamma}^+_1})^3$. 
Applying \eqref{exist_g1+-3}, we obtain $\hat{\mu}^-_1(c)+r_{\hat{\gamma}^+_1}c^{-1/3}<1$. 
Thus it follows from \eqref{est_Psi2} that for 
$s\in[\hat{\mu}^-_1(c),\hat{\mu}^-_1(c)+(r_{\hat{\gamma}^+_1}/2)c^{-1/3}]$
\[
\hat{\Psi}''(s;c)\ge\frac{m_{2,\hat{\mu}^-_1}}2c^{2/3}>0.
\]
By means of this fact and Lemma \ref{lem:com_m1-g1+}, we are led to 
$\hat{\mu}^-_1(c)+(r_{\hat{\gamma}^+_1}/2)c^{-1/3}<\hat{\gamma}^+_{1}(c)$. 
Since $\hat{\Psi}''(s;c)>0$ for $s\in(0,\hat{\gamma}^+_1(c))$, we see that 
for $s\in[\hat{\mu}^-_1(c)+(r_{\hat{\gamma}^+_1}/2)c^{-1/3},\hat{\gamma}^+_{1}(c)]$
\begin{equation}\label{Psi1_lower}
\begin{aligned}
\hat{\Psi}'(s;c)
\ge&\,\hat{\Psi}'\left(\hat{\mu}^-_1(c)+\frac{r_{\hat{\gamma}^+_1}}2c^{-1/3};c\right) \\
=&\,\int_{\hat{\mu}^-_1(c)}^{\hat{\mu}^-_1(c)+(r_{\hat{\gamma}^+_1}/2)c^{-1/3}}
\hat{\Psi}''(s;c)\,ds
\ge\frac{r_{\hat{\gamma}^+_1}m_{2,\hat{\mu}^-_1}}2c^{1/3}. 
\end{aligned}
\end{equation}
Using \eqref{int_eq_Psi2} with $s_1=\mu^-_1(c)$ and $s_2=\hat{\gamma}^+_{1}(c)$ and 
taking account of $\hat{\Psi}''(\hat{\gamma}^+_{1}(c);c)=0$, we have
\begin{align*}
&c\hat{\Psi}'(\hat{\gamma}^+_{1}(c);c)\sin\hat{\Psi}(\hat{\gamma}^+_{1}(c);c)
+\frac{(\hat{\Psi}''(\mu^-_1(c);c))^2}2 \\
&=c\int_{\mu^-_1(c)}^{\hat{\gamma}^+_{1}(c)}(\hat{\Psi}'(s;c))^2\cos\hat{\Psi}(s;c)\,ds \\
&\ge c\int_{\hat{\mu}^-_1(c)+(r_{\hat{\gamma}^+_1}/2)c^{-1/3}}^{\hat{\gamma}^+_{1}(c)}
(\hat{\Psi}'(s;c))^2\cos\hat{\Psi}(s;c)\,ds.
\end{align*}
From Corollary \ref{cor:bound_Psi0-2}, \eqref{c-rate_Psi2_m1-} and 
\eqref{Psi1_lower}, we are led to 
\[
M_{\hat{\alpha}}\sin\psi_-+\frac{M_{2,\hat{\mu}^-_1}^2}2
\ge c^{1/3}\left(\hat{\gamma}^+_{1}(c)-\hat{\mu}^-_1(c)-\frac{r_{\hat{\gamma}^+_1}}2c^{-1/3}\right)
\frac{r_{\hat{\gamma}^+_1}^2m_{2,\hat{\mu}^-_1}^2}4\cos\psi_-.
\]
This gives 
\begin{equation}\label{est_g1+}
\hat{\gamma}^+_{1}(c)
\le\hat{\mu}^-_1(c)+R_{\hat{\gamma}^+_1}c^{-1/3}
\le(c_{\hat{\mu}^-_1}^{1/3}+R_{\hat{\gamma}^+_1})c^{-1/3},
\end{equation}
where 
\[
R_{\hat{\gamma}^+_1}
:=\frac{4\left\{M_{\hat{\alpha}}\sin\psi_-+\left(M_{2,\hat{\mu}^-_1}^2/2\right)\right\}}
{r_{\hat{\gamma}^+_1}^2m_{2,\hat{\mu}^-_1}^2\cos\psi_-}+\frac{r_{\hat{\gamma}^+_1}}2.
\]
Note that $r_{\hat{\mu}^-_1} < R_{\hat{\mu}^-_1}$ due to
\eqref{est_Psi2} with $s=\hat{\gamma}^+_1(c)$ and the first inequality of 
\eqref{est_g1+}. 
Therefore, if we define $c_{\hat{\gamma}_1^+}$ as 
\[ 
c_{\hat{\gamma}_1^+} 
:= \max \left\{(c_{\hat{\mu}^-_1}^{1/3}+R_{\hat{\mu}^-_1})^3,\left(c_{\hat{\mu}^-_1}^{1/3} + \dfrac{m_{0,\hat{\mu}_1^-}}{4M_{\hat{\alpha}}} \right)^3 \right\}, 
\] 
\eqref{est_Psi2} with $s=\hat{\gamma}^+_1(c)$ and \eqref{est_g1+} imply that 
$\hat{\gamma}^+_1(c)<1$ and \eqref{c-rate_g1+m1-} hold for $c>c_{\hat{\gamma}_1^+}$.
\end{proof}

\begin{rem}\label{rem:order-est2}
Since the estimate \eqref{c-rate_Psi2_m1-} holds for $c \in J_{\hat{\delta}_1^-}$ 
and the existence of $\hat{\delta}_1^-(c)$ in the interval $(0,1)$ is not 
guaranteed in the beginning of the proof, we can not directly apply it. 
Therefore, we have to restrict $c>0$ to satisfying \eqref{exist_g1+-1}. 
This kind of restriction is required in the proof of the existence of the other zero points in $(0,1)$. 
\end{rem}

\begin{lemma}\label{lem:c-rate_Psi1_g1+}
There exist $M_{1,\hat{\gamma}^+_1},m_{1,\hat{\gamma}^+_1}>0$ 
such that 
\[
m_{1,\hat{\gamma}^+_1}c^{1/3}\le\hat{\Psi}'(\hat{\gamma}^+_1(c);c)\le M_{1,\hat{\gamma}^+_1}c^{1/3}
\,\,\ \mbox{for}\,\ c\in J_{\hat{\gamma}^+_1}. 
\] 
\end{lemma} 

\begin{proof} 
The upper bound follows from Corollary \ref{cor:bound_Psi0-2} and Lemma \ref{lem:c-rate_a}. 
For the lower bound, we first obtain the estimate \eqref{c-rate_Psi2_m1-} since 
$c\in J_{\hat{\delta}^-_1}$ by $c\in J_{\hat{\gamma}^+_1}$ and \eqref{Jg1+d1-}. 
Applying a similar argument to the proof of Lemma \ref{lem:exist_g1+},  we have 
the inequality \eqref{Psi1_lower}. Thus the lower bound follows from \eqref{Psi1_lower} 
with $s=\hat{\gamma}_1^+(c)$. 
\end{proof}

A similar argument to the proof of Lemma \ref{lem:c-rate_Psi2_m1-} leads us to 
the following lemma. The assumption $c\in J_{\hat{\mu}^+_1}$ is necessary 
for the same reason as the latter case in Remark \ref{rem:order-est}.

\begin{lemma}\label{lem:c-rate_Psi_g1+}
There exist $M_{0,\hat{\gamma}^+_1},m_{0,\hat{\gamma}^+_1}>0$ 
such that 
\[
-M_{0,\hat{\gamma}^+_1}\le\hat{\Psi}(\hat{\gamma}^+_1(c);c)\le-m_{0,\hat{\gamma}^+_1} 
\,\,\ \mbox{for}\,\ c\in J_{\hat{\mu}^+_1}.
\]
\end{lemma}

Recalling \eqref{Jg1+d1-}, we see that $\hat{\delta}_1^-(c)$ exists in $(0,1)$ for 
$c > c_{\hat{\gamma}_1^+}$, where $c_{\hat{\gamma}_1^+}$ is as in Lemma \ref{lem:exist_g1+}. 
However, in order to obtain the estimate of the distance between $\hat{\gamma}_1^+(c)$ 
and $\hat{\delta}_1^-(c)$, we have to restrict $c$ to being large enough as we mentioned 
in Remark \ref{rem:order-est2}. 

\begin{lemma}\label{lem:exist_d1-} 
There exist $c_{\hat{\delta}^-_1}>c_{\hat{\gamma}^+_1}$ and 
$R_{\hat{\delta}^-_1},r_{\hat{\delta}^-_1}>0$ such that $\hat{\delta}^-_1(c)$ exists in $(0,1)$ 
for any $c>c_{\hat{\delta}^-_1}$ and fulfills 
\[
r_{\hat{\delta}^-_1}c^{-1/3}\le\hat{\delta}^-_1(c)-\hat{\gamma}^+_1(c)\le R_{\delta^+_1}c^{-1/3}.
\]
\end{lemma} 

Let us state the estimates of $\hat{\Psi}'(\hat{\delta}_1^-(c),c)$ and 
$\hat{\Psi}''(\hat{\delta}_1^-(c),c)$. According to Remark \ref{rem:order-est}, 
the existence of $\hat{\gamma}_1^-(c)$ in the interval $(0,1]$, which is required in the proof, 
is guaranteed for $c\in J_{\hat{\delta}^-_1}$. 
However, since we use the estimate of $\hat{\Psi}(\hat{\gamma}^+_1(c),c)$ in 
Lemma \ref{lem:c-rate_Psi_g1+}, the following estimates can be proved only for 
$c \in J_{\hat{\mu}_1^+}$.  

\begin{lemma}\label{lem:c-rate_Psi2_d1-}
There exist $M_{2,\hat{\delta}^-_1},m_{2,\hat{\delta}^-_1}>0$ 
such that for $c\in J_{\hat{\mu}^+_1}$
\[
\left\{\begin{array}{l}
-M_{2,\hat{\delta}^-_1}c^{2/3}\le\hat{\Psi}''(\hat{\delta}^-_1(c);c)\le-m_{2,\hat{\delta}^+_1}c^{2/3}, \\
m_{1,\hat{\delta}^-_1}c^{1/3}\le\hat{\Psi}'(\hat{\delta}^-_1(c);c)
\le M_{1,\hat{\delta}^-_1}c^{1/3}.
\end{array}\right.
\]
\end{lemma} 

By the induction and Proposition \ref{prop:com_zero-point}, we obtain the following proposition. 
The estimates of $\hat{\delta}_1^+(c)$, $\hat{\Psi}'(\hat{\delta}_1^+(c),c)$ and 
$\hat{\Psi}''(\hat{\delta}_1^+(c),c)$ are excluded from this proposition 
since the restriction on $c$ in their estimates is especial as we mentioned 
in Remark \ref{rem:order-est3}. As to their estimates, see 
Lemma \ref{lem:c-rate_d1+}--\ref{lem:c-rate_Psi1_d1+}.

\begin{prop}\label{prop:est_order}
For any $j \in \mathbb{N}$, the following properties hold. 
\begin{list}{}{\topsep=0cm\itemsep=0cm\leftmargin=0cm\itemindent=0.25cm} 
\item[(I-i)]
There exists $c_{\hat{\delta}^+_{j+1}}>0$ and 
$R_{\hat{\delta}^+_{j+1}},r_{\hat{\delta}^+_{j+1}}>0$ such that 
$\hat{\delta}^+_{j+1}(c)$ exists in $(0,1)$ for any $c>c_{\hat{\delta}^+_{j+1}}$ and fulfills
\[
r_{\hat{\delta}^+_{j+1}}c^{-1/3}\le\hat{\delta}^+_{j+1}(c)-\hat{\gamma}^-_j(c)
\le R_{\hat{\delta}^+_{j+1}}c^{-1/3}, 
\]
and there exists $c_{\hat{\delta}^-_j}>0$ and 
$R_{\hat{\delta}^-_j},r_{\hat{\delta}^-_j}>0$ such that $\hat{\delta}^-_j(c)$ exists 
in $(0,1)$ for any $c>c_{\hat{\delta}^-_j}$ and fulfills 
\[
r_{\hat{\delta}^-_j}c^{-1/3}\le\hat{\delta}^-_j(c)-\hat{\gamma}^+_j(c)
\le R_{\hat{\delta}^-_j}c^{-1/3}.
\]
\item[(I-ii)] 
There exist $c_{\hat{\mu}^\mp_j}>0$ and $R_{\hat{\mu}^\mp_j},r_{\hat{\mu}^\mp_j}>0$ such that 
$\hat{\mu}^\mp_j(c)$ exist in $(0,1)$ for any $c>c_{\hat{\mu}^\mp_j}$ and fulfill 
\[
r_{\hat{\mu}^\mp_j}c^{-1/3}\le\hat{\mu}^\mp_j(c)-\hat{\delta}^\pm_j(c)
\le R_{\hat{\mu}^\mp_j}c^{-1/3}.
\]
\item[(I-iii)] 
There exist $c_{\hat{\gamma}^\pm_j}>0$ and 
$R_{\hat{\gamma}^\pm_j},r_{\hat{\gamma}^\pm_j}>0$ such that $\hat{\gamma}^\pm_j(c)$ exist 
in $(0,1)$ for any $c>c_{\hat{\gamma}^\pm_j}$ and fulfill
\[
r_{\hat{\gamma}^\pm_j}c^{-1/3}\le\hat{\gamma}^\pm_j(c)-\hat{\mu}^\mp_j(c)
\le R_{\hat{\gamma}^\pm_j}c^{-1/3}. 
\]
\item[(II-i)]
There exists $M_{2,\hat{\delta}^+_{j+1}},m_{2,\hat{\delta}^+_{j+1}},
M_{1,\hat{\delta}^+_{j+1}},m_{1,\hat{\delta}^+_{j+1}}>0$ such that for 
$c\in J_{\hat{\mu}^-_{j+1}}$
\[
\left\{\begin{array}{l}
m_{2,\hat{\delta}^+_{j+1}}c^{2/3}\le\hat{\Psi}''(\hat{\delta}^+_{j+1}(c);c)
\le M_{2,\hat{\delta}^+_{j+1}}c^{2/3}, \\
-M_{1,\hat{\delta}^+_{j+1}}c^{1/3}\le\hat{\Psi}'(\hat{\delta}^+_{j+1}(c);c)
\le -m_{1,\hat{\delta}^+_{j+1}}c^{1/3},
\end{array}\right.
\]
and there exists $M_{2,\hat{\delta}^-_j},m_{2,\hat{\delta}^-_j}
M_{1,\hat{\delta}^-_j},m_{1,\hat{\delta}^-_j}>0$ such that for $c\in J_{\hat{\mu}^+_j}$
\[
\left\{\begin{array}{l}
-M_{2,\hat{\delta}^-_j}c^{2/3}\le\hat{\Psi}''(\hat{\delta}^-_j(c);c)
\le -m_{2,\hat{\delta}^-_j}c^{2/3}, \\
m_{1,\hat{\delta}^-_j}c^{1/3}\le\hat{\Psi}'(\hat{\delta}^-_j(c);c)
\le M_{1,\hat{\delta}^-_j}c^{1/3}.
\end{array}\right. 
\]
\item[(II-ii)] 
There exist $M_{0,\hat{\mu}^\pm_j},M_{2,\hat{\mu}^\pm_j},m_{0,\hat{\mu}^\pm_j},
m_{2,\hat{\mu}^\pm_j}>0$ such that 
\[
\left\{\begin{array}{l}
-M_{0,\hat{\mu}^-_j}\le\hat{\Psi}(\hat{\mu}^-_j(c);c)\le -m_{0,\hat{\mu}^-_j}
\,\,\ \mbox{for}\,\ c\in J_{\hat{\mu}^-_j}, \\
m_{2,\hat{\mu}^-_j}c^{2/3}\le\hat{\Psi}''(\hat{\mu}^-_j(c);c)\le M_{2,\hat{\mu}^-_j}c^{2/3}
\,\,\ \mbox{for}\,\ c\in J_{\hat{\gamma}^+_j}, 
\end{array}\right. 
\]
and 
\[
\left\{\begin{array}{l}
m_{0,\hat{\mu}^+_j}\le\hat{\Psi}(\hat{\mu}^+_j(c);c)\le M_{0,\hat{\mu}^+_j}
\,\,\ \mbox{for}\,\ c\in J_{\hat{\mu}^+_j}, \\
-M_{2,\hat{\mu}^+_j}c^{2/3}\le\hat{\Psi}''(\hat{\mu}^+_j(c);c)\le -m_{2,\hat{\mu}^+_j}c^{2/3} 
\,\,\ \mbox{for}\,\ c\in J_{\hat{\gamma}^-_j}. 
\end{array}\right.
\]
\item[(II-iii)] 
There exists $M_{0,\hat{\gamma}^\pm_j}, M_{1,\hat{\gamma}^\pm_j}, m_{0,\hat{\gamma}^\pm_j}, 
m_{1,\hat{\gamma}^\pm_j}>0$ such that 
\[
\left\{\begin{array}{l}
m_{1,\hat{\gamma}^+_j}c^{1/3}\le\hat{\Psi}'(\hat{\gamma}^+_j(c);c)\le M_{1,\hat{\gamma}^+_j}c^{1/3} 
\,\,\ \mbox{for}\,\ c \in J_{\hat{\gamma}_j^+}, \\
-M_{0,\hat{\gamma}^+_j}\le\hat{\Psi}(\hat{\gamma}^+_j(c);c)\le -m_{0,\hat{\gamma}^+_j} 
\,\,\ \mbox{for}\,\ c\in J_{\hat{\mu}^+_j},
\end{array}\right.
\]
and 
\[
\left\{\begin{array}{l}
-M_{1,\hat{\gamma}^-_j}c^{1/3}\le\hat{\Psi}'(\hat{\gamma}^-_j(c);c)\le -m_{1,\hat{\gamma}^+_j}c^{1/3} 
\,\,\ \mbox{for}\,\ c\in J_{\hat{\gamma}^-_j}, \\
m_{0,\hat{\gamma}^-_j}\le\hat{\Psi}(\hat{\gamma}^-_j(c);c)\le M_{0,\hat{\gamma}^-_j} 
\,\,\ \mbox{for}\,\ c\in J_{\hat{\mu}^-_{j+1}}. 
\end{array}\right.
\]
\end{list}
\end{prop} 

\section{Proof of the main results}\label{sec:main}
Let us prove the main results, namely, Theorem \ref{thm:main}(ii) and (iv). 
Note that Theorem \ref{thm:main}(i) and (iii) have been already proved in Section \ref{sec:profile}. 
The rest of the proofs of the main results is to find $c>0$ such that $\hat{\Psi}(1;c) = -\psi_+$ 
with the help of the properties of  $\hat{\Psi}$ obtained in 
Section \ref{sec:initial}\,--\,\ref{sec:order-est}. 
%
%
If we find such $c>0$, the pair $(\hat{\Psi}(\,\cdot\,;c),c)$ 
is the solution of \eqref{SD_TW_3rdODE} and we can construct a traveling wave 
$\mathcal{W}(t)$ as in \eqref{para-tra} from its pair. Note that 
for the representaition $\mathcal{W}(0)=\{(x,w(x))\,|\,l^-<x<l^+\}$, 
the sign of $\hat{\Psi}, \hat{\Psi}'$ and $\hat{\Psi}''$ are same as that of $w_x, w_{xx}$ 
and $w$, respectively. 

Therefore, the rest of this paper is devoted to proving the existence of $c>0$ such that 
$\hat{\Psi}(1;c) = -\psi_+$.  
In this proof, the estimates of $\hat{\Psi}(\hat{\mu}_j^-(c)\re{;}c)$ for $j \in \mathbb{N}$ in 
Proposition \ref{prop:est_order}(II-ii) play a key role. 

Let the constants $m_{0,\hat{\mu}_j^-} > 0$ in Proposition \ref{prop:est_order}(II-ii) satisfy 
\[ 
\psi_- \ge m_{0,\hat{\mu}_1^-} > m_{0,\hat{\mu}_2^-} > m_{0,\hat{\mu}_3^-} > \cdots,  
\]
if necessary. Theorem \ref{thm:main}(ii) and (iv) can be derived from the following 
theorem.

\begin{thm}\label{thm:main2} 
Assume that $\psi_- \in (0,\pi/2)$. 
Let $\hat{\Psi}(\,\cdot\,;c) = \Psi(\,\cdot\,;\hat{\alpha}(c),c)$ be a solution of \eqref{IVP2}, where $\Psi(\,\cdot\,;\alpha,c)$ is the solution of \eqref{IVP} 
and $\hat{\alpha}(c)$ is the constant obtained in Proposition \ref{thm:exit_a}.
For any $j \in \mathbb{N}$, define $J_{\hat{\delta}_j^\pm}$ and $J_{\hat{\mu}_j^-}$ 
as \eqref{def-dipm} and \eqref{def-mipm}, respectively, and let $m_{0,\hat{\mu}_j^-} > 0$ 
be the constants in Proposition \ref{prop:est_order}. 
\begin{list}{}{\leftmargin=0cm\itemindent=0.2cm\topsep=0cm\itemsep=0cm} 
\item[(i)]
For each $\psi_+\in(0,\psi_-)$, there exists $c >0$ such that $\hat{\Psi}(1;c) = -\psi_+$ 
and $c \not\in J_{\hat{\delta}_1^-}$. 
\item[(ii)]
If $\psi_+\in[m_{0,\hat{\mu}_{j+1}^-},m_{0,\hat{\mu}_j^-})$, then there exist constants 
\[ 0 < c_1 < c_2 < \cdots < c_{2j-1} \]
such that the following statements hold: 
\begin{list}{}{\leftmargin=1cm\itemindent=0.2cm\topsep=0cm\itemsep=0cm}
\item[(a)] 
$\hat{\Psi}(1;c_k) = -\psi_+$ for any $1 \le k \le 2j-1$. 
\item[(b)] 
$c_{2l-1} \in J_{\hat{\mu}_l^-} \setminus J_{\hat{\delta}_l^-}$ for $1 \le l \le j$. 
\item[(c)] 
$c_{2l-2} \in J_{\hat{\delta}_l^+} \setminus J_{\hat{\mu}_l^-}$ for $2 \le l \le j$ if $j \ge 2$. 
\end{list}
\end{list} 
\end{thm} 

\begin{rem}
For $c>0$ as in Theorem \ref{thm:main2}(i), it follows from 
$c \not\in J_{\hat{\delta}_1^-}$, Proposition \ref{prop:com_zero-point}(i) and 
the boundary condition $\hat{\Psi}''(1;c)=0$ that 
the zero points of $\hat{\Psi}(\,\cdot\,;c)$ and its derivatives satisfy 
\begin{equation}\label{rem_main_1}
 0 < \hat{\delta}_1^+(c) < \hat{\gamma}_1^+(c) = 1 \quad \mbox{or} \quad 
 0 < \hat{\delta}_1^+(c) < \hat{\mu}_1^-(c) < \hat{\gamma}_1^+(c) = 1. 
\end{equation}
In the former case of \eqref{rem_main_1}, the profile curve $\mathcal{W}(0)$ 
of the traveling wave $\mathcal{W}(t)$ as in \eqref{para-tra} is convex. 
On the other hand, in the latter case, $\mathcal{W}(0)$ is no longer convex since the sign change 
of the curvature, which is given by $\hat{\Psi}'(\,\cdot\,;c)$, occurs. We also remark that 
in both cases, $\mathcal{W}(0)$ is a plane curve in the upper half plane 
since $\hat{\Psi}''(s;c)>0$ for $s\in(0,1)$. 
 


Theorem \ref{thm:main2}(ii) implies that $\hat{\Psi}$ and its derivatives ``oscillates'' 
for $c>0$ as in (b) or (c) if $\psi_+>0$ is sufficiently small. 
From this ``oscillation'' property, we obtain Theorem \ref{thm:main} by setting 
$m_j := m_{0,\hat{\mu}_j^-}$ for any $j \in \mathbb{N}$, where $m_j$ and $m_{0,\hat{\mu}_j^-}$ 
are the constants in Theorem \ref{thm:main}(iv) and Theorem \ref{thm:main2}(ii), 
respectively. 
\end{rem}

Let us discuss the shooting with respect to $c$.  
In order to prove Theorem \ref{thm:main2}(i), we apply the intermediate value theorem 
to $\hat{\Psi}(1;c)$ with respect to $c \in (0,\infty) \setminus J_{\hat{\delta}_1^-}$. 
Also, in order to prove Theorem \ref{thm:main2}(ii), we apply its theorem to 
$\hat{\Psi}(1;c)$ with respect to $c \in J_{\hat{\mu}_j^-} \setminus J_{\hat{\delta}_j^-}$ and 
$c \in J_{\hat{\delta}_j^+} \setminus J_{\hat{\mu}_j^-}$. 
Thus, we have to study the values $\hat{\Psi}(1;c)$ for $c \in \partial J_{\hat{\mu}_j^-}$ 
and $c \in \partial J_{\hat{\delta}_j^\pm}$. 
The following lemma is a key tool to obtain their values. 

\begin{lemma}\label{lem:limit-deltamu} 
(i) For any $j\in\mathbb{N}$, $J_{\hat{\delta}^+_j}$ is an open set and $\hat{\delta}^+_j(c)$ 
is continuous in $J_{\hat{\delta}^+_j}$.  
Moreover, $\hat{\delta}^+_j(c)=1$ for $c\in\partial J_{\hat{\delta}^+_j}$ if $j \ge 2$. \\
(ii) For any $j\in\mathbb{N}$, $J_{\hat{\delta}^-_j}$ is an open sets and $\hat{\delta}^-_j(c)$ 
is continuous in $J_{\hat{\delta}^-_j}$. 
Moreover, $\hat{\delta}^-_j(c)=1$ for $c\in\partial J_{\hat{\delta}^-_j}$. \\
(iii) For any $j\in\mathbb{N}$, $J_{\hat{\mu}^\pm_j}$ are open sets and $\hat{\mu}^\pm_j(c)$ 
are continuous in $J_{\hat{\mu}^\pm_j}$, respectively.
Moreover, $\hat{\mu}^\pm_j(c)=1$ for $c\in\partial J_{\hat{\mu}^\pm_j}$. 
\end{lemma} 

\begin{proof} 
Let us prove only (i) with $j \ge 2$. For the other cases, the same argument is applicable. 

First, we show that $J_{\hat{\delta}^+_j}$ is an open set and $\hat{\delta}^+_j(c)$ is continuous in $J_{\hat{\delta}^+_j}$. 
The proof can be given by the induction from left end points in the sense of the order 
\[ 
\hat{\delta}_1^+(c) < \hat{\mu}_1^-(c) < \hat{\gamma}_1^+(c) < \cdots < \hat{\delta}_{j-1}^-(c) 
< \hat{\mu}_{j-1}^+(c) < \hat{\gamma}_{j-1}^-(c) < \hat{\delta}_j^+(c) < \cdots. 
\]
Thus we assume that $J_{\hat{\delta}_k^\pm}, J_{\hat{\mu}_k^\pm}$ and 
$J_{\hat{\gamma}_k^\pm}$ are open sets and $\hat{\delta}_k^\pm(c), \hat{\mu}_k^\pm(c)$ 
and $\hat{\gamma}_k^\pm(c)$ are continuous with respect $c$ for $1 \le k \le j-1$. 
Let us fix $c_0\in J_{\hat{\delta}^+_j}$. We prove that for any $\e>0$ there is $\tau>0$ 
such that if $|c-c_0|<\tau$, $\hat{\delta}^+_j(c)$ exists in $(0,1)$ and satisfies 
$|\hat{\delta}^+_j(c)-\hat{\delta}^+_j(c_0)|<\e$. 
Since $\hat{\Psi}'(\hat{\delta}^+_j(c_0);c_0)<0$ and $\hat{\gamma}_{j-1}^-(c_0) < \hat{\delta}_j^+(c_0)$ by Proposition \ref{prop:com_zero-point}, 
it follows from the continuity of $\hat{\Psi}$ with respect to $s$ that there is $\e_0 > 0$ such that 
for any $\e \in (0,\e_0)$
\begin{equation}\label{Psi1_neg_e0-nbd}
\begin{aligned}
&\hat{\Psi}'(s;c_0)<\frac{\hat{\Psi}'(\hat{\delta}^+_j(c_0);c_0)}2<0\,\,\ 
\mbox{for}\,\ s\in[\hat{\delta}^+_j(c_0)-\e,\hat{\delta}^+_j(c_0)+\e], \\
&\hat{\gamma}_{j-1}^-(c_0) < \hat{\delta}_j^+(c_0) - \e.
\end{aligned}
\end{equation}
This fact and Proposition \ref{prop:com_zero-point} imply that 
\begin{equation}\label{sign_Psi_e-nbd}
\hat{\Psi}(s;c_0)>0\,\,\ \mbox{for}\,\ s\in[\hat{\mu}_{j-1}^+(c_0),\hat{\delta}^+_j(c_0)-\e], 
\quad \hat{\Psi}(\hat{\delta}^+_j(c_0)+\e;c_0)<0.  
\end{equation}
From \eqref{Psi1_neg_e0-nbd}, \eqref{sign_Psi_e-nbd}, Lemma \ref{lem:conti_ineq} 
and the continuity of $\hat{\delta}_k^\pm(c), \hat{\mu}_k^\pm(c), \hat{\gamma}_k^\pm(c)$ 
with respect to $c$, we see that there is $\tau>0$ such that if $|c-c_0|<\tau$, 
\begin{equation}\label{Psi1_neg_e-nbd}
\begin{aligned}
&\hat{\Psi}'(s;c)<0\,\,\ 
\mbox{for}\,\ s\in[\hat{\delta}^+_j(c_0)-\e,\hat{\delta}^+_j(c_0)+\e] \\
&\hat{\Psi}(s;c)>0,\,\,\ \mbox{for}\,\ s\in[\hat{\mu}_{j-1}^+(c),\hat{\delta}^+_j(c_0)-\e], \\
&\hat{\Psi}(\hat{\delta}^+_j(c_0)+\e;c)<0,  
\end{aligned}
\end{equation}
and $\hat{\delta}_k^\pm(c), \hat{\mu}_k^\pm(c)$ and $\hat{\gamma}_k^\pm(c)$ exist 
in the interval $(0,1)$ and satisfy
\begin{equation}\label{order-nbd} 
\hat{\delta}_1^+(c) < \hat{\mu}_1^-(c) < \hat{\gamma}_1^+(c) < \cdots < \hat{\delta}_{j-1}^-(c) 
< \hat{\mu}_{j-1}^+(c) < \hat{\gamma}_{j-1}^-(c) < \hat{\delta}^+_j(c_0)-\e. 
\end{equation}
Due to \eqref{Psi1_neg_e-nbd}, there exists the unique zero point of $\hat{\Psi}(\cdot,c)$ in 
$(\hat{\delta}^+_j(c_0)-\e,\hat{\delta}^+_j(c_0)+\e)$ for any $c \in (c_0-\tau, c_0+\tau)$. 
Define its zero point as $\hat{s}^+(c)$. 
On the other hand, by the second property of \eqref{Psi1_neg_e-nbd}, the order \eqref{order-nbd} 
and Proposition \ref{prop:com_zero-point}, we see that there is no zero point of $\hat{\Psi}(\cdot,c)$ 
in $(\hat{\delta}_{j-1}^+(c), \hat{s}^+(c))$ for $c \in (c_0-\tau, c_0+\tau)$.
As a result, we obtain the fact that $\hat{s}^+(c) = \hat{\delta}_j^+(c) \in (0,1)$ and 
$\hat{\delta}_j^+(c)$ satisfies $|\hat{\delta}_j^+(c) - \hat{\delta}_j^+(c_0)| < \e$ 
for $c \in (c_0-\tau, c_0+\tau)$ by the definition of $\hat{s}^+(c)$. 

Let us show that $\hat{\delta}^+_j(c)=1$ for $c\in\partial J_{\hat{\delta}^+_j}$. 
To do it, we prove that for any sequences $\{c_n\}_{n\in\mathbb{N}}\subset J_{\hat{\delta}^+_j}$ 
which converges to $c\in\partial J_{\hat{\delta}^+_j}$, the limit of the sequence 
$\{\hat{\delta}^+_j(c_n)\}_{n\in\mathbb{N}}$ exists and is equal to $1$. 
Since $\hat{\delta}^+_j(c_n)\in(0,1)$ for $n\in\mathbb{N}$ by the definition of 
$J_{\hat{\delta}^+_j}$, $\{\hat{\delta}^+_j(c_n)\}_{n\in\mathbb{N}}$ includes a convergent 
subsequences $\{\hat{\delta}^+_j(c_{n_k})\}_{k\in\mathbb{N}}$. Let $\hat{\delta}^\ast$ be 
a limit of its subsequence, that is, 
\[
\lim_{k\to\infty}\hat{\delta}^+_j(c_{n_k})=\hat{\delta}^\ast.
\]
Note that $\hat{\delta}^\ast\in[0,1]$. By means of the continuity of $\hat{\Psi}$ with respect to $(s,c)$, 
we obtain
\[
0=\lim_{k\to\infty}\hat{\Psi}(\hat{\delta}^+_j(c_{n_k});c_{n_k})=\hat{\Psi}(\hat{\delta}^\ast;c).
\]
Then it follows from the definition of $\hat{\delta}^+_j(c)$ and the continuity of $\hat{\delta}_k^\pm(c), \hat{\mu}_k^\pm(c), \hat{\gamma}_k^\pm(c)$ with respect to $c$ for $1 \le k \le j-1$ that $\hat{\delta}^+_j(c)\le\hat{\delta}^\ast$. 
Since $J_{\hat{\delta}^+_j}$ is open, we have $c\not\in J_{\hat{\delta}^+_j}$, so that  
we are led to $\hat{\delta}^+_j(c)=\hat{\delta}^\ast=1$.  
This means that a limit of any convergent subsequences is unique and its value is equal to $1$, 
which completes the proof. 
\end{proof} 

Now, we prove Theorem \ref{thm:main2}(i). 

\begin{proof}[Proof of Theorem \ref{thm:main2}(i)]
Let $\psi_+ \in (0,\psi_-)$ and define $c_{\rm min}$ as 
\[ c_{\rm min} := \inf J_{\hat{\delta}_1^-}. \]
$J_{\hat{\delta}_1^-}$ is not empty by Lemma \ref{lem:exist_d1-}, 
thus $c_{\rm min}$ is well-defined. 
Moreover, we see that $c_{\rm min} \ge c_\ast > 0$ from \eqref{def_c_ast} and 
Proposition \ref{prop:com_zero-point}. 
Since $c_{\rm min} \in \partial J_{\hat{\delta}_1^-}$, it follows from  
Lemma \ref{lem:limit-deltamu} that
\begin{equation}\label{main2-1-1} 
\hat{\Psi}(1;c_{\rm min}) = \hat{\Psi}(\hat{\delta}_1^-(c_{\rm min}); c_{\rm min}) = 0 > -\psi_+. 
\end{equation}
On the other hand, from Proposition \ref{thm:exit_a}(iv) and $\psi_+ \in (0,\psi_-)$, 
\begin{equation}\label{main2-1-2} 
\lim_{c \downarrow 0} \hat{\Psi}(1;c) = - \psi_- < -\psi_+. 
\end{equation}
Therefore, \eqref{main2-1-1}, \eqref{main2-1-2} and the intermediate value theorem lead us to 
the existence of $c \in (0,\infty) \setminus J_{\hat{\delta}_1^-}$ such that $\hat{\Psi}(1;c) = -\psi_+$. 
\end{proof}

Finally, let us prove Theorem \ref{thm:main2}(ii). 

\begin{proof}[Proof of Theorem \ref{thm:main2}(ii)]
Let $\psi_+\in[m_{0,\hat{\mu}^-_{j+1}},m_{0,\hat{\mu}^-_j})$ for some $j\in\mathbb{N}$. 
First, we prove the existence of $\tilde{c}_l \in J_{\hat{\mu}_l^-} \setminus J_{\hat{\delta}_l^-}$ 
such that $\hat{\Psi}(1;\tilde{c}_l) = 1$ for $1\le l\le j$. 
For $1\le l\le j$, set 
\[
c_{\hat{\delta}_l^-}:= \inf J_{\hat{\delta}^-_l}.
\]
By means of Lemma \ref{lem:limit-deltamu}, we obtain $\hat{\delta}^-_l(c_{\hat{\delta}_l^-})=1$. 
Then
\begin{equation}\label{main2-1} 
\hat{\Psi}(1;c_{\hat{\delta}_l^-})=0>-\psi_+ 
\end{equation} 
and it follows from Proposition \ref{prop:com_zero-point}(v) that $\hat{\mu}^-_l(c_{\hat{\delta}_l^-})$ 
exists in $(0,1)$. 
Therefore, $c_{\hat{\delta}_l^-}$ is a interior point of $J_{\hat{\mu}^-_l}$ by 
Lemma \ref{lem:limit-deltamu}(iii) and we can define 
\[
\tilde{c}_{\hat{\mu}_l^-} 
:=\inf\{\tilde{c}\in(0,c_{\hat{\delta}_l^-})\,|\,\mbox{$\hat{\mu}_l^-(c)$ exists in $(0,1)$ 
for $c\in(\tilde{c}, c_{\hat{\delta}_l^-})$}\}.
\]
By the definition of $\tilde{c}_{\hat{\mu}_l^-}$, we see that 
$\tilde{c}_{\hat{\mu}_l^-}\in\partial J_{\hat{\mu}_l^-}$. 
Then Lemma \ref{lem:limit-deltamu}(iii) implies $\hat{\mu}_l^-(\tilde{c}_{\hat{\mu}_l^-})=1$. 
Using the continuity of $\hat{\Psi}$ with respect to $(s,c)$ and 
Proposition \ref{prop:est_order} (II-ii), 
\begin{equation}\label{main2-2} 
\hat{\Psi}(1;\tilde{c}_{\hat{\mu}_l^-})
=\lim_{c\downarrow\tilde{c}_{\hat{\mu}_l^-}}\hat{\Psi}(\hat{\mu}_l^-(c);c) 
\le -m_{0,\hat{\mu}^-_l} \le -m_{0,\hat{\mu}^-_j} < -\psi_+. 
\end{equation} 
Consequently, it follows from \eqref{main2-1}, \eqref{main2-2}, and the continuity of $\hat{\Psi}$ 
with respect to $c$ that there exists $\tilde{c}_l\in [\tilde{c}_{\hat{\mu}_l^-},c_{\hat{\delta}_l^-})$ 
such that 
\[
\hat{\Psi}(1;\tilde{c}_l)=-\psi_+.
\]
Recalling the definition of $c_{\hat{\delta}_l^-}$ and $\tilde{c}_{\hat{\mu}_l^-}$, we see 
that $\tilde{c}_l \in J_{\hat{\mu}_l^-} \setminus J_{\hat{\delta}_l^-}$. 

Next, we prove the existence of $\hat{c}_l \in J_{\hat{\delta}_l^+} \setminus J_{\hat{\mu}_l^-}$ 
such that $\hat{\Psi}(1;\hat{c}_l) = 1$ for $2\le l\le j$.
For $2\le l\le j$, set 
\[
c_{\hat{\mu}_l^-}:= \inf J_{\hat{\mu}^-_l}.
\]
Note that we do not know whether $J_{\hat{\mu}^-_l}$ is connected or not. Thus 
it is possible that $c_{\hat{\mu}_l^-} \neq \tilde{c}_{\hat{\mu}_l^-}$. 
Then we see $c_{\hat{\mu}_l^-} \le \tilde{c}_{\hat{\mu}_l^-}$ by the definition of them. 
From a similar argument as above, Proposition \ref{prop:com_zero-point}(v), 
Lemma \ref{lem:limit-deltamu} and Proposition \ref{prop:est_order}(II-ii) imply that
\begin{align}
&\hat{\mu}_l^-(c_{\hat{\mu}_l^-}) = 1, \label{main2-3}\\
&\hat{\Psi}(1;c_{\hat{\mu}_l^-})=\lim_{c\downarrow c_{\hat{\mu}_l^-}}\hat{\Psi}(\hat{\mu}_l^-(c);c) 
\le -m_{0,\hat{\mu}^-_l} \le -m_{0,\hat{\mu}^-_j} < -\psi_+. \label{main2-4}
\end{align}
By \eqref{main2-3} and Proposition \ref{prop:com_zero-point}, we see that 
$\hat{\mu}_{l-1}^+(c_{\hat{\mu}_l^-})$ exists in the interval $(0,1)$ and $\hat{\Psi}(\hat{\mu}_{l-1}^+(c_{\hat{\mu}_l^-});c_{\hat{\mu}_l^-}) >0$. 
Therefore, it follows from \eqref{main2-4}, the intermediate value theorem and 
Proposition \ref{prop:com_zero-point} that $\hat{\delta}_l^+(c_{\hat{\mu}_l^-})$ exists in the interval 
$(0,1)$. Hence $c_{\hat{\mu}_l^-}$ is a interior point of $J_{\hat{\delta}^-_+}$ by 
Lemma \ref{lem:limit-deltamu}(i) and we can define 
\[ 
\tilde{c}_{\hat{\delta}_l^+} 
:=\inf\{\tilde{c}\in(0,c_{\hat{\mu}_l^-})\,|\,\mbox{$\hat{\delta}_l^+(c)$ exists in $(0,1)$ for $c\in(\tilde{c}, c_{\hat{\mu}_l^-})$}\}.
\]
By $\tilde{c}_{\hat{\delta}_l^+} \in \partial J_{\hat{\delta}_l^+}$ and $l \ge 2$ and 
Lemma \ref{lem:limit-deltamu}(i), we obtain
\begin{equation}\label{main2-5}
\hat{\Psi}(1;\tilde{c}_{\hat{\delta}_l^+}) 
= \hat{\Psi}(\hat{\delta}_l^+(\tilde{c}_{\hat{\delta}_l^+});\tilde{c}_{\hat{\delta}_l^+}) = 0 > -\psi_+. 
\end{equation} 
Thus it follows from \eqref{main2-4}, \eqref{main2-5}, the intermediate value theorem and 
the definition of $\tilde{c}_{\hat{\delta}_l^+}$ and $c_{\hat{\mu}_l^-}$ that 
for $2\le l\le j$, there exists $\hat{c}_l \in J_{\hat{\delta}_l^+} \setminus J_{\hat{\mu}_l^-}$ 
such that $\hat{\Psi}(1;\hat{c}_l) = -\psi_+$. 

For the order of $\tilde{c}_l$ and $\hat{c}_l$, we see 
\[ 
\tilde{c}_{\hat{\mu}_1^-} < \tilde{c}_1 < c_{\hat{\delta}_1^-} < \tilde{c}_{\hat{\delta}_2^+} < \hat{c}_2 
< c_{\hat{\mu}_2^-} \le \tilde{c}_{\hat{\mu}_2^-} < \cdots < \tilde{c}_{\hat{\delta}_l^+} < \hat{c}_l 
< c_{\hat{\mu}_l^-} \le \tilde{c}_{\hat{\mu}_l^-} < \tilde{c}_l < c_{\hat{\delta}_l^-}. 
\]
Thus, if we set $c_k$ for $k=1,2, \cdots, 2j-1$ as 
\[
c_k := 
\begin{cases}
\tilde{c}_l & \mbox{if} \; \; k = 2l-1 \; \; \mbox{for} \; \; l=1,2, \cdots, j, \\
\hat{c}_l & \mbox{if} \; \; k=2l-2 \; \; \mbox{for} \; \; l=2,3, \cdots, j, 
\end{cases}
\]
we obtain the conclusion of Theorem \ref{thm:main2}(ii).  
\end{proof} 

\begin{rem}\label{rem:last}
In the case $\psi_+ < 0$, by applying a similar argument to the above proof, we can obtain 
a similar result as in Theorem \ref{thm:main2}(ii). 
In this argument, we have to apply the intermediate value theorem to $\hat{\Psi}(1;c)$ 
with respect to $c \in J_{\hat{\mu}_l^+} \setminus J_{\hat{\delta}_{l+1}^+}$ and 
$c \in J_{\hat{\delta}_l^-} \setminus J_{\hat{\mu}_l^+}$. 
As a result, we obtain the following conclusion:

For any $j \in \mathbb{N}$, if $\psi_+ \in (-m_{0,\hat{\mu}_j^+}, -m_{0,\hat{\mu}_{j+1}^-}]$, 
where $m_{0,\hat{\mu}_j^+}$ are the constants in Proposition \ref{prop:est_order}, 
then there exist constants 
\[ 0 < c_1 < c_2 < \cdots < c_{2j} \]
such that the following statements hold: 
\begin{list}{}{\leftmargin=1cm\itemindent=0.2cm\topsep=0.15cm\itemsep=0cm}
\item[(a)]
$\hat{\Psi}(1;c_k) = -\psi_+$ for any $1 \le k \le 2j$. 
\item[(b)] 
$c_{2l-1} \in J_{\hat{\delta}_l^-} \setminus J_{\hat{\mu}_l^+}$ for $1 \le l \le j$. 
\item[(c)] 
$c_{2l} \in J_{\hat{\mu}_l^+} \setminus J_{\hat{\delta}_{l+1}^-}$ for $1 \le l \le j$. 
\end{list} 

On the other hand, we can not prove the existence of the traveling wave for \eqref{SDB} 
if $\psi_+ < -m_{0,\hat{\mu}_1^+}$. 
In order to prove the existence of it in the case that $\psi_+$ is close to $-\psi_-$, 
we have to find $c\ge0$ such that $\hat{\Psi}(1;c) > -\psi_+$. 
However, this $c$ should not exist since $\hat{\Psi}(1;0) = -\psi_- < -\psi_+$ and 
the effect of ``energy loss'' as in Lemma \ref{lem:com_Psi1-2}(i) becomes stronger as 
$c$ increase. 
Therefore, if $\psi_+ < 0$, we expect the necessity of the smallness of $|\psi_+|$ 
to obtain the existence of the traveling wave for \eqref{SDB}. 
\end{rem}

 

\end{document}